\documentclass[a4paper,twoside,10pt]{article} 

\newif\ifpdf
\ifx\pdfoutput\undefined
	\pdffalse              
\else
	\pdfoutput=1           
	\pdftrue               
\fi

\ifpdf
\else
\fi

\usepackage[T1]{fontenc}
\usepackage[latin1]{inputenc}

\ifpdf 
	\usepackage[pdftex]{graphicx} 
\else
	\usepackage[dvips]{graphicx} 
\fi

\usepackage{amsmath}
\usepackage{amsthm}
\usepackage{amsfonts}
\usepackage{amscd}

\begin{document}

\ifpdf
	\DeclareGraphicsExtensions{.pdf,.jpg,.png}
\else
	\DeclareGraphicsExtensions{.eps}
\fi

\pagestyle{headings}


\title{Uniform $\lambda-$Adjustment and $\mu-$Approximation in Banach Spaces}
\author{Boris Burshteyn}
\date{April 14, 2008} 
\maketitle
\[
boris997@astound.net
\]
\[
Walnut\ Creek\ CA,\ USA
\]
\[
\ \ \ 
\]
\[
\ \ \ 
\]
\[
\\
\emph{To\ the\ memory\ of\ my\ teacher,\ Mikhail\ Abramovi\^{c}\ Gol'dman}\\
\\
\]
\[
\ \ \ 
\]
\[
\ \ \ 
\]

\begin{abstract}
We introduce a new concept of perturbation of closed linear subspaces and operators 
in Banach spaces called \emph{uniform $\lambda-$adjustment} which is weaker than perturbations by small gap, operator norm, $q-$norm, and $K_{2}-$approximation. In arbitrary Banach spaces some of the classical Fredholm stability theorems remain true under uniform $\lambda-$adjustment, while other fail. However, uniformly $\lambda-$adjusted subspaces and linear operators retain their (semi--)Fredholm properties in a Banach space which dual is Fr\'{e}chet-Urysohn in ${weak^{*}}$ topology.

We also introduce another concept of perturbation called \emph{uniform\linebreak[4] $\mu-$approximation} which is weaker than perturbations by small gap, norm, and compact convergence, yet stronger than uniform $\lambda-$adjustment. We present Fredholm stability theorems for uniform $\mu-$approximation in arbitrary Banach spaces and a theorem on stability of Riesz kernels and ranges for commuting closed essentially Kato operators.
 
Finally, we define the new concepts of a \emph{tuple of subspaces} and of a \emph{complex of subspaces} in Banach spaces, and present stability theorems for index and defect numbers of Fredholm tuples and complexes under uniform $\lambda-$adjustment and uniform  $\mu-$approximation. 
%
%
\end{abstract}


\newpage

\tableofcontents 

\newpage


\pagestyle{headings} 


\section[Introduction]{Introduction}\label{S:intro}
The goal of this paper is to research the foundation of the Fredholm stability theory in Banach spaces. We consider perturbations of linear subspaces of a general kind based on metric proximity of the sequence of differences between elements of two subspaces to a fixed vector from the ambient Banach space. This approach subsumes several already known classes of perturbations and allows for cases that do not fit into any of the existing class. We prove and present Fredholm style stability theorems -- some of them new, others strengthening classical results. 

The following subsection recaps several current Fredholm stability approaches in their relation to each other.

\subsection[Overview of the Stability Theory]{Overview of the Stability Theory}\label{S:overview}
The concept of a gap between two linear subspaces of a Banach space has been introduced 60 years ago in \cite{krein_krasnoselskii_milman}. Starting from that foundational paper the gap as a measure of proximity between two subspaces or between two operators (encompassing operators with small norm difference) has been used in the research of stability of properties of linear equations with Fredholm and semi--Fredholm operators -- see \cite{krein_krasnoselskii_milman,gohberg_markus,gohberg_krein,gohberg_markus_feldman,neubauer1,neubauer2,kato}.

Soon after \cite{krein_krasnoselskii_milman} it was established in \cite{atkinson} that the property of being a Fredholm operator is stable under perturbation by a compact operator. In the subsequent papers \cite{goldenstein_gohberg_markus,goldenstein_markus,sedaev,nussbaum,sadovskii} a more general theory of stability of Fredholm and semi--Fredholm operators has been developed that considers perturbations by operators with a small measure of non-compactness.

In parallel with the above two stability theories yet another approach has been developed in \cite{sobolev,anselone_palmer1,anselone_palmer2,chandler-wide_zhang,Rakovschik,vladimirski,vainikko1,vainikko2} as a theory of collectively compact operators or, equivalently, as a theory of compact approximation. Results in this direction are closely related to computational methods. A relation between compact operators, compact approximation and measures of non-compactness have been established. For example, in \cite{vladimirski} it has been proved that a sequence of linear operators ${(A_{n})_{\mathbb{N}}}$ mapping a Banach space $X$ onto a Banach Space $Y$ ${K_{2}-}$converges to a null operator (i.e. the ${K_{2}-}$condition holds: if ${(x_{n})_{\mathbb{N}} \subset X}$ is bounded, then ${(A_{n}x_{n})_{\mathbb{N}} \subset Y}$ is relatively compact) if and only if ${\lim_{n \rightarrow \infty} q(\cup_{i \geq n} A_{i}(B_{1})) = 0}$ where ${B_{1}}$ is a unit ball in X, ${q(M)}$ is the Hausdorff measure of non-compactness of a set ${M \subset Y}$. 

Another area of stability research that generalizes the theory of compact perturbations involves perturbations of linear operators by inessential, $\Omega_{+}$ and $\Omega_{-}$ operators (\cite{aiena0,aiena}), strictly singular (\cite{kato0}) and strictly co--singular (\cite{pelczynski,pietsch}) operators -- the work \cite{aiena} presents a thorough overview of the result in this direction. In addition, recently another direction has emerged in stability theory for Fredholm and semi-Fredholm pairs of subspaces of a Banach space in work \cite{gonzales} -- it employs the concepts of strictly singular and strictly cosingular operators in the context of pairs of subspaces.

All of the above approaches proved powerful stability results in Banach spaces, and not only for (semi--)Fredholm operators -- in the cited works it had been proved that they allow for stability of spectral properties. Also, Fredholm properties of operator complexes have been discovered to be stable under small gap, compact or inessential perturbations in the works \cite{fainshtein,fainshtein_shulman,vasilescu,ambrozie,gleason}.\\

What can be said about the relations between these stability theories? It appears that compact paerturbations and compact approximation are related via Hausdorff measure of noncompactness. At the same time compact perturbations are extended by inessential, $\Omega_{+}$ and $\Omega_{-}$ operators, strictly singular and strictly co--singular operators. 

However, we do not know of any so far established relations between small gap and compact approximation, or between small gap and perturbations by compact operators. Another notable difference between these approaches is that only the gap theory and the newly developed concepts of strictly singular and strictly co-singular pairs deal with the subspaces of a Banach space -- the rest is concerned only with linear operators. 

Note that a general theory of multi--valued linear operators with meaningful applications has been extensively developed (see \cite{cross}) and some stability theorems have been obtained there as well. Thus, it would be of certain interest to consider if similar results can be obtained for subspaces of Banach spaces, not just for single-- or multi-- valued operators.\\

The disparity between the gap concept and the rest of the above mentioned approaches brings up a question if there is a small set of coherent underlying principles on which a unified stability theory can be founded. Ideally, its theorems would strengthen and generalize the facts for subspaces and operators from all of the existing theories of gap, compact perturbations, compact approximation,  inessential operators and singular/co-singular perturbations. 
\subsection[New Concepts: $\lambda-$Adjustment and $\mu-$Approximation]{New Concepts: $\lambda-$Adjustment and $\mu-$Approximation}\label{nsclaama}
In the course of this paper we attempt to build a theory from which most of the known (semi--)Fredholm stability results for small gap, compact perturbations and compact approximation follow. It is based on the following simple observation. 

The gap distance  depends on the norm ${\left\|x - y\right\|}$ of a difference between two vectors $x$ and $y$ from different subspaces of a Banach space $X$. Obviously that value is the same as ${\left\|x - y - \theta\right\|}$ where ${\theta \in X}$ is the null vector. By replacing $\theta$ with an arbitrary vector ${z \in X}$, and by considering the sequences of norms of differences ${\left\|x_{n} - y_{n} - z\right\|}$ instead of a single difference ${\left\|x - y\right\|}$, it is possible to define proximity of the sequences of Banach subspaces that extends the concept of gap and that parallels perturbations by operators with the small measure of non-compactness.

Following this idea, in Section \ref{S:ulabs} we define the concept of lower and upper $\lambda-$adjusted sequences of closed subspaces of a Banach space. That concept is rather general as it encompasses the concepts of gap, perturbations by operator with the small measure of non-compactness and $K_{2}-$approximation. Apparently due to this generality an upper semi--Fredholm pair of subspaces and upper semi--Fredholm operators may become unstable under upper $\lambda-$adjustment in an arbitrary Banach space as shown in Section \ref{S:eoi}. Still, in Section \ref{S:sp} we show that any lower semi--Fredholm pair of subspaces remains stable under lower  $\lambda-$adjustment. Also, lower $\lambda-$adjustment of continuous operators preserves Fredholm and lower semi--Fredholm properties.

In the next Section \ref{sibfus} we introduce $BFU-$spaces -- those are Banach spaces which dual is a Fr\'{e}chet-Urysohn topological space in ${weak^{*}}$ topology. The class of $BFU-$spaces is rather broad as it contains all separable and reflexive spaces and their finite products. It appears that in $BFU-$spaces upper semi--Fredholm pairs of subspaces and upper semi--Fredholm operators remain stable under upper $\lambda-$adjustment.

In the following Section \ref{S:umabs} we introduce another perturbation concept -- $\mu-$approximation. In essence, two sequences of subspaces from a Banach space $\mu-$approximate each other if they are $\lambda-$adjusted and if for any vector of the encompassing Banach space its close proximity to the subspaces from the first sequence means that it is also close enough to the subspaces from the second sequence. It appears that any (semi--)Fredholm property of a pair of subspaces from any Banach space is stable under $\mu-$approximation. Furthermore, properties of the essentially Kato operators remain stable under $\mu-$approximation of commuting operators as well. Note that essentially Kato stability has been previously proved only for perturbations by operators with the small norm (see \cite{goldman_krachkovskii2,goldman_krachkovskii22,goldman_krachkovskii3,forster_muller}).

In the last Section \ref{S:tccs} we apply the previously established results to what seems to be the new concepts of tuples of subspaces and of complexes of tuples of subspaces from Banach spaces. We present and discuss definitions of a tuple of subspaces, of a pair complex, of a tuple complex and of an inclusion complex. For each of them we present stability results under $\lambda-$adjustment and under $\mu-$approximation.

\subsection[Research Dateline]{Research Dateline}
A large number of the important results from Sections \ref{S:ulabs}, \ref{sibfus} and \ref{S:umabs} had been obtained (in a slightly different form) in the period of years 1975--1980. Some of the results from Section \ref{S:umabs} (together with the spectral stability theorems omitted from the current paper) where presented in the author's 1977 graduation thesis at the University of Latvia in Riga, Latvia. 

Other material from Sections \ref{S:ulabs}, \ref{sibfus} and \ref{S:umabs}, as well as the entire body of Section \ref{S:tccs} has been developed in 2007--2008 time period.
 
\subsection[Acknowledgments]{Acknowledgments}
The work between 1975--1980 has been conducted under the guidance of Mikhail Abramovi\v{c} Gol'dman to whom I am indebted for teaching me Mathematics.

Current research has been fully supported by my family -- I am  grateful to my wife and children for their kind patience and understanding while I was working on this paper during long late off--work hours.

\newpage
\subsection[Notational Conventions]{Notational Conventions}\label{S:conv}
$\mathbb{N}$ is a set of natural numbers, $\mathbb{N}^{'}$ is an infinite subset of $\mathbb{N}$, $\mathbb{N}^{''}$ is an infinite subset of $\mathbb{N}^{'}$, etc. When the number of nested subsets becomes high, we denote the subset at depth $n$ as $\mathbb{N}^{'n}$
\[
\mathbb{N}\ \supset\ \mathbb{N}^{'}\ \supset\ \mathbb{N}^{''}\ \supset\ \mathbb{N}^{'''}\ \supset\  \mathbb{N}^{'4}\ \supset\ \mathbb{N}^{'5}\ \supset\ ...
\]
When several lines of subsets are considered at the same time, we distinguish them by lower indices as follows:
\begin{multline}
\notag
\begin{aligned}
\mathbb{N}\ \supset\ \mathbb{N}^{'}\ \supset\ \mathbb{N}^{''}\ \supset\ ...\ \mathbb{N}^{'k}\ \supset\ \mathbb{N}^{'k+1}\ \supset\ ...\\
\mathbb{N}\ \supset\ \mathbb{N}_{0}^{'}\ \supset\ \mathbb{N}_{0}^{''}\ \supset\ ...\ \mathbb{N}_{0}^{'k}\ \supset\ \mathbb{N}_{0}^{'k+1}\ \supset\ ...\\
\mathbb{N}\ \supset\ \mathbb{N}_{1}^{'}\ \supset\ \mathbb{N}_{1}^{''}\ \supset\ ...\ \mathbb{N}_{1}^{'k}\ \supset\ \mathbb{N}_{1}^{'k+1}\ \supset\ ...\\
\end{aligned}
\end{multline}
A sequence of elements enumerated by elements ${n \in \mathbb{N}^{'}}$ is denoted by $(e_{n})_{\mathbb{N}^{'}}$. When more than one index is involved the running index from set ${\mathbb{N}^{'}}$ is specified with that set like $(e_{n_{m}})_{m \in \mathbb{N}^{'}}$ or  $(e^{k}_{m})_{m \in \mathbb{N}^{'}}$, etc.

If $(M_{n})_{\mathbb{N}^{'}}$ is a sequence of non-empty sets and $(x_{n})_{\mathbb{N}^{'}}$ is a sequence of elements such that $x_{n} \in M_{n}$ for $n \in \mathbb{N}^{'}$, then we say that $(x_{n})_{\mathbb{N}^{'}}\ is from\ (M_{n})_{\mathbb{N}^{'}}$ and write $(x_{n})_{\mathbb{N}^{'}} \triangleleft (M_{n})_{\mathbb{N}^{'}}$.

A vector $x$ from a unit sphere of a Banach space $X$ is called \emph{a unit vector}; a sequence of unit vectors $(x_{n})_{\mathbb{N}^{'}}$ is called \emph{a unit sequence}.

All \emph{subspaces} and \emph{operators} in Banach spaces are meant to be \emph{linear}.

\emph{A null element} of a Banach space is denoted by $\theta$; \emph{a null subspace} of a Banach space -- the one that consists of a single element $\theta$ -- is denoted by $\{\theta\}$; \emph{a null operator} from a Banach space $X$ to a Banach space $Y$ -- the one that maps every vector from $X$ into $\theta$ from $Y$ -- is also denoted by $\theta$.

If ${X^{1},...,X^{k}}$ are ${k\geq2}$ Banach spaces, then their direct product\linebreak[4] ${\prod_{i=1}^{k}X^{i}\ =\ X^{1} \times \cdots \times X^{k}}$ is a Banach space of ordered $k-$tuples ${(x^{1},...,x^{k})}$ such that ${x^{i} \in X^{i}}$ for each ${i=1,...,k}$; the norm on ${\prod_{i=1}^{k}X^{i}}$ is defined as ${\max\{\left\|x^{1}\right\|,...,\left\|x^{k}\right\|\}}$.

If $X$ and $Y$ are two Banach spaces, then $\mathcal{C}(X,Y)$ is a set of \emph{closed linear operators} -- those which graphs are closed in the product space $X \times Y$; domain $dom(A)$ of a closed operator ${A \in \mathcal{C}(X,Y)}$ may be \emph{a proper subspace} of $X$. 

The set of all \emph{continuous operators} from $\mathcal{C}(X,Y)$ is denoted by $\mathcal{BC}(X,Y)$. Note that operators from $\mathcal{BC}(X,Y)$ \emph{may not} be defined on the whole $X$.

By ${\mathcal{B}(X,Y)}$ we denote a Banach space, furnished with operator norm, of \emph{continuous (i.e. bounded) operators} defined \emph{on all} $X$ and acting into $Y$; ${\mathcal{K}(X,Y)}$ is a space of all \emph{compact operators}. Note inclusions
\[
\mathcal{K}(X,Y)\ \subset\ \mathcal{B}(X,Y)\ \subset\ \mathcal{BC}(X,Y)\ \subset\ \mathcal{C}(X,Y).
\]
\emph{Dimension} of a Banach space $X$, denoted by ${\dim X}$, is the power of a maximal set of linearly independent vectors; if it is not finite, then we write ${\dim X = \infty}$. When dimension numbers are finite, their addition and subtraction  follow usual rules of arithmetics; when at least one of the dimensions is ${\infty}$, then by definition the result of any addition or subtraction is ${\infty}$ as well.

\newtheorem{theorem}{Theorem}[subsection]
\newtheorem{lemma}[theorem]{Lemma}
\newtheorem{proposition}[theorem]{Proposition}
\newtheorem{definition}[theorem]{Definition}
\newtheorem{example}[theorem]{Example}

\newpage
\section[Uniform $\lambda-$Adjustment]{Uniform $\lambda-$Adjustment in Banach Spaces}\label{S:ulabs}

In this section we define lower and upper uniform $\lambda-$adjustment between sequences of closed subspaces of Banach spaces. Uniform $\lambda-$adjustment between sequences of closed linear operators is then defined as uniform $\lambda-$adjustment between their graph subspaces in the product space. Then we show that perturbations by small gap, small norm, small $q-$norm and ${K_{2}-}$approximation become just particular cases of uniform $\lambda-$adjustment, while some other perturbations previously not fit into any of the existing cases still satisfy definition of $\lambda-$adjustment. 

After that we present several examples when properties used to be stable under previous perturbations break down under uniform $\lambda-$adjustment. At the end we prove that lower uniform $\lambda-$adjustment with small $\lambda$ still preserves lower semi--Fredholm properties of a pair of subspaces and that lower $0-$adjustment preserves any lower or upper (semi--)Fredholm properties of continuous linear operators.

\subsection[Uniform $\lambda-$Adjustment]{Uniform $\lambda-$Adjustment}\label{S:ula}

\begin{definition}[Uniform $\lambda-$Adjustment of Sequences of Subspaces]\label{D:ulass} 
Let $(M_{n})_{\mathbb{N}^{'}}$ and $(P_{n})_{\mathbb{N}^{'}}$ be a pair of sequences of closed subspaces from a Banach space $X$, ${M_{n} \neq \{\theta\}}$ for all ${n \in \mathbb{N}^{'}}$ and $\lambda \geq 0$. We say that $(M_{n})_{\mathbb{N}^{'}}$ is lower $uniformly\ \lambda-adjusted$ with $(P_{n})_{\mathbb{N}^{'}}$ (equivalently -- $(P_{n})_{\mathbb{N}^{'}}$ is upper $uniformly\ \lambda-adjusted$ with $(M_{n})_{\mathbb{N}^{'}}$) if for any $\eta > 0$ and for any unit subsequence $(x_{n})_{\mathbb{N}^{''}}$ from $(M_{n})_{\mathbb{N}^{''}}$ there exists a subsequence $(y_{n})_{\mathbb{N}^{'''}}$ from $(P_{n})_{\mathbb{N}^{'''}}$ and a vector $z \in X$ such that 
\[
\varlimsup_{n \in \mathbb{N}^{'''}} \left\|x_{n} - y_{n} - z \right\|\ \leq\ \lambda + \eta. 
\]
The uniform $\lambda-$adjustment between $(M_{n})_{\mathbb{N}^{'}}$ and $(P_{n})_{\mathbb{N}^{'}}$ is a non-negative real number defined as
\[
\lambda_{\mathbb{N}^{'}}[M_{n}, P_{n}]\ :=\ \inf \{ \lambda \in \mathbb{R} \mid (M_{n})_{\mathbb{N}^{'}}\ \text{is lower uniformly $\lambda-$adjusted with}\ (P_{n})_{\mathbb{N}^{'}}\}.
\]
\end{definition}

\begin{definition}[Uniform $\lambda-$Adjustment Between a Subspace and a Sequence of Subspaces]\label{D:ulasss}
Similarly to the uniform adjustment of a pair of sequences of subspaces, we define uniform  adjustment between a subspace $M$ and a sequence of subspaces $(P_{n})_{\mathbb{N}^{'}}$, as well as uniform adjustment between a sequence of subspaces $(M_{n})_{\mathbb{N}^{'}}$ and a subspace $P$, denoting
\[
\begin{aligned}
&\lambda_{\mathbb{N}^{'}}[M, P_{n}]\ :=\ \lambda_{\mathbb{N}^{'}}[M_{n}, P_{n}]\ \ where\ \ M = M_{n}\ for\ n \in {N}^{'},\\
&\lambda_{\mathbb{N}^{'}}[M_{n}, P]\ :=\ \lambda_{\mathbb{N}^{'}}[M_{n}, P_{n}]\ \ where\ \ P = P_{n}\ for\ n \in {N}^{'}.
\end{aligned}
\]
\end{definition}

The uniform $\lambda-$adjustment is well defined between any two sequences of subspaces  $(M_{n})_{\mathbb{N}^{'}}$ and $(P_{n})_{\mathbb{N}^{'}}$ since for any unit subsequence $(x_{n})_{\mathbb{N}^{'}} \triangleleft (M_{n})_{\mathbb{N}^{'}}$ one can choose $y_{n} = \theta$ for $n \in \mathbb{N}^{'}$ and $z = \theta \in X$ so that $(y_{n})_{\mathbb{N}^{'}} \triangleleft (P_{n})_{\mathbb{N}^{'}}$ and
\[
\varlimsup_{n \in \mathbb{N}^{'}} \left\|x_{n} - y_{n} - z \right\| = \varlimsup_{n \in \mathbb{N}^{'}} \left\|x_{n} - \theta - \theta \right\| = 1. 
\]
Thus $\lambda_{\mathbb{N}^{'}}[M_{n}, P_{n}] \leq 1$ for any two sequences of subspaces $(M_{n})_{\mathbb{N}^{'}}$ and $(P_{n})_{\mathbb{N}^{'}}$.
\\
\\
Recall that if $x$ is a vector from a Banach space $X$, and $P$ is a closed subspace of $X$, then \emph{distance from x to P} is defined as 
\[
dist(x, P)\ :=\ \inf \{ \left\| x - y \right\| \mid y \in P \};
\]
if $M$, $P$ are two closed subspaces of a Banach space $X$ then the \emph{gap between M and P} is a real non-negative number defined as
\[ 
\delta(M, P)\ :=\ \sup \{ dist(x, P) \mid x \in M ,\ \left\| x \right\| = 1 \}.
\]

The concept of uniform $\lambda-$adjustment generalizes the concept of gap distance as shown in the following lemma:
\begin{lemma}[Uniform $\lambda-$Adjustment is Weaker than Gap]\label{L:uawg}
Small gap means small adjustment: 
\[
	\lambda_{\mathbb{N}^{'}}[M_{n}, P_{n}]\ \le\ \varlimsup_{n \in \mathbb{N}^{'}}  \delta(M_{n}, P_{n}).
\]
\end{lemma}
\begin{proof}
Consider a unit subsequence $(x_{n})_{\mathbb{N}^{''}} \triangleleft (M_{n})_{\mathbb{N}^{''}}$ and a number ${\eta > 0}$. Definition of ${\delta(M_{n}, P_{n})}$ implies that for every vector ${x_{n} \in M_{n}}$ one can find a vector ${y_{n} \in P_{n}}$ such that ${\left\|x_{n} - y_{n}\right\| < \delta(M_{n}, P_{n}) + \eta}$; hence, setting ${z = \theta}$, we see that 
\begin{multline}
\notag
\begin{aligned}
\varlimsup_{n \in \mathbb{N}^{''}} \left\| x_{n} - y_{n} - z \right\| &\ =\ \varlimsup_{n \in \mathbb{N}^{''}} \left\| x_{n} - y_{n} - \theta \right\|\\
                                                                      &\ =\ \varlimsup_{n \in \mathbb{N}^{''}} \left\| x_{n} - y_{n} \right\|\\
                                                                      &\ \le\ \varlimsup_{n \in \mathbb{N}^{''}} \delta(M_{n}, P_{n}) + \eta. 
\end{aligned}
\end{multline}
Therefore, according to definition of $\lambda-$adjustment we shall conclude that
\[
\lambda_{\mathbb{N}^{'}}[M_{n}, P_{n}]\ \le\ \varlimsup_{n \in \mathbb{N}^{'}}  \delta(M_{n}, P_{n}).\qedhere
\]
\end{proof}

\subsection[The Small Uniform $\lambda-$Adjustment Theorem]{The Small Uniform $\lambda-$Adjustment Theorem}\label{S:sulat}
Throughout this article we will see a significant interplay between the concepts of gap and uniform  $\lambda-$adjustment; in particular, we will be using the following well known theorem (see \cite{krein_krasnoselskii_milman,gohberg_krein,kato,ostrovskii}:

\begin{theorem}[The Small Gap Theorem]\label{T:sgp}
Suppose that ${M}$ and ${P}$ are two closed subspaces from a Banach space ${X}$. Then
\begin{enumerate}
  \item If ${\dim M < \infty}$ and ${\dim M < \dim P}$, then there exists a unit vector ${x \in P}$ such that ${dist(x, M) = 1}$.
  \item	If ${\delta(M, P) < 1}$, then ${\dim M \leq \dim P}$.
  \item If ${\delta(M, P) < 1}$, then ${\dim X/P \leq \dim X/M}$.
  \item If ${\delta(X, P) < 1}$, then ${\delta(X, P) = 0}$ and ${X = P}$.
\end{enumerate}
\end{theorem}

Let us use the Small Gap Theorem in order to prove a similar statement for uniform $\lambda-$adjustment:

\begin{theorem}[The Small Uniform Adjustment Theorem]\label{T:suat}
Let ${({M}_n)_{\mathbb{N}^{'}}}$ and ${({P}_n)_{\mathbb{N}^{'}}}$ be two sequences of closed subspaces, $P$ be a closed subspace from a Banach space $X$. Then
\begin{enumerate}
  \item If ${\lambda_{\mathbb{N}^{'}}[M_{n}, P_{n}] < 1/2}$, then there exist ${K \in \mathbb{N}^{'}}$, ${L \in \mathbb{N}}$ such that for ${m > K}$
\[
\dim M_{m}\ \leq\ L + \dim P_{m}.
\]
  \item For every ${\epsilon > 0}$ there exists a natural number ${K_{\epsilon} \in \mathbb{N}}$ and a\linebreak[4] finite-dimensional subspace ${Q_{\epsilon} \subset \sum_{i=1}^{K_{\epsilon}} M_{i}}$ such that
\begin{equation}\label{1910}
\lambda_{\mathbb{N}^{'}}[M_{n}, P_{n}] + \epsilon\ \geq\ \frac{1}{2} \times \varlimsup_{n \in \mathbb{N}^{'}} \delta(M_{n}, Q_{\epsilon} + \sum_{j=1}^{n} P_{j}).
\end{equation}
  \item If ${\lambda_{\mathbb{N}^{'}}[M_{n}, P] < 1/2}$ then there exist ${K \in \mathbb{N}^{'}}$, ${L \in \mathbb{N}}$ such that for ${m > K}$
\[
\dim X/P\ +\ L\ \leq\ \dim X/M_{m}.
\]
  \item If ${M_{n} = M \subset P_{n}}$ for ${n \in \mathbb{N^{'}}}$ and ${\dim P_{n} \rightarrow \infty}$, then
\[
0 = \lambda_{\mathbb{N}^{'}}[M, P_{n}],\ \ \ \lambda_{\mathbb{N}^{'}}[P_{n}, M]\ \geq\  \frac{1}{2}
\]
  \item If ${\lambda_{\mathbb{N}^{'}}[M_{n}, \{\theta\}] < 1/2}$ then there exists a closed subspace ${Q \subset X}$ with a finite co-dimension ${\dim X/Q < \infty}$ and a natural number ${K \in \mathbb{N}^{'}}$ such that ${M_{m} \cap Q = \{\theta\}}$ for all ${m > K}$ and all norms of natural projectors ${Pr_{m} : M_{m} \oplus Q \rightarrow Q}$ are bounded from above. 
\end{enumerate}
\end{theorem}
\begin{proof}
If for large enough $n$ all ${\dim P_{n}}$ are infinite then proposition 1 is trivially true. Thus we shall only consider a case when infinitely many numbers ${\dim P_{n}}$ are finite, and suppose to the contrary that there exists a subsequence ${\mathbb{N}^{''}}$ such that
\[
\lim_{n \in \mathbb{N}^{''}} ( \dim M_{n} - \dim P_{n}) \rightarrow \infty.
\]
We may assume that all dimensions ${\dim M_{n}}$ are finite as well -- otherwise we could consider finite-dimensional subspaces ${M^{'}_{n} \subset M_{n}}$ with the same property. Now let us build inductively a subsequence ${\mathbb{N}^{'''}}$ such that for all ${n \in \mathbb{N}^{'''}}$
\[
\dim M_{n}\ >\ \dim\ ( \sum_{i \in \mathbb{N}^{'''}, i < n} M_{i}\ +\ \sum_{i \in \mathbb{N}^{'''}, i \leq n} P_{i} ):
\]
\begin{itemize}
  \item According to our assumption there exists ${n_{1} \in \mathbb{N}^{'''}}$ such that\linebreak[4] ${\dim M_{n_{1}} - \dim P_{n_{1}} > 0}$;
  \item suppose that we have already built such numbers ${n_{1}, n_{2}, ..., n_{k}}$ up until $k$. Since all ${M_{n_{i}}}$ and ${P_{n_{i}}}$ are finite-dimensional, the sum of their dimensions ${S_{k} \in \mathbb{N}}$ is finite. According to our assumption there exists some ${r > n_{k}, r \in \mathbb{N}^{'''}}$ such that ${\dim M_{r} - \dim P_{r} > S_{k}}$. If we set ${n_{k+1} = r}$ then it is easy to see that
\begin{multline}
\notag
\begin{aligned}
\dim M_{n_{k+1}} &\ >\ \dim P_{n_{k+1}} + S_{k}\\
                 &\ =\ \dim P_{n_{k+1}}\ +\ ( \sum_{i \leq k} \dim M_{n_{i}}\ +\ \sum_{i \leq k} \dim P_{n_{i}} )\\
                 &\ \geq\ \dim\ ( \sum_{i < k+1} M_{n_{i}}\ +\ \sum_{i \leq k+1} P_{n_{i}} ).
\end{aligned}
\end{multline}
\end{itemize}
According to our construction, we can apply proposition 1 from the Small Gap Theorem to each pair of subspaces
\[
( M_{n_{k+1}}, \sum_{i < k+1} M_{n_{i}}\ +\ \sum_{i \leq k+1} P_{n_{i}} )
\]
choosing a unit subsequence ${(x_{n})_{n \in \mathbb{N}^{'''}} \triangleleft (M_{n})_{n \in \mathbb{N}^{'''}}}$ such that
\begin{equation}\label{Ex:23}
dist(x_{n}, \sum_{i \in \mathbb{N}^{'''}, i < n} M_{i}\ +\ \sum_{i \in \mathbb{N}^{'''}, i \leq n} P_{i} )\ =\ 1.
\end{equation}
Now choose ${\eta > 0}$ so that ${\lambda_{\mathbb{N}^{'}}[M_{n}, P_{n}] + \eta\ <\ \frac{1}{2}}$, also choose a subsequence ${(y_{n})_{\mathbb{N}^{'4}} \triangleleft (P_{n})_{\mathbb{N}^{'4}}}$ and ${z \in X}$ such that
\[
\varlimsup_{n \in \mathbb{N}^{'4}} \left\|x_{n} - y_{n} - z\right\| < \lambda_{\mathbb{N}^{'}}[M_{n}, P_{n}] + \eta.
\]
Hence, applying the triangle inequality for large enough ${n > m}$, ${n, m \in \mathbb{N}^{'4}}$ we obtain
\begin{multline}
\notag
\begin{aligned}
\left\|x_{n} - (x_{m} - y_{m} + y_{n})\right\| &\ =\ \left\|x_{n} - y_{n} - x_{m} + y_{m}\right\|\\
                                               &\ =\ \left\|(x_{n} - y_{n} - z) - (x_{m} - y_{m} - z)\right\|\\
                                               &\ \leq\ \left\|(x_{n} - y_{n} - z)\right\| + \left\|(x_{m} - y_{m} - z)\right\|\\
                                               &\ <\ 2 \times (\lambda_{\mathbb{N}^{'}}[M_{n}, P_{n}] + \eta)\ <\ 1.
\end{aligned}
\end{multline}
However, 
\[
x_{m} - y_{m} + y_{n} \in \sum_{i \in \mathbb{N}^{'''}, i < n} M_{i}\ +\ \sum_{i \in \mathbb{N}^{'''}, i \leq n} P_{i},
\]
Therefore
\[
dist(x_{n}, \sum_{i \in \mathbb{N}^{'''}, i < n} M_{i}\ +\ \sum_{i \in \mathbb{N}^{'''}, i \leq n} P_{i} )\ <\ 1.
\]
which directly contradicts \eqref{Ex:23}. Therefore our assumption is incorrect and we shall conclude that there exist ${K \in \mathbb{N}^{'}}$ and ${L \in \mathbb{N}}$ such that ${\dim M_{m}\ \leq\ L + P_{m}}$ for ${m > K}$ which concludes the proof of proposition 1.\\

Suppose that proposition 2 is not true for some ${\epsilon > 0}$ and choose any unit vector ${x_{1}}$ from ${M_{1}}$, set ${n_{1} = 1}$, ${Q_{1} = sp(\{x_{n_{1}}\})}$ -- a linear space spanned on vector ${x_{n_{1}}}$. Then, since ${\dim Q_{1} < \infty}$, there shall exist some ${n_{2} > n_{1}}$ and a unit vector ${x_{n_{2}} \in M_{n_{2}}}$ such that
\[
	dist(x_{n_{2}},\ Q_{1} + \sum_{j=1}^{n_{2}} P_{j})\ >\ 2 \times (\lambda_{\mathbb{N}^{'}}[M_{n}, P_{n}] + \epsilon).
\]
Next set ${Q_{2} = sp(\{x_{n_{1}}, x_{n_{2}}\})}$ -- a linear space spanned on both ${x_{n_{1}}}$ and ${x_{n_{2}}}$. Again, following our assumption there shall exist some ${n_{3} > n_{2}}$ such that
\[
	dist(x_{n_{3}},\ Q_{2} + \sum_{j=1}^{n_{3}} P_{j})\ >\ 2 \times (\lambda_{\mathbb{N}^{'}}[M_{n}, P_{n}] + \epsilon).
\]
Continuing this process, one can build sequences of numbers ${\mathbb{N^{''}} = (n_{k})_{k \in \mathbb{N}}}$ and unit vectors ${(x_{n_{k}})_{k \in \mathbb{N}} \triangleleft (M_{n_{k}})_{k \in \mathbb{N}}}$ and a sequence of finite-dimensional subspaces ${Q_{k} = sp(\{x_{n_{1}}, x_{n_{2}}, ..., x_{n_{k}}\}) \subset \sum_{j=i}^{n_{k}} M_{j}}$ such that
\begin{equation}\label{E:20}
	dist(x_{n_{k+1}},\ Q_{k} + \sum_{j=1}^{n_{k+1}} P_{j})\ >\ 2 \times (\lambda_{\mathbb{N}^{'}}[M_{n}, P_{n}] + \epsilon).
\end{equation}
Now recall that by definition of uniform $\lambda-$adjustment, for the just built unit subsequence ${(x_{n})_{\mathbb{N}^{''}}}$ and for ${\epsilon > 0}$ one can find a subsequence ${(y_{n})_{\mathbb{N}^{'''}} \triangleleft (P_{n})_{\mathbb{N}^{'''}}}$ and a vector ${z \in X}$ such that
\[
\varlimsup_{n \in \mathbb{N}^{'''}} \left\|x_{n} - y_{n} - z\right\|\ \leq\ \lambda_{\mathbb{N}^{'}}[M_{n}, P_{n}] + \epsilon.
\]
Therefore, there exists ${K \in \mathbb{N}^{'''}}$ such that for all ${n > K}$, ${n \in \mathbb{N}^{'''}}$:
\begin{multline}
\notag
\begin{aligned}
\varlimsup_{n \in \mathbb{N}^{'''}} \left\|x_{n} - ( x_{K} + y_{n} - y_{K})\right\| 
              &\ =\ \varlimsup_{n \in \mathbb{N}^{'''}} \left\|x_{K} - y_{K} - x_{n} + y_{n})\right\|\\
              &\ =\ \varlimsup_{n \in \mathbb{N}^{'''}} \left\|(x_{K} - y_{K} - z) - (x_{n} - y_{n} - z)\right\|\\
              &\ \leq\ \varlimsup_{n \in \mathbb{N}^{'''}} \left\|x_{K} - y_{K} - z\right\| + \varlimsup_{n \in \mathbb{N}^{'''}} \left\|x_{n} - y_{n} - z\right\|\\
              &\ \leq\ 2 \times (\lambda_{\mathbb{N}^{'}}[M_{n}, P_{n}] + \epsilon).
\end{aligned}
\end{multline}
Since ${x_{n} \in M_{n}}$ is a unit vector, ${x_{K} \in Q_{K}}$, and ${y_{n} - y_{K} \in P_{n} + P_{K}}$, we observe that for ${n > K}$
\[
dist(x_{n},\ Q_{K} + P_{n} + P_{K})\ \leq\ 2 \times (\lambda_{\mathbb{N}^{'}}[M_{n}, P_{n}] + \epsilon).
\]
However, if ${n > K}$, then ${Q_{m} \supseteq Q_{K}}$ for all ${m \in [K, n)}$ and ${\sum_{j=1}^{n} P_{j} \supset (P_{n} + P_{K})}$; therefore, for the same $m$ and $n$
\begin{equation}\label{E:21}
dist(x_{n},\ Q_{m} + \sum_{j=1}^{n} P_{j})\ \leq\ dist(x_{n},\ Q_{K} + P_{n} + P_{K})\ \leq\ 2 \times (\lambda_{\mathbb{N}^{'}}[M_{n}, P_{n}] + \epsilon).
\end{equation}
Now combining \eqref{E:20} and \eqref{E:21} and recalling that ${\epsilon > 0}$ we obtain a contradiction:
\[
2 \times (\lambda_{\mathbb{N}^{'}}[M_{n}, P_{n}] + \epsilon)\ <\ 2 \times (\lambda_{\mathbb{N}^{'}}[M_{n}, P_{n}] + \epsilon).
\]
Thus our assumption shall be invalid, hence proposition 2 is true.\\

In order to prove proposition 3, choose numbers
${\epsilon}$ and ${\eta}$ so that
\[
\lambda_{\mathbb{N}^{'}}[M_{n}, P] < \eta < \frac{1}{2},\ \  0 < \epsilon < \frac{1}{2} - \eta. 
\]
Since all ${P_{n}}$ are equal to ${P}$, we may rewrite \eqref{E:21}:
\begin{multline}
\notag
\begin{aligned}
\varlimsup_{n \in \mathbb{N}^{'}} \delta(M_{n},\ Q_{\epsilon} + \sum P) &\ =\ \varlimsup_{n \in \mathbb{N}^{'}} \delta(M_{n},\ Q_{\epsilon} + P)\\
                                                                        &\ \leq\ 2 \times (\lambda_{\mathbb{N}^{'}}[M_{n}, P] + \epsilon)\\
                                                                        &\ <\ 2 \times (\eta + \epsilon)\ <\ 2 \times \frac{1}{2}\ =\ 1.
\end{aligned}
\end{multline}
Applying proposition 3 from the Small Gap Theorem to the pairs of subspaces ${(M_{n},\ Q_{\epsilon} + P)}$ for large enough ${n \in \mathbb{N}^{'}}$, and noticing that ${\dim Q_{\epsilon} < \infty}$ we obtain
\[
\dim X/M_{n}\ \leq\ \dim X/(P + Q_{\epsilon})\ \leq\ \dim X/P - \dim Q_{\epsilon},
\]
\[
\dim X/M_{n}\ +\ \dim Q_{\epsilon}\ \leq\ \dim X/P,
\]
which concludes the proof of the proposition 3.\\

In order to prove proposition 4 note that equality ${0 = \lambda_{\mathbb{N}^{'}}[M, P_{n}]}$ trivially follows from the definition of $\lambda-$adjustment. For the second relation rewrite \eqref{1910}
\[
\lambda_{\mathbb{N}^{'}}[P_{n}, M] + \epsilon\ \geq\ \frac{1}{2} \times \varlimsup_{n \in \mathbb{N}^{'}} \delta(P_{n}, Q_{\epsilon} + \sum M) \ =\ \frac{1}{2} \times \varlimsup_{n \in \mathbb{N}^{'}} \delta(P_{n}, Q_{\epsilon} + M).
\]
Now consider quotient subspaces ${P_{n}/M}$ and ${(Q_{\epsilon} + M)/M}$ in the quotient space ${X/M}$. Since ${\dim (Q_{\epsilon} + M)/M < \infty}$ by construction and since ${\dim P_{n}/M \rightarrow \infty}$ by assumption, we may apply proposition 2 from the Small Gap Theorem
obtaining equality 
\[
1\ =\ \lim_{n \in \mathbb{N}^{'}}\delta(P_{n}/M,\ (Q_{\epsilon} + M)/M)\ =\ \lim_{n \in \mathbb{N}^{'}}\delta(P_{n},\ Q_{\epsilon} + M); 
\]
hence ${\lambda_{\mathbb{N}^{'}}[P_{n}, M] + \epsilon\ \geq\ \frac{1}{2}}$. Since ${\epsilon}$ can be arbitrarily small we conclude that ${\lambda_{\mathbb{N}^{'}}[P_{n}, M]\ \geq\ \frac{1}{2}.}$\\

In order to prove proposition 5 we first establish that there exists a number ${\eta > 0}$, a closed subspace ${Q \subset X}$, and a number ${K \in \mathbb{N}^{'}}$ such that ${\dim X/Q < \infty}$ and ${\delta(M_{m}, Q) > \eta}$ for all ${m > K, m \in \mathbb{N}^{'}}$. Assume the opposite -- then for any number ${2^{-n}}$, ${n \in \mathbb{N}}$, for any closed subspace ${Q \subset X}$ such that ${\dim X/Q < \infty}$ and for any number ${K \in \mathbb{N}^{'}}$ there exists a number ${m > K, m \in \mathbb{N}^{'}}$ such that  ${\delta(M_{m}, Q) \leq 2^{-n}}$. In this case build two sequences of vectors ${(x_{n})_{\mathbb{N}^{''}} \triangleleft (M_{n})_{\mathbb{N}^{''}}}$ and ${(t_{n})_{\mathbb{N}^{''}} \subset X}$ following the inductive process:
\begin{itemize}
  \item According to Hahn-Banach theorem the dual space ${X^{*}}$ contains a functional $f$ with the norm $1$. Since ${\dim X/Ker(f) = 1 < \infty}$, due to our assumption there exists some ${n_{1} \in \mathbb{N}^{'}}$ such that ${\delta(M_{n_{1}}, Ker(f)) < 2^{-1}}$.  Denote ${f_{n_{1}} = f}$, ${Q_{n_{1}} = Ker(f_{n_{1}})}$, ${n_{0} = 0}$ and ${Q_{n_{0}} = X}$. Due to the gap inequality there can be found a unit vector ${x_{n_{1}} \in M_{n_{1}}}$ and a vector ${q_{n_{1}} \in Q_{n_{1}}}$ such that ${\left\|x_{n_{1}} - q_{n_{1}}\right\| < 2^{-1}}$. Also define ${q_{n_{0}} = q}$ where $q$ is a vector such that ${fq = 1}$ -- it is obvious that ${X = Q_{n_{0}} = sp(\{q_{n_{0}}\}) \oplus Q_{n_{1}}}$ and\linebreak[4] ${dist(q_{n{_{0}}}, Q_{n_{1}}) = \left\|q_{0}\right\| = 1}$. 
  \item Suppose that for ${k \geq 1}$ we have already found $k$ numbers ${\{n_{1}, ..., n_{k}\} \subset \mathbb{N}^{'}}$, $k$ closed subspaces ${ X = Q_{n{0}} \supset Q_{n_{1}} \supset ... \supset Q_{n_{k}}}$, $k$ unit vectors ${x_{n_{i}} \in M_{n_{i}}}$, $k$ non-null vectors ${q_{n_{i}} \in Q_{n_{i}}}$ such that ${Q_{n_{i-1}} = sp(\{q_{n_{i-1}}\}) \oplus Q_{n_{i}}}$,\linebreak[4] ${\left\|x_{n_{i}} - q_{n_{i}}\right\| < 2^{-n_{i}}}$ for all ${i \in \{1, ..., k\}}$, and ${dist(q_{n{_i}}, Q_{n_{j}}) = \left\|q_{n_{i}}\right\|}$ for any ${i,j \in \{1, ..., k\}, i < j}$. Since ${q_{n_{k}}}$ is a non-null vector, by the Hahn-Banach theorem there exists a unit functional ${f \in X^{*}}$ such that ${fq_{n_{k}} = \left\|q_{n_{k}}\right\|}$. Denote ${Q = Q_{n_{k}} \cap Ker(f)}$. Obviously ${Q_{n_{k}} \supset Q}$, ${Q_{n_{k}} = sp(\{q_{n_{k}}\}) \oplus Q}$.\linebreak[4] Since ${\dim X/Q = k+1 < \infty}$, according to our assumption we can find a number ${n_{k+1} \in \mathbb{N}^{'}}$, ${n_{k+1} > n_{k}}$, a unit vector ${x_{n_{k+1}} \in M_{n_{k+1}}}$ and a vector ${q_{n_{k+1}} \in Q}$ such that ${\left\|x_{n_{k+1}} - q_{n_{k+1}}\right\| < 2^{-(k+1)}}$. Now rename ${Q_{n_{k+1}} = Q}$;\linebreak[4] obviously ${dist(q_{n{_i}}, Q_{n_{k+1}}) = \left\|q_{n_{i}}\right\|}$ for any ${i < k+1, i \in \{1, ..., k\}}$.
\end{itemize}

Denote ${\mathbb{N}^{''} = \{n{1} < ... < n_{k} < ... \}}$, consider the sequence ${(x_{n})_{\mathbb{N}^{''}} \triangleleft (M_{n})_{\mathbb{N}^{''}}}$ and choose a real number ${\epsilon \in (0, \frac{1}{2} - \lambda_{\mathbb{N}^{'}}[M_{n}, \{\theta\}])}$. According to definition of uniform $\lambda-$adjustment there exist a subsequence ${\mathbb{N}^{'''}}$ and a vector ${z \in X}$ such that for large enough ${n \in \mathbb{N}^{'''}}$
\[
\left\| x_{n} - z \right\|\ \leq\ \lambda_{\mathbb{N}^{'}}[M_{n}, \{\theta\}] + \epsilon.
\]
Then, taking this into account, as well as the inequality ${\left\|q_{n} - q_{m}\right\| \geq \left\|q_{n}\right\|}$ from our construction, and using the triangle inequality we can estimate for the same ${n,m \in \mathbb{N}^{'''}, n > m}$:
\begin{multline}
\notag
\begin{aligned}
\left\|q_{n}\right\| \leq \left\|q_{n} - q_{m}\right\| &\ =\ \left\|q_{n} - x_{n} + x_{n} - q_{m}\right\|\\
                                                       &\ \leq\ \left\|q_{n} - x_{n}\right\| + \left\|x_{n} - q_{m}\right\|\\
                                                       &\ \leq\ 2^{-n} + \left\|x_{n} - q_{m}\right\|\\
                                                       &\ =\ 2^{-n} + \left\|x_{n} - x_{m} + x_{m} - q_{m}\right\|\\
                                                       &\ \leq\ 2^{-n} + \left\|x_{n} - x_{m}\right\| + \left\|x_{m} - q_{m}\right\|\\
                                                       &\ \leq\ 2^{-n} + \left\|x_{n} - x_{m}\right\| + 2^{-m}\\
                                                       &\ =\ 2^{-n} + 2^{-m} + \left\|x_{n} - x_{m}\right\|\\
                                                       &\ <\ 2^{-m+1}\ +\ \left\|x_{n} - z + z - x_{m}\right\|\\
                                                       &\ \leq\ 2^{-m+1}\ +\ \left\|x_{n} - z\left\| + \right\|z - x_{m}\right\|\\
                                                       &\ \leq\ 2^{-m+1}\ +\ \lambda_{\mathbb{N}^{'}}[M_{n}, \{\theta\}] + \epsilon\ +\  \lambda_{\mathbb{N}^{'}}[M_{n}, \{\theta\}] + \epsilon\\
                                                       &\ =\ 2^{-m+1}\ +\ 2 \times (\lambda_{\mathbb{N}^{'}}[M_{n}, \{\theta\}] + \epsilon).
\end{aligned}
\end{multline} 
Now recall that by construction ${\left\|x_{n} - q_{n}\right\| < 2^{-n}}$, therefore, since each ${x_{n}}$ is a unit vector
\[
1\ =\ \left\|x_{n}\right\|\ =\ \left\|x_{n} - q_{n} + q_{n}\right\|\ \leq\ \left\|x_{n} - q_{n}\right\| + \left\|q_{n}\right\|\ <\ 2^{-n} + \left\|q_{n}\right\|.
\]
We may now rewrite the previous two inequalities like this:
\[
1 - 2^{-n}\ <\ \left\|q_{n}\right\|\ <\ 2^{-m+1}\ +\ 2 \times (\lambda_{\mathbb{N}^{'}}[M_{n}, \{\theta\}] + \epsilon).
\]
Following simple arithmetic calculation we obtain:
\[
\frac{1}{2}\ -\ \lambda_{\mathbb{N}^{'}}[M_{n}, \{\theta\}]\ -\ \epsilon\ <\ 2^{-n-1} + 2^{-m}.
\]
However, the left part is a fixed positive number according to our choice of ${\epsilon}$, while the right part converges to $0$ as $n$ and $m$ approach infinity which is an obvious contradiction. Hence, we must conclude that our assumption is incorrect, and, therefore there exists ${L \in \mathbb{N}^{'}}$, a closed subspace ${Q \subset X}$ and a number ${\nu > 0}$ such that ${\dim X/Q = K < \infty}$ and ${\delta(M_{m}, Q) \geq \nu}$ for any number ${m > K, m \in \mathbb{N}^{'}}$.

Now suppose that ${d_{m} = \dim M_{m} \cap Q > 0}$ for infinite number of indices $m$. In this case we know that ${d_{m} < \dim M_{m}}$ since ${\delta(M_{m}, Q) \geq \nu > 0}$. Obviously ${\lambda_{\mathbb{N}^{'}}[M_{n} \cap Q, \{\theta\}] \leq \lambda_{\mathbb{N}^{'}}[M_{n}, \{\theta\}] < 1/2}$; therefore applying the previous considerations we can find another ${Q_{1} \subset Q}$ such that ${\delta(M_{m}\cap Q, Q_{1}) \geq \nu_{1} > 0}$ for large enough $m$. 

However, from proposition 1 it follows that dimensions of all ${M_{m}}$ are limited for large enough $m$. Therefore, we can continue this process only a finite number of times $l$ until we reach some ${Q_{l} \subset X}$ such that ${\dim X/Q_{l} < \infty}$, ${M_{m} \cap Q_{l} = \{\theta\}}$ and ${\delta(M_{m}, Q_{l}) \geq \eta = \min\{\nu_{1}, ..., \nu_{l}\} > 0}$. Therefore, for the same $m$ we may define natural projectors ${Pr_{m} : M_{m} \oplus Q_{l} \rightarrow Q_{l}}$ -- obviously their norms are bounded from above by ${1 + \frac{1}{\eta}}$. This concludes the proof of the proposition 5 and of the entire theorem
\end{proof}

Notice that estimates in the Small Uniform Adjustment Theorem involve constant $1/2$ while estimates in the Small Gap Theorem involve constant $1$. It appears that the estimate of $1/2$ is precise when no limitations are placed on the Banach space $X$ -- examples of that will be shown in the next Subsection \ref{S:eoi}. However, we will see further that these estimates can be improved from $1/2$ to $1$ for a Banach space which dual is Fr\'{e}chet-Urysohn space in ${weak^{*}}$ topology.

It is also worth noticing that study of uniform adjustment is interesting only for subspaces from \emph{infinite-dimensional} Banach spaces because for finite-dimensional Banach spaces it is trivial:

\begin{lemma}\label{L:fdt}
Any two sequences ${(M_{n})_{\mathbb{N}^{'}}}$ and ${(P_{n})_{\mathbb{N}^{'}}}$ of subspaces from a finite-dimensional Banach space ${X}$ are always uniformly $0-$adjusted.
\end{lemma}
\begin{proof}
Indeed, recall that if ${\dim X < \infty}$, then any unit sequence ${(x_{n})_{\mathbb{N}^{''}}}$ contains a convergent subsequence ${(x_{n})_{\mathbb{N}^{'''}} \rightarrow z \in X}$. Hence, for a subsequence ${(x_{n})_{\mathbb{N}^{''}} \triangleleft (M_{n})_{\mathbb{N}^{''}}}$ we may set ${y_{n} = \theta \in P_{n}}$ for ${n \in \mathbb{N}^{'''}}$ and obtain
\[
\varlimsup_{n \in \mathbb{N}^{'''}} \left\|x_{n} - y_{n} - z \right\|\ =\ \varlimsup_{n \in \mathbb{N}^{'''}} \left\|x_{n} - \theta - z \right\|\ =\ \varlimsup_{n \in \mathbb{N}^{'''}} \left\|x_{n} - z \right\|\ =\ 0.
\]
Therefore, ${\lambda_{\mathbb{N}^{'}}[M_{n}, P_{n}] = 0}$.
\end{proof}

\subsection[Uniform $\lambda-$Adjustment of Operators]{Uniform $\lambda-$Adjustment for Operators}\label{S:ulaflo}
We now define uniform $\lambda-$adjustment for linear closed operators. 

\begin{definition}[Uniform $\lambda-$Adjustment of Sequences of Closed Linear Operators]\label{D:ulasclo} 
Let ${X}$ and ${Y}$ are two Banach spaces. Consider $(A_{n})_{\mathbb{N}^{'}}$ and $(
B_{n})_{\mathbb{N}^{'}}$ -- a pair of sequences of operators from $\mathcal{C}(X,Y)$, as well as $(G_{A_{n}})_{\mathbb{N}^{'}}$ and $(G_{B_{n}})_{\mathbb{N}^{'}}$ -- sequences of their respective graphs in the product space ${X \times Y}$ (recall that a graph of a closed operator $A \in \mathcal{C}(X,Y)$ is a closed subspace of ${X \times Y}$ defined as set of ordered pairs ${\{(x, Ax) \mid x \in Dom(A) \subset X\}}$). 

We say that $(A_{n})_{\mathbb{N}^{'}}$ is lower $uniformly\ \lambda-adjusted$ with $(B_{n})_{\mathbb{N}^{'}}$ (equivalently -- $(B_{n})_{\mathbb{N}^{'}}$ is upper $uniformly\ \lambda-adjusted$ with $(A_{n})_{\mathbb{N}^{'}}$) if the sequence of graphs $(G_{A_{n}})_{\mathbb{N}^{'}}$ is lower $uniformly\ \lambda-adjusted$ with the sequence of graphs $(G_{B_{n}})_{\mathbb{N}^{'}}$ in the product space ${X \times Y}$.

The uniform $\lambda-$adjustment between $(A_{n})_{\mathbb{N}^{'}}$ and $(B_{n})_{\mathbb{N}^{'}}$ is defined as uniform $\lambda-$adjustment between their sequences of graphs:
\[
\lambda_{\mathbb{N}^{'}}[A_{n}, B_{n}]\ :=\ \lambda_{\mathbb{N}^{'}}[G_{A_{n}}, G_{B_{n}}].
\]
\end{definition}

\begin{definition}[Uniform $\lambda-$Adjustment between an Operator and a Sequence of Operators]\label{D:ulaoso}
Similarly to the uniform adjustment of a pair of sequences of operators from $\mathcal{C}(X,Y)$, we define uniform  adjustment between an operator $A \in \mathcal{C}(X,Y)$ and a sequence of operators ${(B_{n})_{\mathbb{N}^{'}} \subset \mathcal{C}(X,Y)}$, as well as uniform adjustment between a sequence of operators ${(A_{n})_{\mathbb{N}^{'}} \subset \mathcal{C}(X,Y)}$ and an operator $B \in \mathcal{C}(X,Y)$, denoting
\[
\begin{aligned}
&\lambda_{\mathbb{N}^{'}}[A, B_{n}]\ :=\ \lambda_{\mathbb{N}^{'}}[A_{n}, B_{n}]\ \ where\ \ A = A_{n}\ for\ n \in {N}^{'},\\
&\lambda_{\mathbb{N}^{'}}[A_{n}, B]\ :=\ \lambda_{\mathbb{N}^{'}}[A_{n}, B_{n}]\ \ where\ \ B = B_{n}\ for\ n \in {N}^{'}.
\end{aligned}
\]
\end{definition}

Uniform adjustment of operators is non-trivial only in case when Y has an infinite dimension

\begin{lemma}\label{L:fdto}
If ${X}$ and ${Y}$ are two Banach spaces and $\dim Y < \infty$, then any sequence of operators $(A_{n})_{\mathbb{N}^{'}}$ from $\mathcal{C}(X, Y)$ is lower uniformly $0-$adjusted with a null operator $\theta$ from $\mathcal{B}(X, Y)$. 
\end{lemma}
\begin{proof}
Consider a bounded sequence ${(x_{n},A_{n}x_{n})_{\mathbb{N}^{''}} \triangleleft (G_{A_{n}})_{\mathbb{N}^{''}}}$. Obviously the sequence  $(A_{n}x_{n})_{\mathbb{N}^{''}}$ is bounded in $Y$. Therefore, since $\dim Y < \infty$, there exists a subsequence $(A_{n}x_{n})_{\mathbb{N}^{'''}}$ convergent to some $t \in Y$. Now set
\[
y_{n}\ =\ (x_{n}, \theta)\ \in\ G_{\theta}\ \subset\ X \times Y
\]
for ${n \in \mathbb{N}^{'''}}$ and ${z = (\theta, t) \in X \times Y}$. Then
\begin{multline}\label{E:3}
\notag
\begin{aligned}
\varlimsup_{n \in \mathbb{N}^{'''}} \left\|(x_{n}, A_{n}x_{n}) - y_{n} - z \right\| &\ =\ \varlimsup_{n \in \mathbb{N}^{'''}} \left\|(x_{n}, A_{n}x_{n}) - (x_{n}, \theta) - (\theta, t) \right\|\\
                                                                                    &\ =\ \varlimsup_{n \in \mathbb{N}^{'''}} \left\|(x_{n} - x_{n} - \theta, A_{n}x_{n} - \theta - t) \right\|\\
                                                                                    &\ =\  \varlimsup_{n \in \mathbb{N}^{'''}} \left\|A_{n}x_{n} - t \right\| = 0,
\end{aligned}
\end{multline}
which proves that ${\lambda_{\mathbb{N}^{'}}[A_{n}, \theta] = 0}$.
\end{proof}
\begin{lemma}\label{L:wc0ua}
Any sequence of functionals $(f_{n})_{\mathbb{N}}$ from $X^{*}$ is lower uniformly $0-$adjusted with a null functional from $X^{*}$.
\end{lemma}
\begin{proof}
Set $Y = \mathbb{K}$ where $\mathbb{K}$ is the field of $X$ and apply the previous lemma.
\end{proof}

We will now explain how uniform adjustment generalizes perturbations of linear operators by small gap, norm, $q-$norm and by ${K_{2}}$ convergent operators. For that we need the following two definitions:

\begin{definition}[$\lambda-$adjustment of Subspaces]\label{D:as}
Let $X$ be a Banach space, $M$ and $P$ are two closed subspaces of $X$, ${M \neq\{\theta\}}$. Define two sequences of subspaces from $X$: ${(M_{n})_{\mathbb{N}^{'}}}$ and ${(P_{n})_{\mathbb{N}^{'}}}$ such that ${M_{n} = M}$ and ${P_{n} = P}$ for ${n \in \mathbb{N}^{'}}$. We say that subspace ${M}$ is lower $\lambda-adjusted$ with ${P}$ (equivalently -- ${P}$ is upper $\lambda-$adjusted with ${M}$) if ${(M_{n})_{\mathbb{N}^{'}}}$ is lower uniformly $\lambda-$adjusted with ${(P_{n})_{\mathbb{N}^{'}}}$.

The $\lambda-$adjustment between $M$ and $P$ is defined as
\[
\lambda[M, P]\ :=\ \lambda_{\mathbb{N}^{'}}[M_{n}, P_{n}].
\]
\end{definition}
\begin{definition}[$\lambda-$adjustment of Operators]\label{D:ao}
Let $X$ and $Y$ be Banach spaces, ${A, B \in \mathcal{C}(X,Y)}$ are two closed linear operators;  ${G_{A}}$ and ${G_{B}}$ are their respective graphs -- closed subspaces from the product space $X \times Y$. We say that operator ${A}$ is lower $\lambda-adjusted$ with ${B}$ (equivalently -- ${B}$ is upper $\lambda-$adjusted with ${A}$) if ${G_{A}}$ is lower $\lambda-$adjusted with ${G_{B}}$.

The $\lambda-$adjustment between $A$ and $B$ is defined as
\[
\lambda[A, B]\ :=\ \lambda[G_{A}, G_{B}].
\]
\end{definition}

\begin{theorem}\label{aoggnsk}
Let $X$ and $Y$ are two Banach spaces. Then the following propositions are true:
\begin{enumerate}
  \item Recall that for two closed operators from $\mathcal{C}(X,Y)$ their gap distance is defined as the gap distance of their graphs $G_{A}$ and $G_{B}$ in the product space $X \times Y$:
\[
\delta(A, B)\ :=\ \delta(G_{A}, G_{B});
\]
Then, for any two sequences $(A_{n})_{\mathbb{N}^{'}}$ and $(B_{n})_{\mathbb{N}^{'}}$ from $\mathcal{C}(X, Y)$
\begin{equation}\label{E:4}
\lambda_{\mathbb{N}^{'}}[A_{n}, B_{n}]\ \leq\ \varlimsup_{n \in \mathbb{N}^{'}} \delta(A_{n}, B_{n}).
\end{equation}

  \item Consider a linear closed operator ${A \in \mathcal{C}(X,Y)}$ and a sequence of continuous operators $(C_{n})_{\mathbb{N}^{'}}$ from $\mathcal{B}(X,Y)$ such that ${\lim_{n \in \mathbb{N}^{'}} \left\|C_{n}\right\| = 0}$. Then
\[
\lambda_{\mathbb{N}^{'}}[A + C_{n}, A] = \lambda_{\mathbb{N}^{'}}[A, A + C_{n}] = 0.
\]

  \item Recall that for ${A, C \in \mathcal{C}(X, Y)}$, operator $C$ is called $A-compact$ if the set ${\{Cx_{n} \mid n \in \mathbb{N}\}}$ is relatively compact in $Y$ for every sequence ${(x_{n})_{\mathbb{N}} \subset X}$ whenever the sequence of pairs ${(x_{n}, Ax_{n})_{\mathbb{N}}}$ is bounded in the product space ${X \times Y}$. Now, if $C$ is $A-$compact, then
\[
\lambda[A, A + C] = 0.
\]

  \item Recall that for every bounded subset ${M \subset X}$ its Hausdorff measure of non-compactness is defined as
\[
q(M) := \inf \{ e \in \mathbb{R} \mid M\ is\ covered\ with\ a\ finite\ number\ of\ balls\ radius\ e \};
\]
also recall that if $C$ is a continuous operator from $\mathcal{B}(X, Y)$, then its $q-$norm is defined as ${\left\|C\right\|_{q} := q(C(S))}$, where $S$ is a unit ball from $X$. Now, if ${A \in \mathcal{C}(X, Y)}$ and ${C \in \mathcal{B}(X, Y)}$, then
\[
\lambda[A, A + C] \leq \left\|C\right\|_{q}.
\]

  \item If ${A \in \mathcal{C}(X, Y)}$ and ${C \in \mathcal{K}(X, Y)}$, then
\[
\lambda[A, A + C] = \lambda[A + C, A] = 0.
\]

  \item Recall that a sequence of continuous operators ${(C_{n})_{\mathbb{N^{'}}} \subset \mathcal{BC}(X, Y)}$\linebreak[4] is said to $K_{2}-$approximate null operator if for every bounded sequence ${(x_{n})_{\mathbb{N^{'}}} \triangleleft (dom(C_{n}))_{\mathbb{N}^{'}}}$ the sequence ${(C_{n}x_{n})_{\mathbb{N^{'}}} \subset Y}$ is relatively compact. 
  
  Now, if ${A \in \mathcal{C}(X, Y)}$ and every ${C_{n}}$ is defined on the whole space $X$, i.e. ${(C_{n})_{\mathbb{N'}} \subset \mathcal{B}(X, Y)}$, and ${(C_{n})_{\mathbb{N'}}}$ $K_{2}-$approximates null operator, then  
\[  
\lambda_{\mathbb{N}^{'}}[A + C_{n}, A] = \lambda_{\mathbb{N}^{'}}[A, A + C_{n}] = 0.
\]
\end{enumerate}
\end{theorem}
\begin{proof}
Proposition 1 immediately follows from Lemma \ref{L:uawg}.\\

In order to prove proposition 2 recall a well known fact (\cite{kato}) that if\linebreak[4] ${\lim_{n \in \mathbb{N}^{'}} \left\|C_{n}\right\| = 0}$, then
\begin{equation}\label{E:5}
\lim_{n \in \mathbb{N}^{'}} \delta(A + C_{n}, A) = \lim_{n \in \mathbb{N}^{'}} \delta(A, A + C_{n}) = 0.
\end{equation}
Thus, combining formulas \eqref{E:5} and \eqref{E:4} we see that:
\[
\lambda_{\mathbb{N}^{'}}[A + C_{n}, A] = \lambda_{\mathbb{N}^{'}}[A, A + C_{n}] = 0.
\]

For proposition 3 consider a unit sequence ${(x_{n}, Ax_{n})_{\mathbb{N}^{'}} \in G_{A}}$. Since $C$ is $A-$compact, there exists a subsequence ${(x_{n})_{\mathbb{N}^{''}}}$ that converges to some $t \in Y$. Now put
\[(y_{n}) = (x_{n}, Ax_{n} + Cx_{n})_{\mathbb{N}^{''}} \subset G_{A + C},\ \ \ \ \ z = (\theta, -t) \in X \times Y.
\]
Then
\begin{multline}
\notag
\begin{aligned}
\lim_{n \in \mathbb{N}^{''}} \left\|(x_{n}, Ax_{n}) - y_{n} - z\right\| &\ =\ \lim_{n \in \mathbb{N}^{''}} \left\|(x_{n}, Ax_{n}) - (x_{n}, Ax_{n} + Cx_{n}) - (\theta, -t)\right\|\\
                                                                        &\ =\ \lim_{n \in \mathbb{N}^{''}} \left\|(x_{n} - x_{n} - \theta,\ Ax_{n} - Ax_{n} - Cx_{n} + t)\right\|\\
                                                                        &\ =\ \lim_{n \in \mathbb{N}^{''}} \left\|(\theta, -Cx_{n} + t)\right\|\\
                                                                        &\ =\ \lim_{n \in \mathbb{N}^{''}} \left\|Cx_{n} - t\right\| = 0,
\end{aligned}
\end{multline}
which proves proposition 3: ${\lambda[A, A + C] = 0}$.\\

For proposition 4 consider a unit sequence ${(x_{n}, Ax_{n})_{\mathbb{N}^{'}}}$ -- clearly ${(x_{n})_{\mathbb{N}^{'}}}$ belongs to a unit ball $S \subset X$. According to definition of ${\left\|C\right\|_{q}}$, for any ${\eta > 0}$ there exists a finite cover of $A(S)$ with the balls radius ${\left\|C\right\|_{q}} + \eta$. Therefore, we can find a subsequence ${\mathbb{N}^{''}}$ such that ${(Ax_{n})_{\mathbb{N}^{''}}}$ is contained within one of the balls with center $t \in Y$. Now put
\[
(y_{n}) = (x_{n}, Ax_{n} + Cx_{n})_{\mathbb{N}^{''}} \subset G_{A + C},\ \ \ \ \ z = (\theta, -t) \in X \times Y.
\]
Obviously
\begin{multline}
\notag
\begin{aligned}
\varlimsup_{n \in \mathbb{N}^{''}} \left\|(x_{n}, Ax_{n}) - y_{n} - z\right\| &\ =\ \varlimsup_{n \in \mathbb{N}^{''}} \left\|(x_{n}, Ax_{n}) - (x_{n}, Ax_{n} + Cx_{n}) - (\theta, -t)\right\|\\
                                                                              &\ =\ \varlimsup_{n \in \mathbb{N}^{''}} \left\|(x_{n} - x_{n} - \theta,\ Ax_{n} - Ax_{n} - Cx_{n} + t)\right\|\\
                                                                              &\ =\ \varlimsup_{n \in \mathbb{N}^{''}} \left\|(\theta, -Cx_{n} + t)\right\|\\
                                                                              &\ =\ \varlimsup_{n \in \mathbb{N}^{''}} \left\|Cx_{n} - t\right\| \leq {\left\|C\right\|_{q}} + \eta.
\end{aligned}
\end{multline}
Therefore, ${\lambda[A, A + C] \leq \left\|C\right\|_{q}}$.\\

Proposition 5 follows from proposition 4 once we notice that ${\left\|C\right\|_{q} = 0}$ for a compact operator $C$.\\

In order to prove proposition 6 consider a bounded sequence\linebreak[4] ${(x_{n}, Ax_{n})_{\mathbb{N}^{'}} \subset G_{A}}$. Obviously sequence ${(x_{n})_{\mathbb{N}^{'}} \subset X}$ is bounded, hence, since ${(C_{n})_{\mathbb{N}^{'}}}$\ $K_{2}-$approximates null operator, one can choose a subsequence ${(x_{n})_{\mathbb{N}^{''}}}$ such that ${(C_{n}x_{n})_{\mathbb{N}^{''}}}$ converges to some ${t \in Y}$. Now put 
\[
(y_{n})_{\mathbb{N}^{''}} = (x_{n}, Ax_{n} + C_{n}x_{n})_{\mathbb{N}^{''}} \triangleleft (G_{A + C_{n}})_{\mathbb{N}^{''}},\ \ \ \ z = (\theta, -t) \in X \times Y.
\]
Then
\begin{multline}
\notag
\begin{aligned}
\lim_{n \in \mathbb{N}^{''}} \left\|(x_{n}, Ax_{n}) - y_{n} - z\right\| &\ =\ \lim_{n \in \mathbb{N}^{''}} \left\|(x_{n}, Ax_{n}) - (x_{n}, Ax_{n} + C_{n}x_{n}) - (\theta, -t)\right\|\\
                                                                        &\ =\ \lim_{n \in \mathbb{N}^{''}} \left\|(x_{n} - x_{n} - \theta,\ Ax_{n} - Ax_{n} - C_{n}x_{n} + t)\right\|\\
                                                                        &\ =\ \lim_{n \in \mathbb{N}^{''}} \left\|(\theta, -C_{n}x_{n} + t)\right\|\\
                                                                        &\ =\ \lim_{n \in \mathbb{N}^{''}} \left\|C_{n}x_{n} - t\right\| = 0.
\end{aligned}
\end{multline}
Therefore, ${\lambda_{\mathbb{N}^{'}}[A, A + C_{n}] = 0}$. 

Now consider a bounded sequence ${(x_{n}, Ax_{n} +C_{n}x_{n})_{\mathbb{N}^{'}} \triangleleft (G_{A + C_{n}})_{\mathbb{N}^{''}}}$. Obviously sequence ${(x_{n})_{\mathbb{N}^{'}} \subset X}$ is bounded, hence, since ${(C_{n})_{\mathbb{N}^{'}}}$\ $K_{2}-$approximates null operator, one can choose a subsequence ${(x_{n})_{\mathbb{N}^{''}}}$ such that ${(C_{n}x_{n})_{\mathbb{N}^{''}}}$ converges to some ${t \in Y}$. Now put
\[(y_{n})_{\mathbb{N}^{''}} = (x_{n}, Ax_{n})_{\mathbb{N}^{''}} \subset G_{A},\ \ \ \ z = (\theta, t) \in X \times Y.
\]
Then
\begin{multline}
\notag
\begin{aligned}
\lim_{n \in \mathbb{N}^{''}} \left\|(x_{n}, Ax_{n} + C_{n}x_{n}) - y_{n} - z\right\| &\ =\ \lim_{n \in \mathbb{N}^{''}} \left\|(x_{n}, Ax_{n} + C_{n}x_{n}) - (x_{n}, Ax_{n}) - (\theta, t)\right\|\\
                                                                                     &\ =\ \lim_{n \in \mathbb{N}^{''}} \left\|(x_{n} - x_{n} - \theta,\ Ax_{n} + C_{n}x_{n} - Ax_{n} - t)\right\|\\
                                                                                     &\ =\ \lim_{n \in \mathbb{N}^{''}} \left\|(\theta, C_{n}x_{n} - t)\right\|\\
                                                                                     &\ =\ \lim_{n \in \mathbb{N}^{''}} \left\|C_{n}x_{n} - t\right\| = 0.
\end{aligned}
\end{multline}
Thus ${\lambda_{\mathbb{N}^{'}}[A + C_{n}, A] = 0}$ which proves proposition 6.
\end{proof}

While the previous theorem showed how the known concepts of convergence end up being particular cases of uniform adjustment, the next theorem provides means of constructing non-trivial uniformly adjusted operators that do not degenerate into any of the known cases.

\begin{theorem}\label{nua}
Let $X$ and $Y$ are two Banach spaces, ${(A_{n})_{\mathbb{N}^{'}} \subset \mathcal{C}(X,Y)}$ and ${(B_{n})_{\mathbb{N}^{'}} \subset \mathcal{C}(X,Y)}$. Then the following propositions are true:
\begin{enumerate}
  \item Suppose that ${\lambda_{\mathbb{N}^{'}}[A_{n}, B_{n}] = 0}$, ${C \in \mathcal{B}(X,Y)}$, ${(C_{n})_{\mathbb{N}^{'}} \subset \mathcal{B}(X,Y)}$ and\linebreak[4] ${(C_{n} - C)_{\mathbb{N}^{'}}}$ $K_{2}-$approximates the null operator ${\theta \in \mathcal{B}(X,Y)}$, then\linebreak[4] ${\lambda_{\mathbb{N}^{'}}[A_{n} + C_{n}, B_{n} + C] = 0}$ and  ${\lambda_{\mathbb{N}^{'}}[A_{n} + C, B_{n} + C_{n}] = 0}$.
  \item If ${C \in \mathcal{B}(X,Y)}$ and ${\lambda_{\mathbb{N}^{'}}[\theta, B_{n}] = 0}$, then ${\lambda_{\mathbb{N}^{'}}[C, C + B_{n}] = 0}$.
  \item If ${C \in \mathcal{B}(X,Y)}$ and ${\lambda_{\mathbb{N}^{'}}[B_{n}, \theta] = 0}$, then ${\lambda_{\mathbb{N}^{'}}[C + B_{n}, C] = 0}$.
  \item Suppose that $X$, $Y$, $Z$ are Hilbert spaces and ${K}$ is a compact operator from ${\mathcal{K}(X,Y)}$. Then for any bounded sequence of continuous operators ${(C_{n})_{\mathbb{N}^{'}} \subset \mathcal{B}(Y,Z)}$
\[
\lambda_{\mathbb{N}^{'}}[\theta, C_{n} \circ K] = 0.
\]
\end{enumerate}
\end{theorem}
\begin{proof}
  It is not hard to see that the sequence of operators ${(C_{n})_{\mathbb{N}^{'}}}$ is bounded in $\mathcal{B}(X,Y)$: otherwise there would be a bounded sequence of vectors ${(c_{n})_{\mathbb{N}^{'}} \subset X}$ such that ${\varlimsup_{n \in \mathbb{N}^{'}} \left\|C_{n}c_{n}\right\| = \infty}$; but since ${(C_{n})_{\mathbb{N}^{'}}}$\ $K_{2}-$approximates  ${C}$, the sequence ${(C_{n}c_{n} - Cc_{n})_{\mathbb{N}^{'}}}$ is relatively compact and therefore bounded; yet ${(Cc_{n})_{\mathbb{N}^{'}}}$ is bounded since $C$ is continuous, therefore ${(Cc_{n})_{\mathbb{N}^{'}}}$ is bounded which contradicts our assumption. Therefore, there exists a constant ${K}$ such that
\begin{equation}\label{ck2cb}
\varlimsup_{n \in \mathbb{N}^{'}} \left\|C_{n}\right\| \leq K < \infty.
\end{equation}
  Now, in order to prove proposition 1 consider a unit sequence 
\[
(x_{n}, A_{n}x_{n} + C_{n}x_{n})_{\mathbb{N}^{''}} \triangleleft (G_{A_{N} + C_{n}})_{\mathbb{N}^{''}}.
\]
  From \eqref{ck2cb} we can estimate the norms of ${A_{n}x_{n}}$ using the triangle inequality:
\begin{multline}\label{E:9}
\notag
\begin{aligned}
\varlimsup_{n \in \mathbb{N}^{''}} \left\|A_{n}x_{n}\right\|&\ =\ \varlimsup_{n \in \mathbb{N}^{'}} \left\|A_{n}x_{n} + C_{n}x_{n} - C_{n}x_{n}\right\|\\
                                                            &\ \leq\ \varlimsup_{n \in \mathbb{N}^{''}} \left\|A_{n}x_{n} + C_{n}x_{n}\right\| + \varlimsup_{n \in \mathbb{N}^{''}} \left\|C_{n}x_{n}\right\|\ \leq\ 1 + K\ <\ \infty.
\end{aligned}
\end{multline}
  Thus, setting ${\alpha_{n} = 1 / \left\|(x_{n}, A_{n}x_{n})\right\|}$ for ${n \in {\mathbb{N}^{''}}}$ we obtain
\[
\varliminf_{n \in \mathbb{N}^{''}} \alpha_{n} = \varliminf_{n \in \mathbb{N}^{''}} \frac{1}{\left\|(x_{n},A_{n}x_{n})\right\|} \geq \frac{1}{\max\{1, 1 + K\}} = \frac{1}{1 + K}.
\]
  Now suppose that ${\varlimsup_{n \in \mathbb{N}^{''}} \alpha_{n} = \infty}$. Then both ${(x_{n})_{\mathbb{N}^{''}}}$ and ${(A_{n}x_{n})_{\mathbb{N}^{''}}}$ converge to ${\theta}$; thus, according to \eqref{ck2cb}, ${(C_{n}x_{n})_{\mathbb{N}^{''}}}$ also converges to ${\theta}$ -- but this contradicts our choice of unit sequence ${(x_{n}, A_{n}x_{n} + C_{n}x_{n})_{\mathbb{N}^{''}}}$. Therefore, for some ${P < \infty}$
\begin{equation}\label{E:14}
P \geq \varlimsup_{n \in \mathbb{N}^{''}} \alpha_{n} \geq \varliminf_{n \in \mathbb{N}^{''}} \alpha_{n} \geq \frac{1}{1 + K}.
\end{equation}
  Now consider a unit sequence ${(\alpha_{n}x_{n}, A_{n}(\alpha_{n}x_{n}))_{\mathbb{N}^{''}} \triangleleft (G_{A_{N}})_{\mathbb{N}^{''}}}$ and a number ${\eta > 0}$, then put 
\begin{equation}\label{Ex:24}
\nu = \frac{\eta}{( 1 + \left\|C\right\| ) \times ( 1 + K )};
\end{equation}
  since ${\lambda_{\mathbb{N}^{'}}[A_{n}, B_{n}] = 0}$, there exists a subsequence ${(y_{n})_{\mathbb{N}^{'''}} \subset Y}$ and a vector ${(t,v) \in X \times Y}$ such that
\begin{multline}\label{E:6}
\varlimsup_{n \in \mathbb{N}^{'''}} \left\|(\alpha_{n}x_{n} - y_{n},\ A_{n}(\alpha_{n}x_{n}) - B_{n}y_{n}) - (t,v)\right\| \\
= \varlimsup_{n \in \mathbb{N}^{'''}} \left\|(\alpha_{n}x_{n} - y_{n} - t,\ A_{n}(\alpha_{n}x_{n}) - B_{n}y_{n} - v)\right\| \leq \nu.
\end{multline}
  Since $C$ is continuous, for ${n \in \mathbb{N}^{'''}}$
\begin{equation}\label{E:7}
\left\|C(\alpha_{n}x_{n}) - Cy_{n} - Ct \right\|\ \leq\ \left\|C\right\| \times \left\|\alpha_{n}x_{n} - Cy_{n} - t\right\|\ \leq\ \left\|C\right\| \times \nu.
\end{equation}
  Since ${(C_{n})_{\mathbb{N}^{'}}}$\ $K_{2}-$approximates  ${C}$ and ${(\alpha_{n}x_{n})_{\mathbb{N}^{'''}}}$ is bounded there exists a subsequence ${\mathbb{N}^{'4}}$ and a vector ${u \in Y}$ such that
\begin{equation}\label{E:8}
\varlimsup_{n \in \mathbb{N}^{'4}} \left\|C_{n}(\alpha_{n}x_{n}) - C(\alpha_{n}x_{n}) - u \right\| = 0.
\end{equation}
  Define
\begin{multline}
\notag
\begin{aligned}
  \Delta &\ =\ \varlimsup_{n \in \mathbb{N}^{'4}} \left\|(A_{n} + C_{n})(\alpha_{n}x_{n}) - (B_{n} + C_{n})y_{n} - v - Ct - u \right\|\\
         &\ =\ \varlimsup_{n \in \mathbb{N}^{'4}} \left\|A_{n}(\alpha_{n}x_{n}) + C_{n}(\alpha_{n}x_{n}) - B_{n}y_{n} - C_{n}y_{n} - v - Ct - u \right\|\\
         &\ =\ \varlimsup_{n \in \mathbb{N}^{'4}} \left\| (A_{n}(\alpha_{n}x_{n}) - B_{n}y_{n} - v) + (C(\alpha_{n}x_{n}) - Cy_{n} - Ct) + (C_{n}(\alpha_{n}x_{n}) - C(\alpha_{n}x_{n}) - u)\right\|.
\end{aligned}
\end{multline}
  Combining \eqref{E:6}, \eqref{E:7} and \eqref{E:8} together and using the triangle inequality we get
\begin{multline}\label{E:10}
\begin{aligned}
  \Delta &\ \leq\ \varlimsup_{n \in \mathbb{N}^{'4}} \left\| (A_{n}(\alpha_{n}x_{n}) - B_{n}y_{n} - v \right\|\\
         &\ \ \ \ \ \ \ \ \ \ \ \ \ \ \ \ \ \ \ \ \ \ \ \ \ \ \ \ +\ \varlimsup_{n \in \mathbb{N}^{'4}} \left\| C(\alpha_{n}x_{n}) - Cy_{n} - Ct \right\|\\
         &\ \ \ \ \ \ \ \ \ \ \ \ \ \ \ \ \ \ \ \ \ \ \ \ \ \ \ \ +\ \varlimsup_{n \in \mathbb{N}^{'4}} \left\| C_{n}(\alpha_{n}x_{n}) - C(\alpha_{n}x_{n}) - u \right\|\\
         &\ \leq\ \nu + \left\|C\right\| \times \nu + 0\ =\ \nu \times (1 + \left\|C\right\|).
\end{aligned}
\end{multline}
  Taking into account \eqref{E:14}, we see that the sequence of numbers ${(\alpha_{n})_{\mathbb{N}^{''}}}$ is contained within a compact interval ${[\frac{1}{1 + K},\ P] \subset \mathbb{R}}$, so we can choose a convergent subsequence ${(\alpha_{n})_{\mathbb{N}^{'5}}}$ such that
\[
\lim _{n \in \mathbb{N}^{'5}} \alpha_{n} \rightarrow \alpha\ \in\ [\frac{1}{1 + K},\ P],
\]
  then rewrite \eqref{E:10} (using \eqref{Ex:24}):
\begin{multline}\label{E:11}
\begin{aligned}
\frac{1}{\alpha} \times \Delta &\ =\ \varlimsup_{n \in \mathbb{N}^{'5}} \left\|(A_{n} + C_{n})(x_{n})\ -\ (B_{n} + C_{n})(\frac{1}{\alpha}y_{n})\ -\ \frac{1}{\alpha}(v + Ct + u) \right\|\\
                               &\ \leq\ \frac{1}{\alpha} \times \nu \times ( 1 + \left\|C\right\| )\\
                               &\ \leq\ \nu  \times ( 1 + K ) \times ( 1 + \left\|C\right\| )\\
                               &\ =\ \frac{\eta}{( 1 + \left\|C\right\| ) \times ( 1 + K )}  \times ( 1 + K ) \times ( 1 + \left\|C\right\| )\ =\ \eta.
\end{aligned}
\end{multline}
  We can also rewrite the first component from \eqref{E:6}:
\begin{multline}\label{E:12}
\begin{aligned}
\varlimsup_{n \in \mathbb{N}^{'''}} \left\|x_{n}\ -\ \frac{1}{\alpha}y_{n}\ -\ \frac{1}{\alpha} t \right\| &\ \leq\  \frac{\nu}{\alpha}\\
                                                                                                           &\ \leq\ \nu \times (1 + K)\\
                                                                                                           &\ = \frac{\eta}{( 1 + \left\|C\right\| ) \times ( 1 + K )} \times (1 + K)\ \leq\ \eta.
\end{aligned}
\end{multline}
  Therefore, from \eqref{E:11} and \eqref{E:12} we conclude that given a unit sequence
\[
(x_{n},\ A_{n}x_{n} + C_{n}x_{n})_{\mathbb{N}^{''}} \triangleleft (G_{A_{n} + C_{n}})_{\mathbb{N}^{''}}
\]
  and a number ${\eta > 0}$ we found a subsequence
\[
(\frac{1}{\alpha} y_{n},\ (B_{n} + C_{n})(\frac{1}{\alpha} y_{n}))_{\mathbb{N}^{'5}} \triangleleft (G_{B_{n} + C_{n}})_{\mathbb{N}^{'5}}
\]
  and a vector
\[
(z, k)\ =\ (\frac{1}{\alpha}t,\ \frac{1}{\alpha}(v + Ct + u))\ \in\ X \times Y
\]
  such that
\begin{multline}
\notag
\begin{aligned}
\varlimsup_{n \in \mathbb{N}^{'5}} & \left\|(x_{n},\ A_{n}x_{n} + C_{n}x_{n})\ -\ (\frac{1}{\alpha}y_{n},\ (B_{n} + C_{n})(\frac{1}{\alpha}y_{n})) - (z,\ k) \right\|\\
                                      & =\ \varlimsup_{n \in \mathbb{N}^{'5}} \left\|(x_{n} - \frac{1}{\alpha}y_{n} - \frac{1}{\alpha}t,\ (A_{n} + C_{n})x_{n} - (B_{n} + C_{n})(\frac{1}{\alpha}y_{n}) - \frac{1}{\alpha}(v + Ct + u))\right\|\\
                                      & \leq \ \max \{\eta, \frac{1}{\alpha} \times \Delta\}\\
                                      & \leq \ \max \{\eta, \eta\} = \eta,
\end{aligned}
\end{multline}
  which proves that ${\lambda_{\mathbb{N}^{'}}[A_{n} + C_{n}, B_{n} + C] = 0}$.

We skip the proof of the second equality as it is similar to the proof of the first one.\\

Proposition 2 directly follows from Proposition 1 since a sequence of operators all equal to the same operator $K_{2}$ approximates that operator. Proposition 3 is established by exactly the same argument.\\

Now let us prove proposition 4. Let ${\sup\{\left\|C_{n}\right\|\} = T < \infty}$. Consider a unit sequence ${(x_{n}, \theta)_{\mathbb{N}^{'}} \subset X \times Y}$ and a number ${\eta > 0}$. According to \cite{sedaev} there exists a closed subspace ${L \subset X}$ such that ${\dim X/L < \infty}$ and ${\left\|K|_{L}\right\| \leq \frac{\eta}{T}}$. Decompose Hilbert space ${X = L^{\bot} \oplus L}$ and the sequence ${x_{n} = u_{n} + v_{n}}$ where ${u_{n} \in L^{\bot}, v_{n} \in L}$ and ${\left\|u_{n}\right\| \leq 1}$, ${\left\|v_{n}\right\| \leq 1}$ for ${n \in \mathbb{N}^{'}}$. Since
\[
\dim L^{\bot} = \dim X/L < \infty,
\]
we can choose a convergent subsequence ${(u_{n})_{\mathbb{N}^{''}} \rightarrow u \in L^{\bot}}$. Now set
\[
(y_{n})_{\mathbb{N}^{''}} = (v_{n}, (C_{n} \circ K)v_{n})_{\mathbb{N}^{''}} \triangleleft (G_{C_{n} \circ K})_{\mathbb{N}^{''}},\ \ \ \ z = (u, \theta) \in X \times Y.
\]
Then
\begin{multline}
\notag
\begin{aligned}
\varlimsup_{n \in \mathbb{N}^{''}} \left\|((x_{n}, \theta) - y_{n} - z)\right\| &\ =\ \varlimsup_{n \in \mathbb{N}^{''}} \left\|((x_{n}, \theta)\ -\ (v_{n}, (C_{n} \circ K)v_{n})\ -\ (u, \theta))\right\|\\
                                                                                &\ =\ \varlimsup_{n \in \mathbb{N}^{''}} \left\|((x_{n} - v_{n} - u,\ (C_{n} \circ K)v_{n} - \theta)\right\|\\
                                                                                &\ =\ \varlimsup_{n \in \mathbb{N}^{''}} \left\|((u_{n} - u,\ (C_{n} \circ K)v_{n})\right\|\\
                                                                                &\ =\ \varlimsup_{n \in \mathbb{N}^{''}} \max \{\left\|u_{n} - u\right\|,\ \left\|(C_{n} \circ (K|_{L}))v_{n}\right\|\}\\
                                                                                &\ \leq\ \varlimsup_{n \in \mathbb{N}^{''}} \max \{0,\ \left\|(C_{n} \circ (K|_{L}))\right\| \times \left\|v_{n}\right\|\}\\
                                                                                &\ \leq\ \varlimsup_{n \in \mathbb{N}^{''}} \left\|C_{n}\right\| \times \left\|(K|_{L})\right\|\ \leq\ T \times \frac{\eta}{T}\ =\ \eta.
\end{aligned}
\end{multline}
Therefore, ${\lambda_{\mathbb{N}^{'}}[\theta, C_{n} \circ K] = 0}$.
\end{proof}

Using the last result we now construct a sequence of operators from ${L_{2}[0,1]}$ that are $0-$adjusted with the null operator $\theta$ but that do not $K_{2}-$approximate $\theta$. Let ${(\varphi_{n})_{\mathbb{N}} \subset L_{2}[0,1]}$ be an orthonormal basis, and ${K(x,y) \in L_{2}([0,1] \times [0,1])}$. Consider a sequence of operators ${(S_{n})_{\mathbb{N}} \subset \mathcal{B}(L_{2}[0,1],L_{2}[0,1])}$ defined like this:
\[
(S_{n}f)(x) = \varphi_{n}(x) \times \int_{0}^{1}K(x,y) \times f(y)dy.
\]
Each ${S_{n}}$ is a composition of the multiplication operator ${\varphi_{n}(x) \times f(x)}$ with the compact  integral operator ${\int_{0}^{1}K(x,y) \times f(y)dy}$; therefore, ${\lambda_{\mathbb{N}}[\theta, S_{n}] = 0}$ according to proposition 3 from Theorem \ref{nua}. But because ${(\varphi_{n})_{\mathbb{N}}  \subset L_{2}[0,1]}$ is an orthonormal basis, ${(S_{n})_{\mathbb{N}}}$ does not $K_{2}-$approximate ${\theta}$ for any non-trivial kernel ${K(x,y)}$.
\\

The following theorem shows that if a pair of sequences of operators degenerates into just a pair of continuous operators, then close adjustment between them becomes equivalent to their difference having a small q-norm:
\begin{theorem}
Suppose $Y$ and $Y$ are two Banach spaces and ${A, C \in \mathcal{B}(X,Y)}$. Then
\begin{enumerate}
  \item ${\left\|C\right\|_{q} \leq 2 \times \max \{1,\ \left\|A\right\|\} \times (\left\|C\right\| + \left\|A\right\| + 1) \times \lambda[A, A + C]}$.

  \item C is a compact operator, i.e. ${\left\|C\right\|_{q} = 0}$ if and only if ${\lambda[A, A + C] = 0}$.
\end{enumerate}
\end{theorem}
\begin{proof}
In order to prove proposition 1 consider a sequence of non-null vectors ${(x_{n})_{\mathbb{N}}}$ from a unit ball of $X$, a sequence of real numbers
\[
(\alpha_{n})_{\mathbb{N}} = (\max\{\left\|x_{n}\right\|, \left\|Ax_{n}\right\|\})_{\mathbb{N}}
\]
and a sequence of vectors
\begin{equation}\label{E:19}
(x^{'}_{n})_{\mathbb{N}} = (\frac{1}{\alpha_{n}}x_{n})_{\mathbb{N}}.
\end{equation}
Obviously ${(x^{'}_{n}, Ax^{'}_{n})_{\mathbb{N}} \subset X \times Y}$ is a unit sequence. Therefore, for any ${\eta > 0}$ one can find a subsequence ${(y_{n}, (A + C)y_{n})_{\mathbb{N^{'}}} \triangleleft G_{A + C}}$ and some ${(z, t) \in X \times Y}$ such that
\[
\varlimsup_{n \in \mathbb{N}^{'}} \left\|(x^{'}_{n}, Ax^{'}_{n})\ -\ (y_{n}, (A + C)y_{n})\ -\ (z, t) \right\|\ \leq\ \lambda[A, A + C] + \eta.
\]
Let us rewrite this inequality for each component:
\begin{equation}\label{E:15}
\begin{aligned}
\varlimsup_{n \in \mathbb{N}^{'}} \left\|x^{'}_{n} - y_{n} - z \right\|\ \leq\ \lambda[A, A + C] + \eta,
\end{aligned}
\end{equation}
\begin{equation}\label{E:16}
\begin{aligned}
\varlimsup_{n \in \mathbb{N}^{'}} \left\|Ax^{'}_{n} - (A + C)y_{n} - t \right\|\ \leq\ \lambda[A, A + C] + \eta.
\end{aligned}
\end{equation}
Recall that both $A$ and $C$ are continuous; thus application of first $A$ then $C$ to the left side of \eqref{E:15} results in
\begin{equation}\label{E:17}
\begin{aligned}
\varlimsup_{n \in \mathbb{N}^{'}} \left\|Ax^{'}_{n} - Ay_{n} - Az \right\| \leq \left\|A\right\| \times (\lambda[A, A + C] + \eta),
\end{aligned}
\end{equation}
\begin{equation}\label{E:18}
\begin{aligned}
\varlimsup_{n \in \mathbb{N}^{'}} \left\|Cx^{'}_{n} - Cy_{n} - Cz \right\| \leq \left\|C\right\| \times (\lambda[A, A + C] + \eta).
\end{aligned}
\end{equation}
Now adding \eqref{E:18} and \eqref{E:17}, subtracting \eqref{E:16}, using the triangle inequality, and, finally, applying \eqref{E:15} we observe
\begin{multline}
\notag
\begin{aligned}
\varlimsup_{n \in \mathbb{N}^{'}} & \left\|Cx^{'}_{n} - Cz - Az + t\right\|\\
                                  & =\ \varlimsup_{n \in \mathbb{N}^{'}} \left\|(Cx^{'}_{n} - Cy_{n} - Cz) + (Ax^{'}_{n} - Ay_{n} - Az) - (Ax^{'}_{n} - Ay_{n} - Cy_{n} - t)\right\|\\
                                  & \leq\ \varlimsup_{n \in \mathbb{N}^{'}} (\left\|(Cx^{'}_{n} - Cy_{n} - Cz)\right\| + \left\|(Ax^{'}_{n} - Ay_{n} - Az)\right\| + \left\|(Ax^{'}_{n} - Ay_{n} - Cy_{n} - t)\right\|)\\
                                  & \leq\ (\left\|C\right\| + \left\|A\right\|) \times \varlimsup_{n \in \mathbb{N}^{'}} \left\|x^{'}_{n} - y_{n} - z\right\|\ +\ \left\|(Ax^{'}_{n} - Ay_{n} - Cy_{n} - t)\right\|)\\
                                  & \leq\ (\left\|C\right\| + \left\|A\right\|) \times \varlimsup_{n \in \mathbb{N}^{'}} \left\|x^{'}_{n} - y_{n} - z\right\|\ +\ (\lambda[A, A + C] + \eta)\\
                                  & \leq\ (\left\|C\right\| + \left\|A\right\|) \times (\lambda[A, A + C] + \eta)\ +\ (\lambda[A, A + C] + \eta)\\
                                  & \leq\ (\left\|C\right\| + \left\|A\right\| + 1) \times (\lambda[A, A + C] + \eta).
\end{aligned}
\end{multline}
From this and from \eqref{E:19} we conclude that for any unit sequence ${(x_{n})_{\mathbb{N}}}$ we have found a vector ${Cz - Az + t \in Y}$ such that for some subsequence ${\mathbb{N'}}$ 
\begin{multline}
\notag
\varlimsup_{n \in \mathbb{N}^{'}} \left\|Cx_{n} - \alpha_{n}Cz - \alpha_{n}Az + \alpha_{n}t\right\|\\
\leq |\alpha_{n}| \times (\left\|C\right\| + \left\|A\right\| + 1) \times (\lambda[A, A + C] + \eta).
\end{multline}
Obviously ${|\alpha_{n}| \leq \max\{1, \left\|A\right\|\}}$; thus for some ${\mathbb{N}^{''}}$ the sequence ${(\alpha_{n})_{\mathbb{N}^{''}}}$ converges to some ${\alpha \in [0, \max\{1, \left\|A\right\|\}]}$. Therefore, the previous inequality can be rewritten:
\begin{multline}
\notag
\varlimsup_{n \in \mathbb{N}^{'}} \left\|Cx_{n} - \alpha Cz - \alpha Az + \alpha t \right\|\\
\leq \max\{1, \left\|A\right\|\} \times (\left\|C\right\| + \left\|A\right\| + 1) \times (\lambda[A, A + C] + \eta).
\end{multline}
Denoting
\[
q(\eta)\ :=\ \max\{1, \left\|A\right\|\} \times (\left\|C\right\| + \left\|A\right\| + 1) \times (\lambda[A, A + C] + \eta),
\]
we see that for every non-null sequence ${(x_{n})_{\mathbb{N}}}$ from a unit ball of $X$ we can find a subsequence ${(x_{n})_{\mathbb{N^{''}}}}$ and a vector ${u = \alpha Cz + \alpha Az - \alpha t \in Y}$ such that ${\left\|Cx_{n} -u\right\| \leq q(\eta)}$  for all ${n \in \mathbb{N^{''}}}$. 

Now it is easy to prove that if $S$ is a unit ball from $X$, then Hausdorff measure of non-compactness ${\mu(C(S)) \leq 2 \times q(\eta)}$. Suppose this is not true -- then one can choose a sequence ${(x_{n})_{\mathbb{N}} \subset S}$ such that ${\left\|Cx_{n} - Cx_{m}\right\| > 2 \times q(\eta)}$ for any two different $n$ and $m$. But we have just proved that there exists ${\mathbb{N}^{'}}$ such that ${\left\|Cx_{n} - u\right\| \leq q(\eta)}$ for all ${n \in \mathbb{N}^{'}}$; thus, applying the triangle inequality, we see that for different ${n, m \in \mathbb{N}^{'}}$  
\begin{multline}
\notag
\begin{aligned}
\left\|Cx_{n} - Cx_{m}\right\| & \ =\ \left\|Cx_{n} - u - (Cx_{m} - u)\right\|\\
                               &\ \leq\ \left\|Cx_{n} - u\right\| + \left\|(Cx_{m} - u)\right\| \leq 2 \times q(\eta)
\end{aligned}
\end{multline}
which contradicts our choice of ${(Cx_{n})_{\mathbb{N}}}$. Therefore, ${\mu(C(S)) = \left\|C\right\|_{q} \leq 2 \times q(\eta)}$. Since ${\eta}$ can be arbitrarily small, we conclude that proposition 1 is true.\\

In order to prove proposition 2 recall that, according to Theorem \ref{aoggnsk} proposition 4, ${\lambda[A, A + C] \leq \left\|C\right\|_{q}}$. Therefore, if ${\left\|C\right\|_{q} = 0}$, then ${\lambda[A, A + C] = 0}$. According to proposition 1
\[
\left\|C\right\|_{q}\ \leq\ 2 \times \max\{1, \left\|A\right\|\} \times (\left\|C\right\| + \left\|A\right\| + 1) \times \lambda[A, A + C].
\]
Therefore, if ${\lambda[A, A + C] = 0}$, then ${\left\|C\right\|_{q} = 0}$.
\end{proof}
\subsection[Examples of Instability]{Examples of Instability}\label{S:eoi}

Since uniform adjustment is a non-trivial extension of both gap topology and compact perturbations, it is interesting to find to which extent the nice stability properties of the latter two remain true for the former. Indeed, we shall see that in general stability breaks.

For example, the small gap theorem states that if ${\delta(M, P) < 1}$ and\linebreak[4] ${\dim P < \infty}$ then ${\dim M \leq \dim P}$, i.e. that finite dimension is stable under gap perturbations smaller than $1$. Yet we shall see immediately that in ${l_{\infty}}$ (the space of limited sequences of real numbers with $sup-$norm) a similar statement about ${\lambda_{\mathbb{N}^{'}}[M_{n}, P_{n}]}$ is not true -- smaller than $1$ lower uniform $\lambda-$adjustment may still increase finite dimensions up to infinity. Then we will observe that although ${\delta(X, P) = 1}$ for any proper closed subspace ${P \subset X}$, a similar statement about ${\lambda_{\mathbb{N}^{'}}[l_{\infty}, P_{n}]}$ is not true -- some subspaces ${P_{n}}$ may be significantly 'smaller' than ${l_{\infty}}$, yet ${\lambda_{\mathbb{N}^{'}}[l_{\infty}, P_{n}] = 0}$.

\begin{example}[Non-Trivial Lower $1/2-$Adjustment with Subspace ${\{\theta\}}$]\label{Ex:1}
There exists a sequence of closed subspaces ${(M_{n})_{\mathbb{N}}}$ from ${l_{\infty}}$ such that\linebreak[4] ${\lambda_{\mathbb{N}^{'}}[M_{n}, \{\theta\}] = 1/2}$ and ${\dim M_{n} \rightarrow \infty}$. This is in contrast to the fact that there can be no subspace ${M \subset l_{\infty}}$ with ${\dim M > 0}$ and ${\delta (M, \{\theta\}) < 1}$.
\end{example}
\begin{proof}
Partition ${\mathbb{N}}$ into an infinite sequence of disjoint non-empty subsets:
\[
\mathbb{N} = \bigcup_{n \in \mathbb{N}} B_{n},\ \ B_{n} \neq \emptyset\ for\ all\ n \in \mathbb{N},\ \ B_{i} \bigcap B_{j} = \emptyset\ for\ i\ \neq j.
\]
Define closed subspaces ${M_{n} \subset l_{\infty}}$ consisting of sequences ${(\alpha^{n}_{i})_{i \in \mathbb{N}}}$ which components with indices from ${\mathbb{N} \backslash B{n}}$ are null:
\[
M_{n} := \{ (\alpha^{n}_{i})_{i \in \mathbb{N}} \in l_{\infty} \mid i \notin B_{n} \Rightarrow \alpha^{n}_{i} = 0 \}.
\]
Let us show that ${\lambda_{\mathbb{N}}[M_{n}, {\theta}] = 1/2}$. Consider a unit sequence of vectors\linebreak[4] ${(x_{n})_{\mathbb{N}^{'}} \triangleleft (M_{n})_{\mathbb{N}^{'}}}$, ${x_{n} = (\alpha^{n}_{i})_{i \in \mathbb{N}}}$. Define a vector ${z = (\beta_{i})_{i \in \mathbb{N}}}$:
\begin{itemize}
  \item if ${i \notin \bigcup_{n \in \mathbb{N^{'}}} B_{n}}$ then set ${\beta_{i} = 0}$;
  \item otherwise $i$ belongs to a single ${B_{n}}$, so set ${\beta_{i} = 1/2 \times \alpha^{n}_{i}}$.
\end{itemize}
From the construction of $z$ we can calculate components ${\gamma^{n}_{i}}$ of the difference ${(\gamma^{n}_{i})_{i \in \mathbb{N}} = y_{n} = x_{n} - z}$ for all ${n \in \mathbb{N}^{'}}$:
\begin{itemize}
  \item if ${i \notin \bigcup_{n \in \mathbb{N^{'}}} B_{n}}$ then ${\gamma^{n}_{i} = \alpha^{n}_{i} - \beta_{i} = 0 - 0 = 0}$;
  \item otherwise $i$ belongs to a single ${B_{n}}$, so
\[
\gamma^{n}_{i}\ =\ \alpha^{n}_{i} - \beta_{i}\ =\ \alpha^{n}_{i} - \frac{1}{2} \times \alpha^{n}_{i}\ =\ \frac{1}{2} \times \alpha^{n}_{i}.
\]
\end{itemize}
From this, taking into account that each ${x_{n} = (\alpha^{n}_{i})_{i \in \mathbb{N}}}$ is a unit vector, we can calculate the norms of ${(\gamma^{n}_{i})_{i \in \mathbb{N}} = y_{n}}$ for all ${n \in \mathbb{N}}$:
\[
\left\|y_{n}\right\|\ =\ \sup_{i \in \mathbb{N}} |\gamma^{n}_{i}|\ =\ \sup_{i \in \mathbb{N}} |\frac{1}{2} \alpha^{n}_{i}|\ =\ \frac{1}{2} \times \sup_{i \in \mathbb{N}} |\alpha^{n}_{i}|\ =\ \frac{1}{2}.
\]
Hence
\[
\varlimsup_{n \in \mathbb{N}^{'}} \left\|x_{n} - z\right\|\ =\ \varlimsup_{n \in \mathbb{N}^{'}} \left\|y_{n}\right\|\ =\  \frac{1}{2};
\]
therefore ${\lambda_{\mathbb{N}^{'}}[M_{n}, \{\theta\}] = 1/2 < 1}$. However, obviously ${card\ B_{n} = \dim M_{n}}$ where ${card\ B_{n}}$ is the number of elements in ${B_{n}}$; therefore, if we partition ${\mathbb{N}}$ so that ${card\ B_{n} \rightarrow \infty}$, then ${\dim M_{n} \rightarrow \infty}$.
\end{proof}

\begin{example}[Non-Trivial Upper $0-$Adjustment in $l_{\infty}$]\label{Ex:2}
There exists a sequence of closed subspaces ${(P_{n})_{\mathbb{N}}}$ from ${l_{\infty}}$ such that ${\lambda_{\mathbb{N}}[l_{\infty}, P_{n}] = 0}$ and ${\dim X/P_{n} \rightarrow \infty}$. This is in contrast to the fact that ${\delta (X, P) = 1}$ for any closed subspace ${P \subset l_{\infty}}$ with ${\dim X/P > 0}$.
\end{example}
\begin{proof}
Partition ${\mathbb{N}}$ into an infinite sequence of disjoint non-empty subsets:
\[
\mathbb{N} = \bigcup_{n \in \mathbb{N}} B_{n},\ \ B_{n} \neq \emptyset\ for\ all\ n \in \mathbb{N},\ \ B_{i} \cap B_{j} = \emptyset\ for\ i\ \neq j.
\]
Define closed subspaces ${P_{n} \subset l_{\infty}}$ consisting of sequences ${(\alpha^{n}_{i})_{i \in \mathbb{N}}}$ which components which indices from ${B_{n}}$ are null:
\[
P_{n} := \{ (\alpha^{n}_{i})_{i \in \mathbb{N}} \in l_{\infty} \mid i \in B_{n} \Rightarrow \alpha^{n}_{i} = 0 \}.
\]
First we prove that ${\lambda_{\mathbb{N}}[l_{\infty}, P_{n}] = 0}$. Consider a unit sequence of vectors\linebreak[4] ${(x_{n})_{\mathbb{N}^{'}} \subset l_{\infty}}$, ${x_{n} = (\alpha^{n}_{i})_{i \in \mathbb{N}}}$. Define a vector ${z = (\beta_{i})_{\mathbb{N}}}$:
\begin{itemize}
  \item if ${i \notin \bigcup_{n \notin \mathbb{N^{'}}} B_{n}}$ then set ${\beta_{i} = 0}$;
  \item otherwise $i$ belongs to a single ${B_{n}}$, so set ${\beta_{i} = \alpha^{n}_{i}}$. 
\end{itemize}
Vector $z$ belongs to ${l_{\infty}}$ since ${(x_{n})_{\mathbb{N}^{'}}}$ is a unit sequence from ${l_{\infty}}$. Also, by our construction each vector ${y_{n} = x_{n} - z}$ belongs to ${P_{n}}$. Obviously ${\left\|x_{n} - y_{n} - z\right\| = 0}$ for all ${n \in \mathbb{N^{'}}}$. Therefore we conclude that ${\lambda_{\mathbb{N}}[l_{\infty}, P_{n}] = 0}$. However,\linebreak[4] ${\dim X/P_{n} = card\ B_{n}}$, thus if we choose ${B_{n}}$ so that ${card\ B_{n} \rightarrow \infty}$ then\linebreak[4] ${\dim X/P_{n} \rightarrow \infty}$. This is in contrast to the fact that there can be no closed subspace ${P \subset l_{\infty}}$ with ${\dim X/P > 0}$ and ${\delta(X, P) < 1}$.
\end{proof}

It is a well known fact (\cite{kato}) that the gap between a pair of closed subspaces ${(M, P)}$ from a Banach space $X$ is equal to the reverse gap between their annihilators ${(P^{\bot}, M^{\bot})}$ from the dual space ${X^{*}}$:
\[
	\delta(M, P) = \delta(P^{\bot}, M^{\bot}).
\]
However, the next example shows that a similar statement is not true for uniform $\lambda-$adjustment.
\begin{example}[Subspaces and their Annihilators are not Equally Adjusted]\label{Ex:3}
The sequence of subspaces ${(M_{n})_{\mathbb{N}}}$ from Example \ref{Ex:1} is such that
\[
\frac{1}{2}\ =\ \lambda_{\mathbb{N}}[M_{n}, \{\theta\}]\ \neq\ \lambda_{\mathbb{N}} [\{\theta\}^{\bot}, M_{n}^{\bot}]\ =\ \lambda_{\mathbb{N}} [l_{\infty}^{*}, M_{n}^{\bot}]\ =\ 1.
\]
\end{example}
\begin{proof}
It has already been proved that ${\lambda_{\mathbb{N}}[M_{n}, \{\theta\}] = 1/2}$. Now let us estimate ${\lambda_{\mathbb{N}} [l_{\infty}^{*}, M_{n}^{\bot}]}$. Let ${\mathbb{N}^{'}}$ be a subsequence of ${\mathbb{N}}$ such that each ${n_{k} \in \mathbb{N}^{'}}$ belongs to ${B_{k}}$, and ${x_{n_{k}}}$ is a unite vector from ${M_{n_{k}}}$ which only non-null coordinate is a unit coordinate ${n_{k}}$. Define a unit sequence of functionals from ${l_{\infty}^{*}}$ as ${e_{n_{k}}(x) = \omega_{n_{k}} \in \mathbb{K}}$ where ${\mathbb{K}}$ is the field of $X$ and ${\omega_{n_{k}}}$ is the ${n_{k}}$-th coordinate of $x$ projected on ${x_{n_{k}}}$. Obviously each ${e_{n_{k}} \in l_{\infty}^{*}}$ is well defined in this manner.

Now suppose that, contrary to our assertion,  ${\lambda_{\mathbb{N}} [l_{\infty}^{*}, M_{n}^{\bot}] = \nu < 1}$. Then, for any ${\eta}$ from the open interval ${(\nu, 1)}$ there exists a subsequence of functionals ${(f_{n})_{\mathbb{N}^{''}} \triangleleft (M_{n}^{\bot})_{\mathbb{N}^{''}}}$ and a functional $g$ from ${l_{\infty}}$ such that for all ${n \in \mathbb{N}^{''}}$
\[
\left\| e_{n} - f_{n} - g \right\| = \eta  < 1.
\]
Consider how each functional ${h_{n} = e_{n} - f_{n} - g}$ acts on the previously chosen unit vectors ${x_{n}}$ for all ${n \in \mathbb{N}^{''}}$:
\[
h_{n}x_{n}\ =\ (e_{n} - f_{n} - g)x_{n}\ =\ e_{n}x_{n} - 0 - gx_{n}\ =\ 1 - gx_{n}\ \leq\ \eta\ <\ 1.
\]
Therefore, the following inequality holds for all ${n \in \mathbb{N}^{''}}$:
\[
1 - \eta\ \leq\ gx_{n}.
\]
Now choose a natural number ${K > \left\|g\right\|/(1 - \eta)}$ and add together the first $K$ of the above inequalities:
\[
K \times (1 - \eta)\ \leq\ \sum_{i = 1}^{K} (gx_{n_{i}})\ =\ g(\sum_{i = 1}^{K} x_{n_{i}})\ \leq\  \left\|g\right\| \times (\sum_{i = 1}^{K} x_{n_{i}}).
\]
Recall that vector ${\sum_{i = 1}^{K} x_{n_{i}}}$ is from ${l_{\infty}}$ and that all its components ${x_{n_{i}}}$ are unit vectors from ${l_{\infty}}$ each having a single non-null coordinate at position $n_{i}$, therefore all coordinates up to $K$ of ${\sum_{i = 1}^{K} x_{n_{i}}}$ are either $0$ or $1$, so the norm of ${\sum_{i = 1}^{K} x_{n_{i}}}$ is 1 and the previous inequality transforms to: 
\[
K \times (1 - \eta)\ \leq\  \left\|g\right\| \times (\sum_{i = 1}^{K} x_{n_{i}})\ =\ \left\|g\right\|,
\]
\[
K \leq \frac{\left\|g\right\|}{1 - \eta}.
\]
However, the last inequality contradicts our choice of ${K > \left\|g\right\|/(1 - \eta)}$. Therefore, we shall conclude that 
\[
\lambda_{\mathbb{N}} [l_{\infty}^{*}, M_{n}^{\bot}] = 1.
\]
\end{proof}
Recall that an operator ${A \in \mathcal{C}(X, Y)}$ with a closed range ${R(A) \subset Y}$ is called \emph{semi--Fredholm} if one of its \emph{defect numbers} ${\alpha(A) = \dim Ker(A)}$\linebreak[4] or  ${\beta(A) = \dim Y/R(A)}$ is finite; it is called ${Fredholm}$ if both defect numbers ${\alpha(A)}$ and ${\beta(A)}$ are finite. According to the known results from \cite{krein_krasnoselskii_milman,gohberg_markus,gohberg_krein,gohberg_markus_feldman,neubauer1,neubauer2,kato,goldenstein_gohberg_markus,goldenstein_markus,sedaev,nussbaum,sadovskii}, semi--Fredholm and Fredholm operators remain such, with the stable defect numbers, under perturbations by small gap, norm or $q-$norm. However, our next example shows that Fredholm and semi--Fredholm operators are not stable under upper uniform $0-$adjustment.

\begin{example}[Upper Uniformly $0-$Adjustment Breaks Fredholm Properties]\label{Ex:4}
Let ${I: l_{\infty} \rightarrow l_{\infty}}$ is an identity operator; obviously $I$ is Fredholm with ${\alpha(I)}$ $=$ ${\beta(A) = 0}$. There exist a sequence of closed operators${(A_{n})_{\mathbb{N}} \subset \mathcal{C}(X,Y)}$ such that ${\lambda_{\mathbb{N}} [I, A_{n}] = 0}$, yet none of ${A_{n}}$ is either Fredholm or semi--Fredholm.
\end{example}
\begin{proof}
Consider a sequence of subspaces ${(P_{n})_{\mathbb{N}}}$ from ${l_{\infty}}$ built in the previous Example \ref{Ex:2} -- we have proved there that ${\lambda_{\mathbb{N}} [l_{\infty}, P_{n}] = 0}$. Therefore, if ${(A_{n})_{\mathbb{N}}}$ is a sequence of closed operators ${A_{n}: l_{\infty} \rightarrow l_{\infty}}$ each being constant on ${P_{n}}$, then it is easy to show that  ${\lambda_{\mathbb{N}} [I, A_{n}] = 0}$. But we have already seen that, depending on the choice of partitioning ${\mathbb{N} = \bigcup B_{n}}$ it is possible to have values of ${\dim l_{\infty}/P_{n}}$ being any arbitrary large natural numbers or even infinity. Therefore, one can easily construct ${A_{n}}$ to have any defect numbers of any finite or infinite values, and/or having images to be closed or not closed (when ${\dim l_{\infty}/P_{n} = \infty}$), thus making each or some of ${A_{n}}$ not a Fredholm, not an upper semi--Fredholm or not a lower semi--Fredholm operator.
\end{proof}

In the final example of this subsection we show that uniform adjustment of a sequence of subspaces does not mean that individual subspaces from that sequence are adjusted. This situation differs from perturbations by small gap, norm or $q-$norm where convergence of sequences of subspaces or operators means that individual members of the sequence are also becoming close to the target. 

\begin{example}[Uniform Adjustment and Adjustment Differ]\label{Ex:5}
There exists a sequence of closed subspaces ${(P_{n})_{\mathbb{N}}}$ from ${l_{\infty}}$ such that ${\lambda_{\mathbb{N}}[l_{\infty}, P_{n}] = 0}$, yet for all ${n \in \mathbb{N}}$ ${\lambda[l_{\infty}, P_{n}] \geq 1/2}$.
\end{example}
\begin{proof}
Consider the sequence ${(P_{n})_{\mathbb{N}}}$ from ${l_{\infty}}$ built in Example \ref{Ex:2}. As we have noted before, if ${\mathbb{N}}$ is partitioned into infinite subsets ${B_{n}}$, then ${\dim l_{\infty}/P_{n} = \infty}$. Therefore, applying proposition 4 from the Small Uniform Adjustment Theorem, we get ${\lambda[l_{\infty}, P_{n}] \geq 1/2}$ for any ${n \in \mathbb{N}}$. However, it has been proved in Example \ref{Ex:2} that ${\lambda_{\mathbb{N}}[l_{\infty}, P_{n}] = 0}$. 
\end{proof}

\subsection[Lower Semi--Fredholm Stability of Subspaces]{Lower Semi--Fredholm Stability of Subspaces}\label{S:sp} 

Our previous Example \ref{Ex:4} had shown that there is no stability of any of the Fredholm properties under upper uniform $\lambda-$adjustment even if ${\lambda = 0}$. However, our next theorem shows that lower semi--Fredholm pair of closed subspaces is stable under lower $\lambda-$adjustment for small $\lambda$.

\ 

In order to formulate that theorem we shall consider some preliminary concepts. Let $M$, $N$ be two closed subspaces in a Banach space $X$ such that\linebreak[3] ${M + N = \overline{M + N}}$, ${M+N = (M \cap N) \oplus S}$. Consider a natural mapping
\[
\Phi\ :\ \Pi\ =\ M \cap S\ \times\ M \cap N\ \times\ N \cap S\ \rightarrow\ M + N\\
\]
\[
\Phi\ :\ (u, t, v)\ \rightarrow\ u + t + v.
\]
Define a complete norm on ${\Pi}$
\[
\left\|(u, t, v)\right\| = \max \{\left\|u\right\|,\left\|t\right\|,\left\|v\right\|\}.
\]
Then, since
\[
\overline{M + N}\ =\ M + N\ =\ (M \cap N) \oplus S\ =\ M \cap S\ \oplus\ M \cap N\ \oplus\ N \cap S,
\]
it is clear that $\Phi$ is a continuous bijection from the Banach space ${\Pi}$ onto the Banach space ${M + N}$. Therefore, according to the open mapping theorem operator ${\Phi^{-1}}$ is a continuous operator. Let us denote ${\varphi_{S}(M, N)\ :=\ \left\|\Phi^{-1}\right\|}$.

\ 

Also recall that a pair of closed subspaces $(M,N)$ from a Banach space $X$ is called \emph{lower semi--Fredholm} if its \emph{lower defect number} ${\alpha(M,N) := \dim(M \cap N)}$ is finite; it is called \emph{upper semi--Fredholm} if its \emph{upper defect number}\linebreak[4] ${\beta(M,N) := \dim X/(M+N)}$ is finite; it is called a \emph{Fredholm pair} if both defect numbers ${\alpha(M,N)}$ and ${\beta(M,N)}$ are finite, an \emph{index} of a (semi--)Fredholm pair is defined as
\[
ind(M,N)\ :=\ \alpha(M,N) - \beta(M,N). 
\]
Note that for a lower semi--Fredholm pair there always exist many closed subspaces $S$ such that ${M+N = (M \cap N) \oplus S}$ if ${\alpha(M,N) = \dim(M \cap N) < \infty}$. Therefore, denote
\[
\varphi(M, N)\ :=\ \inf \{ \varphi_{S}(M, N)\ \mid\ M+N = (M \cap N) \oplus S\}.
\]
We are now ready to formulate our first stability theorem: 
\begin{theorem}[Lower Semi--Fredholm Pairs are Stable]\label{lsfps}
Let $M$, $N$ be two closed subspaces in a Banach space $X$ such that ${M + N = \overline{M + N}}$. Let ${(M_{n})_{\mathbb{N}^{'}}}$, ${(N_{n})_{\mathbb{N}^{'}}}$ are two sequences of closed subspaces from $X$, and set
\[
\lambda_{M} = \lambda_{\mathbb{N}^{'}}[M_{n}, M],\ \ \ \lambda_{N} = \lambda_{\mathbb{N}^{'}}[N_{n}, N];
\]
then the following propositions are true
\begin{enumerate}
  \item Suppose that pair ${(M,N)}$ is lower semi--Fredholm and define a real number
\[
\omega\ =\ 2 \times (\min(\lambda_{M}, \lambda_{N}) + \varphi(M, N) \times (\lambda_{M} + \lambda_{N})).
\]  
If ${\omega < 1/2}$, then for large enough ${n \in \mathbb{N}^{'}}$ pairs ${(M_{n},N_{n})}$ are also lower semi--Fredholm and
\[
\varlimsup_{n \in\mathbb{N}^{'}} \alpha (M_{n}, N_{n})\ <\ \infty.
\]
  \item Suppose that ${M+N = (M \cap N) \oplus S}$ and define a real number
\[
\omega_{S}(M) = 2 \times (\lambda_{M} + \varphi_{S}(M, N) \times (\lambda_{M} + \lambda_{N})).
\]  
If ${(H_{n})_{\mathbb{N^{''}}}}$ is a sequence of closed subspaces, ${H_{n} \subset M_{n}}$ for all ${n \in \mathbb{N}^{''}}$ and ${\delta(H_{n}, N_{n}) \rightarrow 0}$, then
\[
\lambda_{\mathbb{N}^{''}}[H_{n}, M \cap N]\ \leq\ \omega_{S}(M).
\]
\end{enumerate}
\end{theorem}
\begin{proof}
Proposition 2 will be the key to the proof of the proposition 1, so we start with 2 first. Consider ${\eta > 0}$ and a unit sequence ${(x_{n})_{\mathbb{N}^{''}} \triangleleft (H_{n})_{\mathbb{N}^{''}}}$. According to the definition of $\lambda-$adjustment, our goal is to find a sequence ${(y_{n})_{\mathbb{N}^{'''}} \subset M \cap N}$ and a vector ${z \in X}$ such that
\begin{equation}\label{E:50}
\varlimsup_{\mathbb{N}^{'''}} \left\|x_{n} - y_{n} - z\right\|\ \leq\ \omega_{S}(M) + \eta.
\end{equation}
In order to do that first choose a number $\epsilon$ from an open interval
\begin{equation}\label{E:57}
\epsilon \in (0,\ \frac{\eta}{4 \times (1 + 2 \times \varphi_{S})}).
\end{equation}
Then, since ${\delta(H_{n}, N_{n}) \rightarrow 0}$ we can find ${(t_{n})_{\mathbb{N}^{''}} \triangleleft (N_{n})_{\mathbb{N}^{''}}}$ such that
\begin{equation}\label{E:51}
\lim_{\mathbb{N}^{''}} \left\|x_{n} - t_{n}\right\| \rightarrow 0.
\end{equation}
According to definition of $\lambda-$adjustment, there exist subsequences ${(u_{n})_{\mathbb{N}^{'''}} \subset M}$, ${(v_{n})_{\mathbb{N}^{'''}} \subset N}$ and vectors ${u, v \in X}$ such that
\begin{multline}\label{E:52}
\begin{aligned}
& \varlimsup_{\mathbb{N}^{'''}} \left\|x_{n} - u_{n} - u\right\|\ \leq\ \lambda_{M} + \epsilon,\\
& \varlimsup_{\mathbb{N}^{'''}} \left\|t_{n} - v_{n} - v\right\|\ \leq\ \lambda_{N} + \epsilon.
\end{aligned}
\end{multline}
Now choose a number $\gamma$ from an open interval 
\begin{equation}\label{E:56}
\gamma \in (0,\ \frac{\eta}{2 \times (2 + \varphi_{S})}).
\end{equation}
Then, from \eqref{E:52} for large enough ${m \in \mathbb{N}^{'''}}$ we obtain
\[
\left\|x_{m} - u_{m} - u\right\|\ \leq\ \lambda_{M} + \epsilon + \gamma;
\]
therefore, applying the triangle inequality, for the same large enough $m$ and $n$
\begin{multline}\label{E:55}
\begin{aligned}
\varlimsup_{n \in \mathbb{N}^{'''}} \left\|x_{n} - u_{n} - x_{m} + u_{m}\right\| &\ =\  \varlimsup_{n \in \mathbb{N}^{'''}} \left\|(x_{n} - u_{n} - u) - (x_{m} - u_{m} - u)\right\|\\
                                             &\ \leq\ \varlimsup_{n \in \mathbb{N}^{'''}} \left\|x_{n} - u_{n} - u\right\| + \left\|x_{m} - u_{m} - u\right\|\\
                                             &\ \leq\ \lambda_{M} + \epsilon + \gamma\ +\ \lambda_{M} + \epsilon + \gamma\\
                                             &\ =\ 2 \times (\lambda_{M} + \epsilon + \gamma).
\end{aligned}
\end{multline}
At the same time, combining \eqref{E:51} and \eqref{E:52} and again using the triangle inequality observe
\begin{multline}
\notag
\begin{aligned}
\varlimsup_{\mathbb{N}^{'''}} \left\|u_{n} - v_{n} + u - v\right\| &\ =\ \varlimsup_{\mathbb{N}^{'''}} \left\|u_{n} - v_{n} + u - v - (x_{n} - t_{n}) + (x_{n} - t_{n})\right\|\\
                                                                  &\ \leq \  \varlimsup_{\mathbb{N}^{'''}} \left\|u_{n} - v_{n} + u - v - (x_{n} - t_{n})\right\| + \varlimsup_{\mathbb{N}^{'''}} \left\|(x_{n} - t_{n})\right\|\\
                                                                  &\ \leq \  \varlimsup_{\mathbb{N}^{'''}} \left\|(t_{n} - v_{n} - v) - (x_{n} - u_{n} - u)\right\| + 0\\
                                                                  &\ \leq \ \varlimsup_{\mathbb{N}^{'''}} \left\|t_{n} - v_{n} - v\right\| + \varlimsup_{\mathbb{N}^{'''}} \left\|x_{n} - u_{n} - u\right\|\\
                                                                  &\ \leq \ \lambda_{M} + \epsilon\ +\ \lambda_{N} + \epsilon\ =\ \lambda_{M} + \lambda_{N} + 2 \times \epsilon.
\end{aligned}
\end{multline}
Thus, for large enough ${m \in \mathbb{N}^{'''}}$
\[
\left\|u_{m} - v_{m} + u - v\right\|\ \leq\ \lambda_{M} + \lambda_{N} + 2 \times \epsilon + \gamma;
\]
therefore, applying the triangle inequality, for the same $m$
\begin{multline}\label{E:53}
\begin{aligned}
\varlimsup_{n \in \mathbb{N}^{'''}} \left\|u_{n} - v_{n} - u_{m} + v_{m}\right\| &\ =\  \varlimsup_{n \in \mathbb{N}^{'''}} \left\|(u_{n} - v_{n} + u - v) - (u_{m} - v_{m} + u - v)\right\|\\
                                             &\ \leq\ \varlimsup_{n \in \mathbb{N}^{'''}} \left\|u_{n} - v_{n} + u - v\right\| + \left\|u_{m} - v_{m} + u - v\right\|\\
                                             &\ \leq\ (\lambda_{M} + \lambda_{N} + 2 \times \epsilon)\ +\ (\lambda_{M} + \lambda_{N} + 2 \times \epsilon + \gamma)\\
                                             &\ =\ 2 \times (\lambda_{M} + \lambda_{N} + 2 \times \epsilon) + \gamma.
\end{aligned}
\end{multline}
Since ${u_{n} \in M}$  and ${v_{n} \in N}$ for ${n \in \mathbb{N}^{'''}}$, we can decompose vectors ${u_{n}}$ and ${v_{n}}$:
\begin{multline}
\notag
\begin{aligned}
& u_{n} = a_{n} + b_{n}:\ \ a_{n} \in M \cap S,\ b_{n} \in M \cap N,\\
& v_{n} = d_{n} + e_{n}:\ \ d_{n} \in N \cap S,\ e_{n} \in M \cap N.
\end{aligned}
\end{multline}
Therefore, we can represent
\begin{multline}
\notag
\begin{aligned}
u_{n} - v_{n} - u_{m} + v_{m} &\ =\ a_{n} + b_{n} - d_{n} - e_{n} - a_{m} - b_{m} + d_{m} + e_{m}\\
                              &\ =\ (a_{n} - a_{m})\ +\ (b_{n} - e_{n} - b_{m} + e_{m})\ -\ (d_{n} - d_{m}).
\end{aligned}
\end{multline}
Note that ${a_{n} - a_{m} \in M \cap S}$,\ ${b_{n} - e_{n} - b_{m} + e_{m} \in M \cap N}$\ and\linebreak[4] ${d_{n} - d_{m} \in N \cap S}$. Therefore, taking into account \eqref{E:53}, we obtain for large enough ${m \in \mathbb{N}^{'''}}$:
\begin{multline}\label{E:54}
\begin{aligned}
\varlimsup_{n \in \mathbb{N}^{'''}} \left\|a_{n} - a_{m}\right\| &\ \leq\ \varlimsup_{n \in \mathbb{N}^{'''}} \varphi_{S} \times \left\|u_{n} - v_{n} - u_{m} + v_{m}\right\|\\
                                                                &\ =\ \varphi_{S} \times \varlimsup_{n \in \mathbb{N}^{'''}} \left\|u_{n} - v_{n} - u_{m} + v_{m}\right\|\\
                                                                &\ \leq\ \varphi_{S} \times (2 \times (\lambda_{M} + \lambda_{N} + 2 \times \epsilon) + \gamma).
\end{aligned}
\end{multline}
We are now ready to prove that the sequence ${(y_{n})_{\mathbb{N}^{'''}}}$ and the vector ${z \in X}$ we are looking for can be chosen by setting ${y_{n} = b_{n}}$ for all ${n \in \mathbb{N}^{'''}}$ and ${z = x_{m} - b_{m}}$ for some large enough ${m \in \mathbb{N}^{'''}}$. Indeed, let us estimate using the triangle inequality and taking into account inequalities \eqref{E:55} and \eqref{E:54} together with estimates for $\epsilon$ from \eqref{E:57} and $\gamma$ from \eqref{E:56}:
\begin{multline}
\notag
\begin{aligned}
\varlimsup_{n \in \mathbb{N}^{'''}} \left\|x_{n} - y_{n} -z\right\| &\ =\ \varlimsup_{n \in \mathbb{N}^{'''}} \left\|x_{n} - b_{n} - ( x_{m} - b_{m})\right\|\\
                                                                    &\ =\ \varlimsup_{n \in \mathbb{N}^{'''}} \left\|x_{n} - b_{n} - (x_{m} - b_{m}) - (a_{n} - a_{m}) + (a_{n} - a_{m})\right\|\\
                                                                    &\ \leq\ \varlimsup_{n \in \mathbb{N}^{'''}} \left\|x_{n} - b_{n} - (x_{m} - b_{m}) - (a_{n} - a_{m})\right\| + \varlimsup_{n \in \mathbb{N}^{'''}} \left\|(a_{n} - a_{m})\right\|\\
                                                                    &\ =\ \varlimsup_{n \in \mathbb{N}^{'''}} \left\|x_{n} - a_{n} - b_{n} - (x_{m} - a_{m} - b_{m})\right\| + \varlimsup_{n \in \mathbb{N}^{'''}} \left\|(a_{n} - a_{m})\right\|\\
                                                                    &\ =\ \varlimsup_{n \in \mathbb{N}^{'''}} \left\|x_{n} - u_{n} - (x_{m} - u_{m})\right\| + \varlimsup_{n \in \mathbb{N}^{'''}} \left\|(a_{n} - a_{m})\right\|\\
                                                                    &\ \leq\ 2 \times (\lambda_{M} + \epsilon + \gamma)\ +\ \varphi_{S} \times (2 \times (\lambda_{M} + \lambda_{N} + 2 \times \epsilon) + \gamma)\\
                                                                    &\ =\ 2 \times (\lambda_{M} + \varphi_{S} \times (\lambda_{M} + \lambda_{N}))\\
                                                                    &\ \ \ \ \ \ \ \ \ \ \ \ \ + 2 \times (1 + 2 \times \varphi_{S}) \times \epsilon\\
                                                                    &\ \ \ \ \ \ \ \ \ \ \ \ \ + (2 + \varphi_{S}) \times \gamma\\
                                                                    &\ <\ 2 \times (\lambda_{M} + \varphi_{S} \times (\lambda_{M} + \lambda_{N})) + \frac{\eta}{2} + \frac{\eta}{2}\\
                                                                    &\ =\ \omega_{S}(M) + \eta. 
\end{aligned}
\end{multline}
This concludes the proof of the proposition 2.\\

In order to prove proposition 1 let us first assume that ${\lambda_{M} \leq \lambda_{N}}$ and that at least one of ${\lambda_{M}}$ or ${\lambda_{N}}$ is not null. Then, since ${\omega < 1/2}$, a simple arithmetic calculation shows that
\[
\varphi(M, N) < \frac{\frac{1}{4} - \lambda_{M}}{\lambda_{M} +\lambda_{N}}
\]
Therefore, according to our definition of ${\varphi(M, N)}$, and since ${\alpha(M,N) < \infty}$, we can decompose ${M+N = (M \cap N) \oplus S}$ so that ${\varphi_{S}(M,N)}$ satisfies the same inequality
\[
\varphi_{S}(M, N)\ <\ \frac{\frac{1}{4} - \lambda_{M}}{\lambda_{M} +\lambda_{N}}
\]
so that simple reverse arithmetic calculation leads to the inequality
\begin{equation}\label{E:59}
\omega_{S}(M)\ =\ 2 \times (\lambda_{M} + \varphi_{S}(M, N) \times (\lambda_{M} + \lambda_{N}))\ <\ \frac{1}{2}. 
\end{equation}
In case when both ${\lambda_{M}}$ and ${\lambda_{N}}$ are equal to $0$,  choose any decomposition\linebreak[4] ${M+N = (M \cap N) \oplus S}$ -- in this case the calculation yields ${\omega_{S}(M) = 0 < 1/2}$. Thus, we can always choose decomposition ${M+N = (M \cap N) \oplus S}$ so that \eqref{E:59} is true.
 
Now we are ready to establish that ${M_{n} + N_{n} = \overline{M_{n} + N_{n}}}$ for large enough ${n \in \mathbb{N}^{'}}$. Assuming the opposite, we can find a subsequence ${\mathbb{N}^{''}}$ such that for all ${n \in \mathbb{N}^{''}}$
\[
M_{n} + N_{n} \neq \overline{M_{n} + N_{n}}. 
\]
Then, according to the well known fact from (\cite{kato}, theorem 4.19, p. 226), for all ${n \in \mathbb{N}^{''}}$ there would be an infinite-dimensional closed subspace ${H_{n} \subset M_{n}}$ such that ${\delta(H_{n}, N_{n}) < 1/n}$, in other words ${\delta(H_{n}, N_{n}) \rightarrow \infty}$. However, due to \eqref{E:59} we may apply proposition 2 thus establishing that ${\lambda_{\mathbb{N}^{''}}[H_{n}, M \cap N]\ \leq\ \omega_{S}(M) < 1/2}$. Then, according to the Small Uniform Adjustment Theorem applied to the sequence ${(H_{n})_{\mathbb{N}^{''}}}$ and the finite-dimensional space ${M \cap N}$ it follows that ${\dim H_{n} < \infty}$ for large enough ${n \in \mathbb{N}^{''}}$ which contradicts our choice of ${H_{n}}$. Thus we conclude that ${M_{n} + N_{N} = \overline{M_{n} + N_{n}}}$ for large enough ${n \in  \mathbb{N}^{'}}$.

Now consider subspaces ${H_{n} = M_{n} \cap N_{n}}$ for all ${n \in \mathbb{N}^{'}}$. Obviously ${H_{n} \subset M_{n}}$ and ${\delta(H_{n}, N_{n}) = 0}$ for all ${n \in \mathbb{N}^{'}}$. Therefore, we can apply proposition 2 thus establishing ${\lambda_{\mathbb{N}^{''}}[H_{n}, M \cap N]\ \leq\ \omega_{S}(M) < 1/2}$. Then, according to the Small Uniform Adjustment Theorem applied to the sequence ${(H_{n})_{\mathbb{N}^{''}}}$ and the finite-dimensional space ${M \cap N}$ it follows that there exists ${L \in \mathbb{N}}$ such that\linebreak[4] ${\dim H_{n} < L}$ for large enough ${n \in \mathbb{N}^{'}}$. Therefore
\[
\varlimsup_{n \in \mathbb{N}^{'}} \alpha(M_{n},N_{n})\ =\ \varlimsup_{n \in \mathbb{N}^{'}} \dim M_{n} \cap N_{n}\ =\ \varlimsup_{n \in \mathbb{N}^{'}} \dim H_{n}\ <\ L\ <\ \infty.
\]
Thus we have established that for large enough ${n \in \mathbb{N}}$ pairs ${(M_{n},N_{n})}$ are lower semi--Fredholm. This concludes our proof of the proposition 1 in case when ${\lambda_{M} \leq \lambda_{N}}$. In case when ${\lambda_{N} < \lambda_{M}}$ the same proof remains valid by swapping $M$ and $N$.
\end{proof}

\subsection[Lower Stability of Index of Continuous Operators]{Lower Stability of Index of Continuous Operators}\label{soioclo}
It is easy to see that index of Fredholm pair of subspaces is not stable under lower uniform  $\lambda-$adjustment -- any two pairs of subspaces in a finite-dimensional space are uniformly $0-$adjusted yet their indices may obviously differ. However, for linear operators their index is stable under perturbation by lower uniformly $0-$adjusted continuous operators. In order to prove that we first consider the following theorem that establishes the structure of the continuous operators lower uniformly $\lambda-$adjusted with the null operator. 

\begin{theorem}[The Structure of Lower Uniform $\lambda-$Adjustment]\label{solua}
Let $X$ and $Y$ be two Banach spaces, ${(A_{n})_{\mathbb{N}^{'}} \subset \mathcal{C}(X,Y)}$. Then the following propositions are true:
\begin{enumerate}
  \item If ${\lambda_{\mathbb{N}^{'}}[A_{n}, \theta] < 1/2}$, then for each ${\epsilon \in (0,\ 1/2 - \lambda_{\mathbb{N}^{'}}[A_{n}, \theta])}$ there exists a closed subspace ${Y_{\epsilon} \subset Y}$ such that ${\dim Y/Y_{\epsilon} < \infty}$ and for large enough $n$
${A_{n} \in \mathcal{BC}(X, Y)}$ with the norm estimate
\[
\left\|A_{n} \mid_{A_{n}^{-1}(Y_{\epsilon})}\right\|\ <\ 2 \times \lambda_{\mathbb{N}^{'}}[A_{n}, \theta] + \epsilon.
\]
  \item If ${\lambda_{\mathbb{N}^{'}}[A_{n}, \theta] = 0}$ and sequence ${(x_{n}, A_{n}x_{n})_{\mathbb{N}^{''}} \subset X \times Y}$ is bounded, then sequence ${(A_{n}x_{n})_{\mathbb{N}^{''}} \subset Y}$ is relatively compact.
  \item If ${\lambda_{\mathbb{N}^{'}}[A_{n}, \theta] = 0}$, then for large enough ${n \in \mathbb{N}^{'}}$ each ${A_{n}}$ can be decomposed into ${A_{n} = C_{n} + F_{n}}$ where both ${C_{n}, F_{n} \in \mathcal{B}(dom(A_{n}), Y)}$, the sequence ${(C_{n})_{\mathbb{N}^{'}}}$ ${K_{2}-}$approximates $\theta$, and ${\varlimsup_{\mathbb{N}^{'}} \dim F_{n}(X) < \infty}$.
  \item If ${\lambda_{\mathbb{N}^{'}}[A_{n}, \theta] = 0}$ and ${(A_{n})_{\mathbb{N}^{'}} \subset \mathcal{C}(X,Y)}$, then ${\left\|A_{n}\right\|_{q} \rightarrow 0}$.
\end{enumerate}
\end{theorem}
\begin{proof}
Suppose that we have established the norm estimate from proposition 1, as well as the finite dimension estimate for large enough $n$. Then ${A_{n}}$ is a finite-dimensional extension of a continuous operator defined on a closed subspace of ${dom(A_{n})}$ and thus is a continuous operator itself, i.e. ${A_{n} \in \mathcal{BC}(X, Y)}$. Now let us prove the norm and dimension estimates. 

Suppose to the contrary that the norm estimate is not true for large enough $n$. Then choose any ${\epsilon \in (0,\ 1/2 - \lambda_{\mathbb{N}^{'}}[A_{n}, \theta])}$ and construct inductively a unit sequence of vectors ${(y_{n})_{\mathbb{N}^{''}} \subset Y}$ and a sequence ${(Y_{n})_{\mathbb{N}^{''}}}$ of closed subspaces of $Y$:
\begin{itemize}
  \item According to our assumption there exists a unit vector ${y \in Y}$ such that for some ${n_{1} \in \mathbb{N}}$
\[
\left\|A_{n_{1}}^{-1}y\right\|\ \leq\ 2 \times \lambda_{\mathbb{N}^{'}}[A_{n}, \theta] + \epsilon;
\]
Using Hahn-Banach Theorem find a unit functional ${f \in Y^{*}}$ such that ${fy = 1}$. Now put ${y_{n_{1}} = y}$, ${Y_{n_{1}} = Ker(f)}$; obviously ${dist(y_{n_{1}}, Y_{n_{1}}) = 1}$.
  \item Suppose that for some ${k > 1}$ we have built unit vectors ${y_{n_{1}},...,y_{n_{k}}}$ and closed subspaces ${Y_{n_{1}} \supset ... \supset Y_{n_{k}}}$ such that
\[
y_{n_{i+1}} \in Y_{n_{i}},\ \ {Y_{n_{i}} = sp(\{y_{n_{i+1}}\}) \oplus Y_{n_{i+1}}}
\]
  and ${dist(y_{n_{i}}, Y_{n_{i}}) = 1}$ for all ${i=1,...,k}$. Obviously ${\dim Y/Y_{n_{k}} = k < \infty}$; therefore, according to our assumption there exists a unit vector ${y \in Y_{n_{k}}}$ and a number ${m > n_{k}}$ such that
\[
\left\|A_{m}^{-1}y\right\|\ \leq\ 2 \times \lambda_{\mathbb{N}^{'}}[A_{n}, \theta] + \epsilon;
\]
Now put ${n_{k+1} = m}$, ${y_{n_{k+1}} = y}$, then applying Hahn-Banach Theorem\linebreak[4] choose a unit functional ${f \in Y_{n_{k}}^{*}}$ such that ${fy_{n_{k+1}} = 1}$ and set ${Y_{n_{k+1}} = Ker(f)}$. Obviously ${Y_{n_{k}} = y_{n_{k+1}} \oplus Y_{n_{k+1}}}$ and ${dist(y_{n_{k+1}}, Y_{n_{k+1}}) = 1}$. 
\end{itemize}
Now define a sequence of numbers ${\mathbb{N}^{''} = \{n_{1},...,n_{k},...\}}$. From the construction it is clear that ${Y_{n} \supset Y_{m}}$ and ${dist(y_{n}, Y_{m}) = 1}$ whenever ${n > m}$. Define a sequence of vectors ${(x_{n})_{\mathbb{N}^{''}}}$ such that ${x_{n} = A_{n}^{-1}y_{n}}$ for all ${n \in \mathbb{N}^{''}}$. Obviously ${(x_{n}, A_{n}x_{n})_{\mathbb{N}^{''}} \subset X \times Y}$ is a unit sequence since ${\left\|x_{n}\right\| \leq 2 \times \lambda_{\mathbb{N}^{'}}[A_{n}, \theta] + \epsilon < 1}$ and ${\left\|A_{n}x_{n}\right\| = \left\|y_{n}\right\| = 1}$ for all ${n \in \mathbb{N}^{''}}$. Therefore, for any number\linebreak[4] ${\nu \in (0,\ 1/2 - \lambda_{\mathbb{N}^{'}}[A_{n}, \theta])}$ there exists a sequence ${(t_{n})_{\mathbb{N}^{'''}} \subset X}$ and a vector\linebreak[4] ${(u,v) \in X \times Y}$ such that
\[
\varlimsup_{\mathbb{N}^{'''}}\left\|(x_{n},y_{n}) - (t_{n},\theta) - (u,v)\right\|\ \leq\  \lambda_{\mathbb{N}^{'}}[A_{n}, \theta] + \nu.
\]
Projecting onto Y obtain
\[
\varlimsup_{\mathbb{N}^{'''}}\left\|y_{n} - v\right\|\ \leq\ \lambda_{\mathbb{N}^{'}}[A_{n}, \theta] + \nu;
\]
Therefore, if ${\gamma \in (0,\ 1/2 - \lambda_{\mathbb{N}^{'}}[A_{n}, \theta] - \nu)}$, then applying the triangle inequality for large enough ${n > m}$ and taking into account the choice of $\nu$ and $\gamma$ obtain
\begin{multline}
\notag
\begin{aligned}
\left\|y_{n} - y_{m}\right\| & \ =\ \left\|y_{n} - v - y_{m} + v\right\|\\
                             &\ \leq\ \left\|y_{n} - v\right\| + \left\|y_{m} - v\right\|\\
                             &\ \leq\ \lambda_{\mathbb{N}^{'}}[A_{n}, \theta] + \nu + \gamma + \lambda_{\mathbb{N}^{'}}[A_{n}, \theta] + \nu + \gamma\\
                             &\ =\ 2 \times (\lambda_{\mathbb{N}^{'}}[A_{n}, \theta] + \nu + \gamma)\ <\ 1.
\end{aligned}
\end{multline}
However, according to our construction ${dist(y_{n}, Y_{m}) = 1}$ for any ${n >m}$; hence, since ${y_{m} \in Y_{m} \subset Y_{n}}$ we have
\[
1 \leq \left\| y_{n} - y_{m}\right\|.
\]
Now, combining the last two inequalities we come to a contradiction
\[
1 \leq \left\| y_{n} - y_{m}\right\| < 1.
\]
This concludes the proof of the proposition 1.\\

In order to prove proposition 2 recall that since ${\lambda_{\mathbb{N}^{'}}[A_{n}, \theta] = 0}$, for a unit sequence ${(x_{n}, A_{n}x_{n})_{\mathbb{N}^{''}} \subset X \times Y}$ there exists a subsequence ${(y_{1,n})_{\mathbb{N}^{'''}} \subset X}$ and a vector ${(u_{1},v_{1}) \in X \times Y}$ such that
\[
\varlimsup_{\mathbb{N}^{'''}}\left\|(x_{n},A_{n}x_{n}) - (y_{1,n},\theta) - (u_{1},v_{1})\right\|\ \leq\  2^{-3}.
\]
Projecting onto $Y$ we see that for large enough ${n \in \mathbb{N}^{'''}}$
\[
\left\|A_{n}x_{n} - v_{1}\right\|\ \leq\ 2^{-3} + 2^{-3}\ =\ 2^{-2};
\]
therefore, for large enough ${n, m \in \mathbb{N}^{'''}}$
\begin{multline}
\notag
\begin{aligned}
\left\|A_{n}x_{n} - A_{m}x_{m}\right\| &\ =\ \left\|A_{n}x_{n} - v_{1} - A_{m}x_{m} + v_{1}\right\|\\
                                       &\ \leq\ \left\|A_{n}x_{n} - v_{1}\right\| + \left\|A_{m}x_{m} - v_{1}\right\|\\
                                       &\ \leq\ 2^{-2} + 2^{-2}\ =\ 2^{-1}. 
\end{aligned}
\end{multline}
Let $m$ be such as above and define ${m_{1}=m}$, ${\mathbb{N}_{1} = \{n \in \mathbb{N}^{'''} \mid n > m_{1}\}}$. Again, since ${\lambda_{\mathbb{N}^{'}}[A_{n}, \theta] = 0}$, there exists a subsequence of numbers ${\mathbb{N}_{1}^{'''} \subset \mathbb{N}_{1}}$, a sequence of vectors ${(y_{2,n})_{\mathbb{N}_{1}^{'''}} \subset X}$ and a vector ${(u_{2},v_{2}) \in X \times Y}$ such that
\[
\varlimsup_{\mathbb{N}_{1}^{'''}}\left\|(x_{n},A_{n}x_{n}) - (y_{2,n},\theta) - (u_{2},v_{2})\right\|\ \leq\ 2^{-4}.
\]
Projecting onto $Y$ we see that for large enough ${n \in \mathbb{N}_{1}^{'''}}$
\[
\left\|A_{n}x_{n} - v_{2}\right\|\ \leq\ 2^{-4} + 2^{-4}\ =\ 2^{-3};
\]
therefore, for large enough ${n, m \in \mathbb{N}_{1}^{'''}}$
\begin{multline}
\notag
\begin{aligned}
\left\|A_{n}x_{n} - A_{m}x_{m}\right\| &\ =\ \left\|A_{n}x_{n} - v_{2} - A_{m}x_{m} + v_{2}\right\|\\
                                       &\ \leq\ \left\|A_{n}x_{n} - v_{2}\right\| + \left\|A_{m}x_{m} - v_{2}\right\|\\
                                       &\ \leq\ 2^{-3} + 2^{-3}\ =\ 2^{-2}. 
\end{aligned}
\end{multline}
Let $m$ be such as above and define ${m_{2}=m}$, ${\mathbb{N}_{2} = \{n \in \mathbb{N}_{1}^{'''} \mid n > m_{2}\}}$. Continuing this way we can find a series of subsequences
\begin{multline}
\notag
\begin{aligned}
& \mathbb{N}^{''} &\supset&\ \ \mathbb{N}_{1} &\supset&\ \ \mathbb{N}_{2} &\supset&\ \ ...\ &\supset&\ \ \mathbb{N}_{k} &\supset&\ \ ...\\
& (x_{n})_{\mathbb{N}^{''}} &\supset&\ \ (x_{n})_{\mathbb{N}_{1}} &\supset&\ \ (x_{n})_{\mathbb{N}_{2}} &\supset&\ \ ...\ &\supset&\ \  (x_{n})_{\mathbb{N}_{k}} &\supset&\ \ ...
\end{aligned}
\end{multline}
such that for each $k$ and for all ${n, m \in \mathbb{N}_{k}}$ ${\left\|A_{n}x_{n} - A_{m}x_{m}\right\| \leq 2^{-k}}$. Now build a new diagonal sequence ${\mathbb{N}_{d}}$ from all ${\mathbb{N}_{1}, \mathbb{N}_{2}, ..., \mathbb{N}_{k}, ...}$:
\[
\mathbb{N}_{d} = \{n_{i} \mid n_{i} = \min\{n \mid n \in \mathbb{N}_{i}\}\}.
\]
If ${i, j \in \mathbb{N}_{d}}$ and ${i > j}$ then both ${x_{i}}$ and ${x_{j}}$ belong to ${\mathbb{N}_{j}}$ by our construction. Therefore ${\left\|A_{i}x_{i} - A_{j}x_{j}\right\| \leq 2^{-j}}$ for any two ${i > j}$ from ${\mathbb{N}_{d}}$; hence ${(A_{n}x_{n})_{\mathbb{N}_{d}} \subset Y}$ is a Cauchy sequence. Thus, since $Y$ is a complete metric space, the sequence ${(A_{n}x_{n})_{\mathbb{N}_{d}}}$ converges to some vector from $Y$. So, for any subsequence\linebreak[4]  ${(A_{n}x_{n})_{\mathbb{N}^{''}} \subset (A_{n}x_{n})_{\mathbb{N}^{'}} \subset Y}$ we have found a convergent subsequence\linebreak[4] ${(A_{n}x_{n})_{\mathbb{N}_{d}} \subset (A_{n}x_{n})_{\mathbb{N}^{''}}}$ which means that the sequence ${(A_{n}x_{n})_{\mathbb{N}^{'}}}$ is relatively compact in $Y$.

We have just proved proposition 2 for a unit sequence ${(x_{n}, A_{n}x_{n})_{\mathbb{N}^{''}} \subset X \times Y}$. When ${(x_{n}, A_{n}x_{n})_{\mathbb{N}^{''}}}$ is not a unit sequence but just a bounded sequence, consider the value ${\epsilon = \inf\left\|(x_{n}, A_{n}x_{n})\right\|}$. If ${\epsilon = 0}$ then we can choose a subsequence ${(x_{n}, A_{n}x_{n})_{\mathbb{N}^{'''}} \subset X \times Y}$ that converges to ${(\theta, \theta) \in X \times Y}$ thus ${(A_{n}x_{n})_{\mathbb{N}^{''}} \subset Y}$ converges to $\theta \in Y$. 
Otherwise choose a subsequence ${(x_{n}, A_{n}x_{n})_{\mathbb{N}^{'''}}}$ such that norms of its elements ${\alpha_{n} = \left\|(x_{n}, A_{n}x_{n})\right\|}$ approach ${\epsilon > 0}$ and scale that subsequence by dividing its elements onto their norms ${\alpha_{n} \neq 0}$ so that it becomes a unit sequence -- by the preceding proof ${(\frac{1}{\alpha_{n}}A_{n}x_{n})_{\mathbb{N}^{'''}}}$ has a convergent subsequence. Now rescale that convergent subsequence back by multiplying each element on ${\alpha_{n}}$ -- it will still remain a convergent subsequence since ${\alpha_{n} \rightarrow \epsilon < \infty}$. Therefore, for any bounded sequence ${(x_{n}, A_{n}x_{n})_{\mathbb{N}^{''}} \subset X \times Y}$ the sequence ${(A_{n}x_{n})_{\mathbb{N}^{''}} \subset Y}$ has a convergent subsequence, therefore ${(A_{n}x_{n})_{\mathbb{N}^{''}}}$ is relatively compact. This concludes our proof of the proposition 2.\\

In order to prove proposition 3 we first note that given a closed subspace $Z$ of a Banach space $X$ with a finite co--dimension ${\dim X/Z = K < \infty}$ one can find a projection ${P : X \rightarrow Z}$ such that ${\left\|P\right\| \leq K + 1}$. Indeed, using Hahn-Banach theorem K times, it is possible to build $K$ unit vectors ${x_{i} \in X}$ and K unit functional ${f_{i} \in X^{*}}$ such that ${f_{i}x_{j} = \delta_{i}^{j}}$ and ${Z = \cap Ker(f_{i})}$. After that, $K$ one-dimensional projections ${Q_{i} : X \rightarrow sp(\{x_{i}\})}$ are defined as ${Q_{i}x = f_{i}(x)x_{i}}$. It is clear that ${Q = \sum Q_{i}}$ is a projection onto ${sp(\{x_{1},...,x_{K}\})}$ and that ${\left\|Q\right\| \leq K}$, also ${P = I - Q}$ is a projection onto $Z$ and ${\left\|P\right\| \leq K + 1}$.

Now consider a sequence of continuous operators ${(A_{n})_{\mathbb{N}^{'}} \subset \mathcal{BC}(X,Y)}$ such that ${\lambda_{\mathbb{N}^{'}}[A_{n}, \theta] = 0}$. According to proposition 1 there exists a closed subspace ${Y_{1} \subset Y}$ such that ${\dim Y/Y_{1} = K < \infty}$ and a number ${m \in \mathbb{N}^{'}}$ such that for ${n > m}$
\begin{equation}\label{E:70}
\left\|A_{n} \mid_{A_{n}^{-1}(Y_{1})}\right\|\ <\ 1.
\end{equation}
Obviously ${\dim dom(A_{n})/A_{n}^{-1}(Y_{1}) \leq K}$ for same $n$; then, as we noted before, there exist projections ${P_{n} : dom(A_{n}) \rightarrow A_{n}^{-1}(Y_{1})}$ with ${\left\|P_{n}\right\| \leq K + 1}$. 

Let us define ${\mathbb{N}^{''} = \{n \in \mathbb{N}^{'} \mid n > m\}}$ and prove that ${(A_{n} \circ P_{n})_{\mathbb{N}^{''}} \subset \mathcal{BC}(X, Y)}$ ${K_{2}}-$approximates the null operator $\theta$. For that we need to show that if\linebreak[4] ${(x_{n})_{\mathbb{N}^{''}} \triangleleft ((dom(A_{n}))_{\mathbb{N}^{''}}}$ is bounded then ${(A_{n}P_{n}x_{n})_{\mathbb{N}^{''}} \subset Y}$ is relatively compact. Since norms of all ${P_{n}}$ are bounded by ${K + 1}$, the sequence ${(P_{n}x_{n})_{\mathbb{N}^{''}}}$ is bounded if ${(x_{n})_{\mathbb{N}^{''}}}$ is bounded. Therefore, since ${P_{n}x_{n} \in A_{n}^{-1}(Y_{1})}$, we can apply inequality \eqref{E:70} and conclude that ${(A_{n}P_{n}x_{n})_{\mathbb{N}^{''}}}$ is also bounded. Therefore, the sequence ${(P_{n}x_{n}, A_{n}P_{n}x_{n})_{\mathbb{N}^{''}}}$ is bounded in ${X \times Y}$. Now, applying proposition 2 to this bounded sequence we conclude that ${(A_{n}P_{n}x_{n})_{\mathbb{N}^{''}} \subset Y}$ is relatively compact which means that ${(A_{n} \circ P_{n})_{\mathbb{N}^{''}}}$ ${K_{2}}-$approximates  $\theta$. Now for ${n \in \mathbb{N}^{''}}$ set
\[
C_{n} = A_{n} \circ P_{n},\\ \ \ F_{n} = A_{n} \circ (I - P_{n}).
\]
Obviously ${A_{n} = C_{n} + F_{n}}$. Also, we had proved that ${(C_{n})_{\mathbb{N}^{''}}}$ ${K_{2}}-$approximates  $\theta$. Thus, since
\begin{multline}
\notag
\begin{aligned}
\dim F_{n}(dom(A_{n})) &\ =\ \dim A_{n} \circ (I - P_{n})(dom(A_{n}))\\
                       &\ \leq\ \dim(I - P_{n})(dom(A_{n}))\\
                       &\ =\ \dim dom(A_{n})/A_{n}^{-1}(Y_{1})\ =\ K\ <\ \infty,
\end{aligned}
\end{multline}
the proof of the proposition 3 is finished.\\

In order to prove proposition 4 consider decomposition ${A_{n} = C_{n} + F_{n}}$ for ${n \in \mathbb{N}^{''}}$ from proposition 3. By our assumption ${dom(A_{n}) = X}$ for all ${n \in \mathbb{N}^{'}}$. Therefore, according to \cite{vladimirski}, ${\left\| C_{n}\right\|_{q} \rightarrow 0}$ since ${(C_{n})_{\mathbb{N}^{''}}}$\linebreak[4] ${K_{2}}-$approximates the null operator $\theta$. Using this after applying the triangle inequality for $q-$norm and noticing that ${\left\|F_{n} \right\|_{q} = 0}$ since dimensions of all ${F_{n}(X)}$ are finite, we obtain
\begin{multline}
\notag
\begin{aligned}
\varlimsup_{\mathbb{N}^{''}} \left\| A_{n}\right\|_{q}\ =\ \varlimsup_{\mathbb{N}^{''}} \left\|C_{n} + F_{n}\right\|_{q} &\ \leq\ \varlimsup_{\mathbb{N}^{''}} (\left\|C_{n}\right\|_{q}\ +\ \left\|F_{n} \right\|_{q})\\
             &\ \leq\ \varlimsup_{\mathbb{N}^{''}} (\left\|C_{n}\right\|_{q}\ +\ 0)\ =\ \varlimsup_{\mathbb{N}^{''}} \left\|C_{n}\right\|_{q}\ =\ 0.
\end{aligned}
\end{multline}
Finally, since ${\mathbb{N}^{'} \setminus \mathbb{N}^{''}}$ is a finite set, we infer that
\[
\lim_{\mathbb{N}^{'}} \left\| A_{n}\right\|_{q} = 0.
\]
This concludes our proof of the proposition 4.
\end{proof}

We are now ready to prove the following
\begin{theorem}[(Semi--)Fredholm Operators are Stable under Lower $0-$Adjustment]\label{foasucla}
Let $X$ and $Y$ be two Banach spaces, ${A \in \mathcal{B}(X,Y)}$ and ${(A_{n})_{\mathbb{N}^{'}} \subset \mathcal{B}(X,Y)}$. Suppose that ${\lambda_{\mathbb{N}^{'}}[A_{n}, \theta] = 0}$. Then the following propositions are true:
\begin{enumerate}
  \item If $A$ is a lower semi--Fredholm operator, then for large enough ${n \in \mathbb{N}^{'}}$ operators ${A + A_{n}}$ are also lower semi--Fredholm. Moreover, there exists ${L \in \mathbb{N}}$ such that ${\alpha(A + A_{n}) \leq L}$ for the same $n$.
  \item If $A$ is a Fredholm operator, then for large enough ${n \in \mathbb{N}^{'}}$ operators ${A + A_{n}}$ are also Fredholm and ${ind(A + A_{n}) = ind(A)}$. Moreover, there exists ${S \in \mathbb{N}}$ such that ${\beta(A + A_{n}) \leq S}$ for the same $n$.
\end{enumerate}
\end{theorem}
\begin{proof}
Proposition 1 is a consequence from the previous Theorem \ref{lsfps}. Indeed, let us apply a well known technique defining subspaces from ${X \times Y}$ (see \cite{kato})
\[
M\ =\ M_{n}\ =\ X \times \{\theta\},\ \ \ N\ =\ G_{A},\ \ \ N_{n}\ =\ G_{A + A_{n}};
\]
obviously ${(M, N)}$ is a lower semi--Fredholm pair of closed subspaces in ${X \times Y}$, also it is clear that ${\lambda_{\mathbb{N}^{'}}[M_{n}, M] = 0}$. Let us prove that ${\lambda_{\mathbb{N}^{'}}[N_{n}, N] = 0}$. By definition ${\lambda_{\mathbb{N}^{'}}[N_{n}, N] = \lambda_{\mathbb{N}^{'}}[A + A_{n}, A]}$. Since ${A \in \mathcal{B}(X, Y)}$ and ${\lambda_{\mathbb{N}^{'}}[A_{n}, \theta] = 0}$ it follows from Theorem \ref{nua} proposition 3 that ${\lambda_{\mathbb{N}^{'}}[A + A_{n}, A] = 0}$ . Therefore, ${\lambda_{\mathbb{N}^{'}}[N_{n}, N] = 0}$. Thus, by Theorem \ref{lsfps} there exists a number ${K_{1} \in \mathbb{N}}$ such that pairs ${}(M_{n}, N_{n})$ are lower semi--Fredholm and values of ${\beta(M_{n}, N_{n})}$ are limited by some finite number $L$ for ${n > K_{1}}$. However it is obvious that ${\alpha(M_{n}, N_{n}) = \alpha(A + A_{n})}$, so the values of ${\alpha(A + A_{n})}$ are limited by a finite number $L$ for the same $n$.

In order to prove proposition 2 observe that according to proposition 3 from the previous Theorem \ref{solua} we may assume that all ${A_{n}}$ can be decomposed into ${A_{n} = C_{n} + F_{n}}$ so that ${C_{n}, F_{n} \in \mathcal{B}(X, Y)}$, the sequence ${(C_{n})_{\mathbb{N}^{'}}}$ ${K_{2}-}$approximates $\theta$, and
\begin{equation}\label{E:1234}
\varlimsup_{\mathbb{N}^{'}} \dim F_{n}(X)\ \leq\ R
\end{equation}
for some ${R < \infty}$. Applying a well-known fact about ${K_{2}-}$approximation from \cite{vladimirski} we conclude that there exists ${K_{2} \in \mathbb{N}}$ such that for ${n > K_{2}, n \in \mathbb{N}^{'}}$ operators ${A + C_{n}}$ are also Fredholm and their indices are all equal to the index of $A$. Therefore, since all ${F_{n}}$ are finite, by the classical result from \cite{atkinson} operators ${A + C_{n} + F_{n}}$ are Fredholm and their indices are all equal to the index of $A$ for the same $n$. At the same time
\[
\alpha(A + C_{n} + F_{n})\ \leq\ \alpha(A + C_{n}) + \dim F_{n}(X)\ \leq\ \alpha(A + C_{n})\ +\ R
\]
due to \eqref{E:1234}. Hence, for all ${n > \max(K_{1}, K_{2})}$ values of ${\beta(A + A_{n})}$ are limited by ${S = L + R + ind(A)}$:
\begin{multline}
\notag
\begin{aligned}
\beta(A + A_{n})\ =\ \beta(A + C_{n} + F_{n}) & \ =\ \alpha(A + C_{n} + F_{n})\ -\ ind(A + C_{n} + F_{n})\\
                                              & \ =\ \alpha(A + C_{n} + F_{n})\ -\ ind(A)\\
                                              & \ \leq\ \alpha(A + C_{n})\ +\ R\ -\ ind(A)\\
                                              & \ \leq\ L\ +\ R\ -\ ind(A)\\
                                              & \ =\ S.
\end{aligned}
\end{multline}
\end{proof}

\newpage
\section[Stability in $BFU-$Spaces]{(Semi--)Fredholm Stability in $BFU-$Spaces}\label{sibfus}

In the previous section we presented examples when Fredholm stability properties break under uniform $\lambda-$adjustment. In particular, in the space ${l_{\infty}}$ the analog of the small gap theorem is valid only for constants ${< 1/2}$ while semi--Fredholm and Fredholm linear operators are not stable under uniform upper $\lambda-$adjustment for any ${\lambda \geq 0}$. Those examples prompt one to look for a special class of Banach spaces where the analog of the small gap theorem still holds for constants ${< 1}$ and where stability of Fredholm properties is preserved under $\lambda-$adjustment.

In this section we present one such class that we call \emph{the class of BFU-spaces} -- those are Banach spaces for which every norm bounded set of functionals is a Fr\'{e}chet-Urysohn space in ${weak^{*}}$ topology. After proving in topological Fr\'{e}chet--Urysohn spaces a general lemma about the convergence points of the sequence of subsets, we proceed with proving in $BFU-$spaces the analog of the small gap theorem with the constants ${< 1}$ and stability of semi--Fredholm and Fredholm linear operators and pairs of subspaces. The latter theorems show improved estimates for lower uniform $\lambda-$adjustment, as well as stability of (semi--)Fredholm properties under upper $\lambda-$adjustment.

\subsection[$BFU-$Spaces]{$BFU-$Spaces}

Recall that a Hausdorff topological space ${(T, \tau)}$ is called \emph{a Fr\'{e}chet--Urysohn space} if for every set ${G \subset T}$ and every element ${t \in \overline{G}}$ there is a sequence ${(t_{n})_{\mathbb{N}} \subset G}$ convergent to t (see  \cite{kakol_saxon,arhangelskii_ponomarev}). 
\begin{definition}[${BFU-}$Space]\label{D:bfus}
We call a Banach space X a ${BFU-}$space if every norm bounded set of functionals from its dual ${X^{*}}$ is a Fr\'{e}chet--Urysohn space in topology induced by the ${weak^{*}}$ topology ${w^{*} = \sigma(X^{*}, X)}$.
\end{definition}
As we shall see shortly, the class of ${BFU-}$spaces is rather broad. In order to establish its content we will need one preliminary definition and one result.

Recall from \cite{kantorovich_akilov} p. 284 that a subset $G$ from a topological space ${(T, \tau)}$ is called \emph{relatively countably compact} if any sequence ${(g_{n})_{\mathbb{N}} \subset G}$ has a limit point ${g \in T}$, i.e. there exists ${g \in T}$ such that for each open neighborhood ${U \in \tau}$ of $g$ and for each number ${n \in \mathbb{N}}$ there exists another number ${m > n}$ such that\linebreak[4] ${g_{m} \in U}$.

The following fundamental lemma can be found in \cite{kantorovich_akilov} p.285; its proof has first appeared in \cite{whitley}\footnote{See also \cite{cohen,kelley_namioka} where similar ideas are used for the proof of the Eberlein--\^{S}mulian theorem.}:
\begin{lemma}[Closure in the Second Dual]\label{citsd}
Let $E$ be a relatively weakly countably compact subset from a Banach space X. Let $E_{1}$ be ${\sigma(X^{**}, X^{*})-}$closure of the set ${\pi(E)}$ where ${\pi:X \rightarrow X^{**}}$ -- the canonical injection. Then ${E_{1}}$ is ${\sigma(X^{**}, X^{*})-}$compact and for every element ${F \in E_{1}}$ there is a sequence ${(x_{n})_{\mathbb{N}} \subset E}$ with a single limit point ${x \in X}$ such that ${\pi(x) = F}$.
\end{lemma}
We are now ready to examine the content of the class of $BFU-$spaces in the following lemma:
\begin{lemma}[Examples of $BFU-$Spaces]\label{eobfus}
The class of $BFU-$Spaces includes
\begin{enumerate}
  \item separable Banach spaces,
  \item reflexive Banach spaces,
  \item products of separable and $BFU-$spaces,
  \item products of any finite number of reflexive and separable spaces,
  \item quotient spaces of ${BFU-}$spaces,
  \item complemented subspaces of ${BFU-}$spaces.
\end{enumerate}
\end{lemma}
\begin{proof}
If $X$ is a separable Banach space, then $X^{*}$ is metrizable in ${w^{*}}$ topology. Therefore, ${X^{*}}$ is Fr\'{e}chet--Urysohn since every metrizable space is Fr\'{e}chet--Urysohn.

Let $X$ be a reflexive Banach space and consider a norm bounded set $G$ from its dual space ${X^{*}}$. By the Banach-Alaoglu theorem its closure ${\overline{G}^{w^{*}}}$ in ${weak^{*}}$ topology is compact. Since $X$ is reflexive, by the Eberlein--\^{S}mulian theorem ${\overline{G}^{w^{*}}}$ is countably compact in ${weak^{*}}$ topology. Since $X$ is reflexive, the weak topology ${w^{**} = \sigma(X^{*}, X^{**})}$ induced by ${X^{**}}$ coincides with the weak topology ${w^{*} = \sigma(X^{*}, X)}$. Therefore, $G$ is relatively countably compact in ${w^{**}}$ topology. Let ${\pi:X^{*} \rightarrow X^{***}}$ be a canonical injection. By the previous lemma \ref{citsd} every element $g$ from the closure of ${\pi(G) \subset X^{***}}$ in the ${\sigma(X^{***}, X^{**})}$ topology is a single limit point of some sequence ${(g_{n})_{\mathbb{N}} \subset \pi(G)}$. By the Eberlein--\^{S}mulian theorem there exists a  subsequence ${(g_{n})_{\mathbb{N}^{'}}}$ convergent in ${\sigma(X^{***}, X^{****})}$ topology. Since ${X}$ is reflexive, topologies ${\sigma(X^{***}, X^{****})}$ and ${\sigma(X^{***}, X^{**})}$ coincide. Therefore, ${(g_{n})_{\mathbb{N}^{'}}}$ is convergent in ${\sigma(X^{***}, X^{**})}$ topology. Since $g$ is the only limit point of ${(g_{n})_{\mathbb{N}}}$ in ${\sigma(X^{***}, X^{**})}$ topology, we shall conclude that ${(g_{n})_{\mathbb{N}^{'}}}$ converges to $g$ in that topology. Therefore, the set ${\pi(G) \subset X^{***}}$ is Fr\'{e}chet--Urysohn in ${\sigma(X^{***}, X^{**})}$ topology. Since $X$ is reflexive ${X^{} = X^{**}}$, ${X^{*} = X^{***}}$, also ${G = \pi(G) \subset X^{*}}$ and ${\sigma(X^{***}, X^{**})}$ topology coincides with ${\sigma(X^{*}, X)}$ topology. Therefore, the set ${G}$ is Fr\'{e}chet--Urysohn in ${\sigma(X^{*}, X)}$ topology.

A product of a ${BFU-}$space and a separable Banach space is again a\linebreak[4] ${BFU-}$space -- this follows from the Banach-Alaoglu theorem and from the fact that a product of a compact Fr\'{e}chet-Urysohn space onto a compact space with the first axiom of countability is Fr\'{e}chet--Urysohn (see \cite{archangelskii,olson}). 

From the last statement and from the fact that a finite product of separable spaces is separable, while a finite product of reflexive spaces is reflexive, it follows that a product of any finite number of reflexive and separable spaces is a ${BFU-}$space.

Due to a natural isomorphism between a dual to a quotient space of a Banach space over a closed subspace and the annihilator of that closed subspace it is clear that a quotient space of a ${BFU-}$space is a ${BFU-}$space. 

Also, one can easily establish that a closed complemented subspace of a ${BFU-}$space is still a ${BFU-}$space since it is isomorphic to a quotient taken over its complement.
\end{proof}

Now we present three technical lemmas. The first lemma is equivalent to a theorem about a product of compact spaces from \cite{archangelskii,olson}. Since it will be the key to the proofs of stability of (semi--)Fredholm properties, we lay out its independent and elementary proof. The second lemma is a seemingly obvious statement in general topological spaces which proof appears to be somewhat elusive unless presented rigorously. The third lemma establishes, with the help of the first and second lemmas, a connection between the Fr\'{e}chet-Urysohn property and uniform $\lambda-$adjustment.

\begin{lemma}[Closure of a Sequence of Subsets]\label{L:css}
Let ${(T, \tau)}$ be a compact Fr\'{e}chet-Urysohn topological space, ${(G_{n})_{\mathbb{N}}}$ is a sequence of non-empty subsets from $T$ and an element ${t \in T}$ belongs to a closure of every subset ${G_{n} \subset X}$ for ${n \in \mathbb{N}}$. Then there exists a subsequence ${\mathbb{N}^{'} \subset \mathbb{N}}$ and a sequence of elements ${(t_{n})_{\mathbb{N^{'}}} \triangleleft (G_{n})_{\mathbb{N^{'}}}}$ convergent to $t$. 
\end{lemma}
\begin{proof}
Define
\begin{equation}\label{E:81}
G = \{ g \in T \mid (t_{n})_{\mathbb{N^{'}}} \rightarrow g\ for\ some\ (t_{n})_{\mathbb{N^{'}}} \triangleleft (G_{n})_{\mathbb{N^{'}}} \}.
\end{equation}
Let us first prove that ${t \in \overline{G}}$. Consider any open neighborhood $V$ of $t$. Since ${(T, \tau)}$ is a Hausdorff compact space it is regular and, therefore, there exists a neighborhood $U$ of $t$ such that ${\overline{U} \subset V}$.
According to the lemma's condition there exists a sequence ${(t_{n})_{\mathbb{N^{'}}} \triangleleft (G_{n})_{\mathbb{N^{'}}}}$ such that ${t_{n} \in U}$ for all ${n \in \mathbb{N^{'}}}$. Therefore, denoting ${S = \{t_{n} \mid n \in \mathbb{N}^{'}\}}$, we have the inclusion 
${\overline{S} \subset \overline{U} \subset V}$. Let us prove that one can choose a subsequence ${(t_{n})_{\mathbb{N^{''}}} \subset S}$ convergent to some ${g \in \overline{U} \subset V}$.
 
Suppose that ${\overline{S} \setminus S \neq \emptyset}$. Then choose ${g \in \overline{S} \setminus S}$; since ${(T, \tau)}$ is Fr\'{e}chet-Urysohn, there exists a sequence from $S$ convergent to $g$. After a simple renumeration of that sequence, one constructs a subsequence ${(t_{n})_{\mathbb{N^{''}}} \rightarrow g}$. Now consider the case when ${\overline{S} \setminus S = \emptyset}$ and define ${S_{n_{1}} = S \setminus \{t_{n_{1}}\}}$. By the same reasoning we might find that ${\overline{S_{n_{1}}} \setminus S_{n_{1}} = \emptyset}$. In this case, since ${(T, \tau)}$  is regular, there exists a neighborhood ${U_{n_{1}}}$ of ${t_{n_{1}}}$ such that ${\overline{U_{n_{1}}} \cap S_{n_{1}} = \emptyset}$. Continuing this way we could build a sequence of closed sets ${S_{n_{k}}}$ and non-intersecting neighborhoods ${U_{n_{k}}}$ each neighborhood containing a single element ${t_{n_{k}}}$, for all of the infinite number of elements ${t_{n_{k}}}$ from $S$. As ${(T, \tau)}$ is compact and by assumption $S$ is closed, there exists a finite cover of $S$ by some non-intersecting finite number of neighborhoods ${U_{n_{k_{1}}}, ..., U_{n_{k_{j}}}}$ which is a contradiction.

Therefore, one can choose a subsequence ${(t_{n})_{\mathbb{N^{''}}} \subset S}$ convergent to some ${g \in \overline{U} \subset V}$. Therefore ${g \in G}$ according to \eqref{E:81} and ${g \in V}$ at the same time. As $V$ was an arbitrary neighborhood of $t$, we must conclude that ${t \in \overline{G}}$.\\

Now, since ${(T, \tau)}$ is Fr\'{e}chet-Urysohn, we can choose a sequence ${(g_{n})_{\mathbb{N}} \subset G}$, ${(g_{n})_{\mathbb{N}} \rightarrow t}$. According to \eqref{E:81} there exists a series of infinite sets of natural numbers ${\mathbb{N}_{1}, \mathbb{N}_{2}, ...}$ and a series of sequences ${(t_{m}^{1})_{m \in \mathbb{N}_{1}}, (t_{m}^{2})_{m \in \mathbb{N}_{2}}, ...}$ such that ${(t_{m}^{k})_{m \in \mathbb{N}_{k}} \rightarrow g_{k}}$ and ${(t_{m}^{k})_{m \in \mathbb{N}_{k}} \triangleleft (G_{m})_{m \in \mathbb{N}_{k}}}$ for each ${k \in \mathbb{N}}$. Let us build a new series of sets of natural numbers ${\mathbb{N}^{'}_{1}, \mathbb{N}^{'}_{2}, ...}$ by induction:
\begin{itemize}
  \item Define set ${\mathbb{N}^{'}_{1} = \mathbb{N}_{1}}$ and number ${min_{1} = \min \{n \mid n \in \mathbb{N}^{'}_{1}\}}$.
  \item If ${\mathbb{N}^{'}_{k}}$ and ${min_{k}}$ have already been defined, then declare
\begin{multline}
\notag
\begin{aligned}
& \mathbb{N}^{'}_{k+1} = \{n \in \mathbb{N}_{k+1} \mid n > min_{k}\},\\
& min_{k+1} = \min \{n \mid n \in \mathbb{N}^{'}_{k+1}\}.
\end{aligned}
\end{multline}
\end{itemize}
From this construction it is clear that
\begin{equation}\label{E:400}
min_{1}\ <\ min_{2}\ <\ ...\ <\ min_{k}\ <\ ...
\end{equation}
Also, the sequences of elements ${(t_{m}^{k})_{m \in \mathbb{N}^{'}_{k}} \triangleleft (G_{m})_{m \in \mathbb{N}^{'}_{k}}}$ enumerated by new sets ${\mathbb{N}^{'}_{k}}$ still converge to ${g_{k}}$ for each ${k \in \mathbb{N}}$. The important property of the just built series of sets is such that the total number of elements which lower indices are bounded from above is finite. That is, for any ${L \in \mathbb{N}}$
\begin{equation}\label{E:401}
card \{t^{k}_{m} \mid m \leq L, k \in \mathbb{N}\}\ <\ \infty.
\end{equation}
The proof is quite trivial -- by \eqref{E:400} there can be no more than ${L}$ elements in the first sequence, no more than ${L-1}$ elements in the second sequence, etc. -- totally the count does not exceed
\[
\sum_{i=0}^{L-1}L-i = \frac{L \times (L + 1)}{2}.
\]
Having established \eqref{E:401} it is also easy to see that
\begin{equation}\label{E:402}
\{t^{k}_{m} \mid m  \in \mathbb{N}^{'}_{k}, k \in \mathbb{N}\}\ \setminus\ \bigcup_{i > L}G_{i}\ \subset\ \{t^{k}_{m} \mid m \leq L, k \in \mathbb{N}\}
\end{equation}

Now for each ${i = 0, 1, 2, ...}$ define sets
\[
H_{i} = \bigcup_{k=i+1}^{\infty} \{t_{m}^{k} \mid m \in \mathbb{N}^{'}_{k}\}.
\]
Let us prove that for each ${i \in \mathbb{N}}$
\begin{equation}\label{E:404}
t \in \overline{H_{i}}. 
\end{equation}
Let $U$ be a neighborhood of $t$ and ${i \in \mathbb{N}}$. Since ${(g_{n})_{\mathbb{N}} \rightarrow t}$ there exists ${k > i}$ such that ${g_{k} \in U}$. Since ${(t_{m}^{k})_{m \in \mathbb{N}^{'}_{k}} \rightarrow g_{k}}$, there exists some ${{m} \in \mathbb{N}^{'}_{k}}$ such that ${t_{m}^{k} \in U}$. Yet ${t_{m}^{k} \in H_{i}}$ since ${k > i}$, therefore ${t \in \overline{H_{i}}}$.\\

Now suppose that lemma is not true. Since ${(T, \tau)}$ is Fr\'{e}chet-Urysohn and ${t \in \overline{H_{0}}}$, we can find a sequence ${(z_{n})_{\mathbb{N}} \subset H_{0}}$ such that ${(z_{n})_{\mathbb{N}} \rightarrow t}$. Suppose that for any ${L \in \mathbb{N}}$ we can find ${m > L}$ such that ${z_{m} \in \cup_{i > L}G_{i}}$. Then we could choose an infinite subsequence ${\mathbb{N}^{'}}$ such ${(z_{m})_{\mathbb{N}^{'}} \triangleleft (G_{m})_{\mathbb{N}^{'}}}$; at the same time ${(z_{m})_{\mathbb{N}^{'}} \rightarrow t}$ (since ${(z_{n})_{\mathbb{N}} \rightarrow t}$) which would prove our lemma contrary to the assumption that the lemma is false. Therefore,
\[
\{z_{n} \mid n \geq L\}\ \bigcap\ (\bigcup_{i > L}G_{i})\ =\ \emptyset
\]
for some ${L \in \mathbb{N}}$. Yet each ${z_{n}}$ is some ${t^{k}_{m}}$, thus by the previous equality
\[
\{z_{n} \mid n \geq L\}\ \subset\ \{t^{k}_{m} \mid m  \in \mathbb{N}^{'}_{k}, k \in \mathbb{N}\}\ \setminus\ \bigcup_{i > L}G_{i};
\]
therefore by \eqref{E:402}
\[
\{z_{n} \mid n \geq L\}\ \subset\ \{t^{k}_{m} \mid m \leq L, k \in \mathbb{N}\},
\]
thus by \eqref{E:401} the set ${\{z_{n} \mid n \in \mathbb{N}\}}$ is finite. However ${(z_{n})_{\mathbb{N}} \rightarrow t}$ and all ${z_{n}}$ belong to the finite number of sets ${G_{1}, ..., G_{L}}$ which means that for some ${j_{1} \leq L}$ we can choose a subsequence from ${\{z_{n} \mid n \geq L\} \cap G_{j_{1}}}$ which converges to t. Since ${\{z_{n} \mid n \in \mathbb{N}\}}$ is finite, we infer that elements from that subsequence must be a single constant element from ${G_{j_{1}}}$. Therefore, since ${(T, \tau)}$ is Hausdorff, we shall conclude that that element is ${t}$, which means that ${t \in G_{j_{1}}}$.\\

Repeat the previous argument for ${H_{j_{1}+1}}$ instead of ${H_{0}}$ -- since we have already proved in \eqref{E:404} that ${t \in \overline{H_{j_{1}}}}$, we will find that ${t \in G_{j_{2}}}$ for some ${j_{2} \geq j_{1}+1}$. Continuing this process we can prove that there exists an infinite series of numbers ${j_{1} < j_{2} < ...}$ such that ${t \in \cap_{k=1}^{\infty}G_{j_{k}}}$. Therefore we can simply choose ${z_{j_{k}} = t}$ for every ${k \in \mathbb{N}}$ so that ${(z_{j_{k}})_{k \in \mathbb{N}} \triangleleft (G_{j_{k}})_{k \in \mathbb{N}}}$ and ${(z_{j_{k}})_{k \in \mathbb{N}} \rightarrow t}$ which contradicts to our assumption that the lemma is false.
\end{proof}

\begin{lemma}[Closure of a Countable Union of Subsets]\label{L:ccus}
Let ${(T, \tau)}$ be a topological space, ${(G_{n})_{\mathbb{N}}}$ be a sequence of subsets from $T$. For each ${n \in \mathbb{N}}$ define sets
\[
H_{n} = \bigcup_{i = n}^{\infty} G_{n}.
\]
Suppose that there exists a subsequence ${(h_{n})_{\mathbb{N}_{1}} \triangleleft (H_{n})_{\mathbb{N}_{1}}}$ convergent to some element ${t \in T}$. Then there exists a subsequence ${(g_{n})_{\mathbb{N}_{2}} \triangleleft (G_{n})_{\mathbb{N}_{2}}}$ convergent to $t$.
\end{lemma}
\begin{proof}
Since ${\mathbb{N}_{1}}$ is an infinite subset of ${\mathbb{N}}$, we can uniquely enumerate its elements as ${\mathbb{N}_{1} = \{n_{1} < n_{2} < ... < n_{k} < ...\}}$. We will use this enumeration to build a sequence ${(g_{n})_{\mathbb{N}_{2}} \triangleleft (G_{n})_{\mathbb{N}_{2}}}$ inductively:
\begin{itemize}
  \item Since ${h_{n_{1}} \in H_{n_{1}}}$, by definition of ${H_{n_{1}}}$ there exists some number ${r \geq n_{1}}$ such that ${h_{n_{1}} \in G_{r}}$. Denote ${r_{1} = l}$, ${p_{1} = 1}$ and ${g_{r_{1}} = h_{n_{1}} = h_{n_{p_{1}}}}$.
  \item Suppose we have built three finite sets of $k$ natural numbers ${\{r_{1} < ... < r_{k}\}}$, ${\{p_{1} < ... < p_{k}\}}$ and ${\{n_{p_{1}} < ... < n_{p_{k}}\}}$ such that ${r_{i} \geq n_{p_{i}}}$, ${g_{r_{i}} \in G_{r_{i}}}$ and ${g_{r_{i}} = h_{n_{p_{i}}}}$ for each ${i=1,...,k}$. Consider set ${H_{r_{k} + 1}}$. Since ${(h_{n})_{\mathbb{N}_{1}} \triangleleft (H_{n})_{\mathbb{N}_{1}}}$, there exists some ${n_{m} \in \mathbb{N}_{1}}$, ${n_{m} \geq r_{k} + 1}$ such that ${h_{n_{m}} \in H_{n_{m}}}$. Therefore, by definition of ${H_{n_{m}}}$ there is some ${j \geq n_{m}}$ such that ${h_{n_{m}} \in G_{j}}$. Denote ${r_{k+1} = j}$, ${p_{k+1} = m}$ and  ${g_{r_{k+1}} = h_{n_{m}} = h_{n_{p_{k+1}}}}$. Obviously, by our construction ${r_{k+1} = j \geq n_{m} \geq r_{k} + 1 > r_{k}}$, also ${r_{k+1} = j \geq n_{m} = n_{p_{k+1}}}$.
\end{itemize}
  Now consider the just built set of natural numbers ${\mathbb{N}_{2} = \{r_{1}, r_{2}, ..., r_{k}, ...\}}$ and the sequence ${(g_{r_{k}})_{k \in \mathbb{N}}}$. By construction ${(g_{r_{k}})_{k \in \mathbb{N}} \triangleleft (G_{r_{k}})_{k \in \mathbb{N}}}$. Let us prove that this sequence converges to $t$. Let $U$ be a neighborhood of $t$; since ${(h_{n_{i}})_{i \in \mathbb{N}}}$ converges to $t$, there exists some ${L \in \mathbb{N}}$ such that for any ${j > L}$ ${h_{n_{j}} \in U}$. By our construction numbers ${n_{p_{k}}}$ strictly increase, so there exists some ${q}$ such that ${n_{p_{q}} > L}$. Hence if ${j > q}$ then ${n_{p_{j}} > L}$ and therefore ${g_{r_{j}} = h_{n_{p_{j}}} \in U}$ which means that ${(g_{r_{k}})_{k \in \mathbb{N}}}$ converges to $t$.
\end{proof}

The next lemma illustrates the key technique of separating the null functional from the subsets of annihilators of the uniformly $\lambda-$adjusted subspaces. This lemma itself and its separation technique will be used further in a number of the stability results.
\begin{lemma}[The Null Functional Separation]\label{L:absnf}
Let ${(M_{n})_{\mathbb{N}^{'}}}$ and ${(P_{n})_{\mathbb{N}^{'}}}$ be two sequences of closed subspaces in a $BFU-$space $X$, ${M_{n} \neq \{\theta\}}$ for all ${n \in \mathbb{N}^{'}}$. Suppose that real positive numbers $r$ and $\eta$ are such that
\[
r \times (\lambda_{\mathbb{N}^{'}}[M_{n}, P_{n}] + \eta) < 1,
\]
and for every ${n \in \mathbb{N}^{'}}$ the sets of functionals ${T_{n} \in X^{*}}$ and ${V_{n} \in X^{*}}$ are defined like this
\begin{multline}\label{E:250}
\begin{aligned}
& T_{n} = \{f \in P_{n}^{\bot}\ \bigm|\ \left\|f\right\| \leq r\ and\ \sup \{|fx|\ \bigm|\ x \in M_{n}, \left\|x\right\| = 1\} \geq 1 \},\\
& V_{n} = \bigcup_{i \geq n, i \in \mathbb{N}^{'}} T_{i}.
\end{aligned}
\end{multline}
Suppose that ${T_{n} \neq \emptyset}$ for all ${n \in \mathbb{N}^{'}}$. Then for some ${k \in \mathbb{N}^{'}}$ the null functional ${\theta \in X^{*}}$ does not belong to the weak closure of ${V_{k}}$, i.e. ${\theta \notin \overline{V_{k}}^{w^{*}}}$.
\end{lemma}
\begin{proof}
Assume the opposite -- that for some infinite subset ${\mathbb{N}^{''} \subset \mathbb{N}^{'}}$ ${\theta \in \overline{V_{n}}^{w^{*}}}$ for all ${n \in \mathbb{N}^{''}}$. Obviously, all defined this way sets ${V_{n}}$ are contained within a closed ball ${B_{r} \subset X^{*}}$ radius ${r}$ centered at ${\theta \in X^{*}}$ which is a norm-bounded subset of ${X^{*}}$, therefore it is ${w^{*}}$-compact due to the Banach-Alaoglu theorem; in addition, the topological space ${(B_{r}, w^{*})}$ is Fr\'{e}chet-Urysohn since $X$ is a $BFU-$space.
Then we can apply Lemma \ref{L:css} to the compact Fr\'{e}chet-Urysohn space ${(B_{r}, w^{*})}$ and its sequence of subsets ${(V_{n})_{\mathbb{N}^{''}}}$ concluding that there exists a subsequence ${(f_{n})_{\mathbb{N}^{'''}} \triangleleft (V_{n})_{\mathbb{N}^{'''}}}$ such that 
\begin{equation}\label{E:255}
(f_{n})_{\mathbb{N}^{'''}} \xrightarrow[w^{*}]{} \theta \in X^{*}.
\end{equation} 
Taking into account that all ${V_{n}}$ are unions of ${T_{i}}$ with indices ${i \geq n}$, we can apply Lemma \ref{L:ccus} and conclude that there exists an infinite subset ${\mathbb{N}_{1}^{''} \subset \mathbb{N}^{''}}$ and a subsequence ${(f_{n})_{\mathbb{N}_{1}^{''}} \triangleleft (T_{n})_{\mathbb{N}_{1}^{''}}}$ such that ${(f_{n})_{\mathbb{N}_{1}^{''}} \rightarrow \theta}$ in ${weak^{*}}$ topology. 
Now recall that by \eqref{E:250} there exists a sequence of unit vectors ${(x_{n})_{\mathbb{N}_{1}^{''}} \triangleleft (M_{n})_{\mathbb{N}_{1}^{''}}}$ such that ${|f_{n}x_{n}| \rightarrow 1}$. Also, by definition of $\lambda-$adjustment there exists a subsequence ${(y_{n})_{\mathbb{N}_{1}^{'''}} \triangleleft (P_{n})_{\mathbb{N}_{1}^{'''}}}$ and a vector ${z \in X}$ such that
\[
\varlimsup_{\mathbb{N}_{1}^{'''}} \left\| x_{n} - y_{n} - z\right\|\ \leq\ \lambda_{\mathbb{N}^{'}}[M_{n}, P_{n}] + \eta
\] 
In addition ${f_{n}y_{n} = 0}$ since ${f_{n} \in T_{n} \subset P_{n}^{\bot}}$; also ${f_{n}z \rightarrow 0}$ by \eqref{E:255}. Therefore, using all that rewrite
\begin{multline}
\notag
\begin{aligned}
1 = \varlimsup_{\mathbb{N}_{1}^{'''}} |f_{n}x_{n}| &\ =\ \varlimsup_{\mathbb{N}_{1}^{'''}} |f_{n}x_{n} - f_{n}z + f_{n}z|\\
                                                   &\ \leq\ \varlimsup_{\mathbb{N}_{1}^{'''}} |f_{n}x_{n} - f_{n}z| + \varlimsup_{\mathbb{N}_{1}^{'''}} |f_{n}z|\\
                                                   &\ \leq\ \varlimsup_{\mathbb{N}_{1}^{'''}} |f_{n}x_{n} - f_{n}z| + 0\\
                                                   &\ =\ \varlimsup_{\mathbb{N}_{1}^{'''}} |f_{n}x_{n} - \theta - f_{n}z|\\
                                                   &\ =\ \varlimsup_{\mathbb{N}_{1}^{'''}} |f_{n}x_{n} - f_{n}y_{n} - f_{n}z|\\
                                                   &\ \leq\ \varlimsup_{\mathbb{N}_{1}^{'''}} \left\|f_{n}\right\| \times \left\|x_{n} - y_{n} - z\right\|\\
                                                   &\ \leq\ r \times \varlimsup_{\mathbb{N}_{1}^{'''}} \left\|x_{n} - y_{n} - z\right\|\\
                                                   &\ \leq\ r \times (\lambda_{\mathbb{N}^{'}}[M_{n}, P_{n}] + \eta).
\end{aligned}
\end{multline}
However, the above inequality is in direct contradiction with our choice of $r$ and $\eta$: ${r \times (\lambda_{\mathbb{N}^{'}}[M_{n}, P_{n}] + \eta) < 1}$.
\end{proof}

\subsection[The $BFU$ Small Uniform Adjustment Theorem]{The $BFU$ Small Uniform Adjustment Theorem}\label{bfusuat}

The following theorem is a stronger version of the Small Uniform Adjustment Theorem \ref{T:suat} from Subsection \ref{S:ula} -- all improvements are due to the $BFU$ property of the Banach space $X$. Note that constant $1$ is used in all conditions for uniform $\lambda-$adjustment; also proposition 1 contains single subspaces ${P_{n}}$ instead of their sums. Finally, the last proposition 4 establishes good structural properties under upper uniform adjustment of subspaces with the entire space.
\begin{theorem}[The $BFU$ Small Uniform Adjustment Theorem]\label{T:bfusuat}
Let ${({M}_n)_{\mathbb{N}^{'}}}$ and ${({P}_n)_{\mathbb{N}^{'}}}$ be two sequences of closed subspaces in a $BFU-$space $X$. Then
\begin{enumerate}
  \item For every ${\epsilon > 0}$ there exists a finite-dimensional subspace ${Q_{\epsilon} \subset X}$ such that
\begin{equation}\label{191}
\lambda_{\mathbb{N}^{'}}[M_{n}, P_{n}] + \epsilon\ \geq\ \varlimsup_{n \in \mathbb{N}^{'}} \delta(M_{n}, Q_{\epsilon} + P_{n}).
\end{equation}
  \item If ${\lambda_{\mathbb{N}^{'}}[M_{n}, P_{n}] < 1}$, then there exist ${K \in \mathbb{N}^{'}}$, ${L \in \mathbb{N}}$ such that for ${m > K}$
\[
\dim M_{m}\ \leq\ L\ +\ \dim P_{m}.
\]
  \item If ${\lambda_{\mathbb{N}^{'}}[M_{n}, P_{n}] < 1}$ then there exist ${K \in \mathbb{N}^{'}}$, ${L \in \mathbb{N}}$ such that for ${m > K}$
\[
\dim X/P_{m}\ \leq\ L + \dim X/M_{m}.
\]
  \item The following conditions are equivalent:
  \begin{enumerate}
    \item ${\lambda_{\mathbb{N}^{'}}[X, P_{n}] < 1}$. 
    \item ${\lambda_{\mathbb{N}^{'}}[X, P_{n}] = 0}$.
    \item If ${(x_{n})_{\mathbb{N}^{'}} \subset X}$ is a bounded sequence, then there exists a sequence ${(y_{n})_{\mathbb{N}^{'}} \triangleleft (P_{n})_{\mathbb{N}^{'}}}$ such that ${(x_
    {n} - y_{n})_{\mathbb{N}^{'}}}$ is relatively compact.
    \item There exists a finite dimensional subspace ${Q \subset X}$ and a number ${M \in \mathbb{N}}$ such that for ${n > M}$
\[
X = Q + P_{n};
\]
    also, the norms of the natural projections ${Pr_{n} : Q^{\bot} \oplus P_{n}^{\bot} \rightarrow Q^{\bot}}$ are all bounded from above.
  \end{enumerate}
\end{enumerate}
\end{theorem}
\begin{proof}
Denote ${\lambda = \lambda_{\mathbb{N}^{'}}[M_{n}, P_{n}]}$. If ${\lambda = 1}$ then set ${Q_{\epsilon} = \{\theta\}}$ for every ${\epsilon > 0}$. Since the gap between any two subspaces does not exceed 1, the following is true
\[
\varlimsup_{n \in \mathbb{N}^{'}} \delta(M_{n}, Q_{\epsilon} + P_{n})\ \leq\ 1\ <\ 1 + \epsilon\ =\ \lambda + \epsilon\ =\  \lambda_{\mathbb{N}^{'}}[M_{n}, P_{n}] + \epsilon
\]
which proves proposition 1. Therefore, we only need to consider the case when ${\lambda < 1}$.
For any ${\epsilon \in (0, 1 - \lambda)}$ partition set ${\mathbb{N}^{'}}$ into two subsets ${\mathbb{N}_{0}^{''}, \mathbb{N}_{1}^{''}}$ defined like this:
\begin{multline}
\notag
\begin{aligned}
& \mathbb{N}_{0}^{''}\ =\ \{n \in \mathbb{N}^{''} \mid \delta(M_{n}, P_{n})\ \leq\ \lambda + \epsilon\},\\
& \mathbb{N}_{1}^{''}\ =\ \{n \in \mathbb{N}^{''} \mid \delta(M_{n}, P_{n})\ >\ \lambda + \epsilon\}\ =\ \mathbb{N}^{'} \setminus \mathbb{N}_{0}^{''}.
\end{aligned}
\end{multline}
Suppose that the set ${\mathbb{N}_{1}^{''}}$ is finite. In this case all large enough ${n \in \mathbb{N}^{'}}$ belong to ${\mathbb{N}_{0}^{''}}$ so we can set ${Q_{\epsilon} = \{\theta\}}$ and this will prove proposition 1. Thus we only need to consider the case when ${\mathbb{N}_{1}^{''}}$ is infinite. Since ${\delta(M_{n}, P_{n}) > \lambda + \epsilon}$ for all ${n \in \mathbb{N}_{1}^{''}}$, then for every ${n \in \mathbb{N}_{1}^{''}}$ there exists a vector ${x_{n} \in M_{n}}$ such that ${\left\|x_{n}\right\| = 1}$ and ${dist(x_{n}, P_{n}) > \lambda + \epsilon}$. Applying Hahn-Banach theorem to ${x_{n}}$ and ${P_{n}}$ we can find a functional ${f_{n} \in P_{n}^{\bot} \subset X^{*}}$ such that ${f_{n}x_{n} = 1}$ and ${\left\|f_{n}\right\| < \frac{1}{\lambda + \epsilon}}$. Therefore, denoting by ${S_{n}}$ the unit sphere from ${M_{n}}$, we may define for every ${n \in \mathbb{N}_{1}^{''}}$ a non-empty set
\begin{equation}\label{E:111}
T_{n} = \{f \in P_{n}^{\bot}\ \bigm| \ \left\|f\right\| \leq \frac{1}{\lambda + \epsilon}\ and \ \sup \{|fx|\ \bigm| \ x \in S_{n} \subset M_{n}\} \geq 1 \}.
\end{equation} 
Let us consider a sequence ${(V_{n})_{\mathbb{N}^{''}}}$ of unions of sets ${T_{i}}$ defined for every ${n \in \mathbb{N}_{1}^{''}}$:
\begin{equation}\label{E:93}
V_{n} = \bigcup_{i \geq n, i \in \mathbb{N}_{1}^{''}} T_{i}.
\end{equation}
Set ${r = \frac{1}{\lambda + \epsilon}}$ and ${\eta = \frac{\epsilon}{2}}$ -- obviously ${r \times (\lambda + \eta) < 1}$. Then we have all the conditions of the Null Functional Separation Lemma \ref{L:absnf} from which it follows that ${\theta \notin \overline{V_{n}}^{w^{*}}}$ for large enough ${n \in \mathbb{N}_{1}^{''}}$. Therefore, there exists a base neighborhood ${W \in w^{*}}$ of ${\theta \in X^{*}}$ and a number ${M \in \mathbb{N}}$ such that for ${n > M}$
\begin{equation}\label{E:97}
V_{n} \cap W = \emptyset.
\end{equation}
According to definition of the weak topology ${w^{*}}$ the base neighborhood $W$ is determined by a real number ${\nu > 0}$ and by a finite number of vectors ${v_{1}, ..., v_{k}}$ from $X$ such that ${f \in W}$ if and only if ${|fv_{i}| < \nu}$ for all ${i \in \{1,...,k\}}$. Therefore, due to \eqref{E:97}, for every ${n > M}$ and for every ${f \in V_{n}}$ there always exists some ${i \in \{1,...,k\}}$ such that ${|fv_{i}| \geq \nu > 0}$. Hence, defining a finite-dimensional space ${V = sp(\{v_{1},...,v_{k}\})}$, we conclude that ${V_{n} \cap V^{\bot} = \emptyset}$ for ${n > M}$, ${n \in \mathbb{N}_{1}^{''}}$. Therefore, due to \eqref{E:93}, for the same ${n}$
\begin{equation}\label{E:98}
T_{n} \cap V^{\bot} = \emptyset.
\end{equation}
Our goal now is to prove that ${\delta (M_{n}, V + P_{n}) \leq \lambda + \epsilon}$ for all ${n > M}$, ${n \in \mathbb{N}_{1}^{''}}$. Suppose the opposite -- then there shall exist a real number ${\gamma \in (0,\ 1 -(\lambda + \epsilon))}$, a natural number ${j > M}$, ${j \in \mathbb{N}_{1}^{''}}$ and a unit vector ${x_{j} \in M_{j}}$ such that
\begin{equation}\label{E:110}
d_{j} = dist(x_{j}, V + P_{j}) \geq \lambda + \epsilon + \gamma.
\end{equation}
Applying the Hahn-Banach theorem construct a functional ${f_{j}}$ such that ${f_{j}x_{j} = 1}$, ${\left\|f_{j}\right\| = \frac{1}{d_{j}}}$ and ${f_{j} \in (V + P_{j})^{\bot}}$. Obviously ${(V + P_{j})^{\bot} \subset V^{\bot}}$, therefore ${f_{j} \in V^{\bot}}$. Thus ${f_{j} \notin T_{j}}$ due to \eqref{E:98}. However, by our construction\linebreak[4] ${f_{j}x_{j} = 1}$ and ${f_{j} \in (V + P_{j})^{\bot} \subset P_{j}^{\bot}}$. Therefore, according to definition of ${T_{j}}$ from \eqref{E:111} the only way ${f_{j}}$ can be outside of ${T_{j}}$ is if ${\left\|f_{j}\right\| > \frac{1}{\lambda + \epsilon}}$, but it is impossible since by ${f_{j}}$ construction and by \eqref{E:110} the opposite is true: 
\[
\left\|f_{j}\right\| = \frac{1}{d_{j}} \leq \frac{1}{\lambda + \epsilon + \gamma} < \frac{1}{\lambda + \epsilon}.
\]
Thus we must conclude that our assumption is wrong and therefore\linebreak[4] ${\delta (M_{n}, V + P_{n}) \leq \lambda + \epsilon}$ for all ${n > M}$, ${n \in \mathbb{N}_{1}^{''}}$. Recall that ${V = sp(\{v_{1},...,v_{k}\})}$, therefore ${\dim V < \infty}$. Denoting ${Q_{\epsilon} = V}$, we conclude that for a given ${\epsilon > 0}$ a finite dimensional space ${Q_{\epsilon}}$ has been found such that for the sequence ${\mathbb{N}_{1}^{''}}$
\begin{equation}\label{E:115}
\lambda_{\mathbb{N}^{'}}[M_{n}, P_{n}] + \epsilon\ =\ \lambda + \epsilon\ \geq\ \varlimsup_{n \in \mathbb{N}_{1}^{''}} \delta(M_{n}, V + P_{n})\ =\ \varlimsup_{n \in \mathbb{N}_{1}^{''}} \delta(M_{n}, Q_{\epsilon} + P_{n}).
\end{equation}
Now recall that in the beginning of the proof we partitioned ${\mathbb{N}^{'}}$ into two subsets ${\mathbb{N}_{0}^{''}, \mathbb{N}_{1}^{''} \subset \mathbb{N}^{'}}$:
\begin{multline}\label{E:112}
\begin{aligned}
& \mathbb{N}_{0}^{''}\ =\ \{n \in \mathbb{N}^{''} \mid \delta(M_{n}, P_{n})\ \leq\ \lambda + \epsilon\},\\
& \mathbb{N}_{1}^{''}\ =\ \{n \in \mathbb{N}^{''} \mid \delta(M_{n}, P_{n})\ >\ \lambda + \epsilon\}\ =\ \mathbb{N}^{'} \setminus \mathbb{N}_{0}^{''}
\end{aligned}
\end{multline}
If ${\mathbb{N}_{0}^{''}}$ is finite, then inequality \eqref{E:115} can be rewritten by replacing ${\mathbb{N}_{1}^{''}}$ with ${\mathbb{N}^{'}}$
\[
\lambda_{\mathbb{N}^{'}}[M_{n}, P_{n}] + \epsilon\ \geq\ \varlimsup_{n \in \mathbb{N}^{'}} \delta(M_{n}, Q_{\epsilon} + P_{n})
\]
which obviously proves proposition 1. Now suppose that ${\mathbb{N}_{0}^{''}}$ is an infinite set. Note that ${\delta(M_{n}, P_{n})\ \geq\ \delta(M_{n}, Q_{\epsilon} + P_{n})}$ since ${P_{n} \subset Q_{\epsilon} + P_{n}}$. Therefore, the first line from \eqref{E:112} can be rewritten for the sequence ${\mathbb{N}_{0}^{''}}$
\begin{equation}\label{E:114}
\lambda_{\mathbb{N}^{'}}[M_{n}, P_{n}] + \epsilon\ =\ \lambda + \epsilon\ \geq\ \varlimsup_{n \in \mathbb{N}_{0}^{''}} \delta(M_{n}, P_{n})\ \geq\ \varlimsup_{n \in \mathbb{N}_{0}^{''}} \delta(M_{n}, Q_{\epsilon} + P_{n}).
\end{equation}
Therefore, as inequality \eqref{E:114} contains only ${\mathbb{N}_{0}^{''}}$ while inequality \eqref{E:115} contains only ${\mathbb{N}_{1}^{''}}$, we can combine inequalities \eqref{E:114} and \eqref{E:115} into a single inequality containing the whole set ${\mathbb{N}^{'}}$:
\[
\lambda_{\mathbb{N}^{'}}[M_{n}, P_{n}] + \epsilon\ \geq\ \varlimsup_{n \in \mathbb{N}^{'}} \delta(M_{n}, Q_{\epsilon} + P_{n}),
\]
which concludes our proof of the proposition 1.\\

In order to prove proposition 2 suppose that ${\lambda_{\mathbb{N}^{'}}[M_{n}, P_{n}] < 1}$; then, according to proposition 1, for any ${\epsilon \in (\lambda_{\mathbb{N}^{'}}[M_{n}, P_{n}],\ 1)}$ there exists a finite dimensional space ${Q_{\epsilon}}$ such that
\begin{equation}\label{193}
1\ >\ \lambda_{\mathbb{N}^{'}}[M_{n}, P_{n}] + \epsilon\ \geq\ \varlimsup_{n \in \mathbb{N}^{'}} \delta(M_{n}, Q_{\epsilon} + P_{n}).
\end{equation}
Therefore, it follows from the proposition 2 of the Small Gap Theorem that for large enough $n$
\[
\dim M_{n}\ \leq\ \dim (Q_{\epsilon} + P_{n})\ \leq\ \dim Q_{\epsilon} + \dim P_{n}.
\]
Denoting ${L = \dim Q_{\epsilon} < \infty}$ concludes the proof of the proposition 2.\\

In order to prove proposition 3 note that due to \eqref{193} and due to the proposition 3 of the Small Gap Theorem for large enough $n$
\[
\dim X / (Q_{\epsilon} + P_{n})\ \leq\ \dim X / M_{n}.
\]
Denoting ${L = \dim Q_{\epsilon} < \infty}$ we can rewrite the above inequality for the same large enough $n$:
\[
\dim X / P_{n}\ -\ L\ \leq\ \dim X / (Q_{\epsilon} + P_{n})\ \leq\ \dim X / M_{n},
\]
which proves proposition 3.\\

In order to prove proposition 4 note that implications ${(c) \Rightarrow (b) \Rightarrow (a)}$ are trivial so we only need to prove ${(a) \Rightarrow (d) \Rightarrow (c)}$.\\

Suppose that ${(a)}$ is true -- we need to prove ${(d)}$. Let ${S_{n}}$ be a unit sphere in ${P_{n}^{\bot}}$ for every ${n \in \mathbb{N}^{'}}$. Define
\begin{equation}\label{E:200}
T_{n} = S_{n},\ \ \ \ V_{n} = \bigcup_{i \geq n, i \in \mathbb{N}^{'}} T_{i}.
\end{equation} 
Also set ${r = 1}$ and choose some ${\eta \in (0, 1 - \lambda)}$ -- obviously ${r \times (\lambda + \eta) < 1}$. Then, since ${M_{n} = X}$ for all ${n \in \mathbb{N}^{'}}$, we can easily verify that 
\[
T_{n} = S_{n} = \{f \in P_{n}^{\bot}\ \bigm|\ \left\|f\right\| \leq r\ and\ \sup \{|fx|\ \bigm|\ x \in M_{n} = X, \left\|x\right\| = 1\} \geq 1 \}.
\]
Therefore all conditions of the Null Functional Separation Lemma \ref{L:absnf} are true, thus  it follows that ${\theta \notin \overline{V_{n}}^{w^{*}}}$ for large enough ${n \in \mathbb{N}^{'}}$. Therefore, there exists a base neighborhood ${W \in w^{*}}$ of ${\theta \in X^{*}}$ and a number ${M \in \mathbb{N}}$ such that for ${n > M, n \in \mathbb{N}^{'}}$
\begin{equation}\label{E:208}
V_{n} \cap W = \emptyset.
\end{equation}
According to definition of the weak topology ${w^{*}}$ the base neighborhood $W$ is determined by a real number ${\eta > 0}$ and by a finite number of vectors ${q_{1}, ..., q_{k}}$ from $X$ such that ${f \in W}$ if and only if ${|fq_{i}| < \eta}$ for all ${i \in \{1,...,k\}}$. Therefore, due to \eqref{E:208}, for every ${n > M}$ and for every ${f \in V_{n}}$ there always exists some ${i \in \{1,...,k\}}$ such that ${|fq_{i}| \geq \eta > 0}$. Hence, defining a finite-dimensional space ${Q = sp(\{q_{1},...,q_{k}\})}$, we conclude that ${V_{n} \cap Q^{\bot} = \emptyset}$ for ${n > M}$, ${n \in \mathbb{N}^{'}}$. Therefore, due to \eqref{E:200}, for the same ${n}$
\[
S_{n} \cap Q^{\bot} = \emptyset.
\]
Since ${S_{n}}$ are unit spheres from ${P_{n}}$, we shall conclude from the previous equation that ${P_{n}^{\bot} \cap Q^{\bot} = \{\theta\}}$ for ${n > M, n \in \mathbb{N}^{'}}$. Therefore ${X = Q + P_{n}}$ for the same $n$ -- since ${\dim Q \leq k}$ this proves the first part of the proposition ${(d)}$.

In order to prove the second part note that ${P_{n}^{\bot} \cap Q^{\bot} = \{\theta\}}$ for large enough $n$, hence
\[
\dim P_{n}^{\bot} \leq \dim X^{*}/Q^{\bot} = \dim Q \leq k < \infty.
\]
Therefore, all spaces ${Q^{\bot}}$, ${P_{n}^{\bot}}$ and ${Q^{\bot} \oplus P_{n}^{\bot}}$ are closed, the natural projections ${Pr_{n} : Q^{\bot} \oplus P_{n}^{\bot} \rightarrow Q^{\bot}}$ all exist and are continuous. Now suppose that the norms of projections ${Pr_{n}}$ are unbounded. Then there exists a subsequence of projections ${(Pr_{n})_{\mathbb{N}^{''}}}$ such that ${(\left\|Pr_{n}\right\|)_{\mathbb{N}^{''}} \rightarrow \infty}$. Therefore, there exists a sequence of unit functionals ${(f_{n})_{\mathbb{N}^{''}} \triangleleft (P_{n}^{\bot})_{\mathbb{N}^{''}}}$ such that ${dist(f_{n}, Q^{\bot}) \rightarrow 0}$ or, equivalently, that
\begin{equation}\label{E:777}
\left\|f_{n} - r_{n}\right\| \rightarrow 0
\end{equation}
for some sequence of functionals ${(r_{n})_{\mathbb{N}^{''}} \subset Q^{\bot}}$. Recall that we have proved that there exists some ${\nu > 0}$ such that for each ${f_{n}}$ from a unit sphere ${S_{n} \subset P_{n}^{\bot}}$ for large enough $n$ there always exists some ${k_{n}}$ such that ${|f_{n}q_{k_{n}}| \geq \nu > 0}$. Therefore, since there is only a finite number of vectors ${q_{1},...,q_{k}}$, we can choose a subsequence ${\mathbb{N}^{'''}}$ and some vector ${q_{j}}$ such that
\[
\varliminf_{n \in \mathbb{N}^{'''}} |f_{n}q_{j}| \geq \nu > 0.
\]
Now combining the previous inequality with \eqref{E:777} and recalling that ${r_{n}q_{j} = 0}$ since ${(r_{n})_{\mathbb{N}^{''}} \subset Q^{\bot}}$ we obtain a contradiction:
\begin{multline}
\notag
\begin{aligned}
0\ <\ \nu\ \leq\ \varliminf_{n \in \mathbb{N}^{'''}} |f_{n}q_{j}| &\ =\ \varliminf_{n \in \mathbb{N}^{'''}} |f_{n}q_{j} - r_{n}q_{j}|\\
                                                                  &\ \leq\ \varliminf_{n \in \mathbb{N}^{'''}} (\left\|f_{n} - r_{n}\right\| \times \left\|q_{j}\right\|)\\
                                                                  &\ =\ \left\|q_{j}\right\| \times \varliminf_{n \in \mathbb{N}^{'''}} (\left\|f_{n} - r_{n}\right\|)\\
                                                                  &\ =\ \left\|q_{j}\right\| \times 0\ =\ 0.
\end{aligned}
\end{multline}
This concludes our proof of the implication ${(a) \Rightarrow (d)}$.\\

Now let us prove implication ${(d) \Rightarrow (c)}$. We can safely assume that ${(d)}$ is true for all ${n \in \mathbb{N}^{'}}$. Since the norms of the projections ${Pr_{n} : Q^{\bot} \oplus P_{n}^{\bot} \rightarrow Q^{\bot}}$ are all bounded by some ${K < \infty}$ and since ${\dim X^{*}/Q^{\bot} = k < \infty}$, using the Hahn-Banach theorem up to $k$ times one can extend each projection\linebreak[4] ${Pr_{n} : {Q^{\bot} \oplus P_{n}^{\bot}\rightarrow Q^{\bot}}}$ toward a projection ${Pr_{n}^{'} : X^{*} \rightarrow Q^{\bot}}$ so that the norms of the extended projections ${Pr_{n}^{'}}$ are still bounded by ${K \times k}$. Let ${G_{n}}$ be kernels of ${Pr_{n}^{'}}$ -- define subspaces from $X$ like this
\[
H_{n} = \ ^{\bot}G_{n} = \bigcap_{f \in G_{n}} Ker(f).
\]
Obviously each ${H_{n}}$ is a closed subspace and ${H_{n} \subset P_{n}}$, also ${X = H_{n} \oplus Q}$ and the norms of natural projections ${W_{n} : X = H_{n} \oplus Q \rightarrow Q}$ are bounded by the same constant ${K \times k}$.

Now let ${(x_{n})_{\mathbb{N}^{'}}}$ be a bounded sequence of vectors from $X$. Decompose each ${x_{n} = W_{n}x_{n} + (I - W_{n})x_{n}}$ -- the sequence ${(W_{n}x_{n})_{\mathbb{N}^{'}} \subset Q}$ is bounded and therefore is relatively compact since ${\dim Q < \infty}$, while ${((I - W_{n})x_{n})_{\mathbb{N}^{'}} \triangleleft (P_{n})_{\mathbb{N}^{'}}}$ since ${((I - W_{n})x_{n})_{\mathbb{N}^{'}} \triangleleft (H_{n})_{\mathbb{N}^{'}}}$ and ${H_{n} \subset P_{n}}$. Denoting ${y_{n} = (I - W_{n})x_{n} \in P_{n}}$ we infer that ${(x_{n} - y_{n})_{\mathbb{N}^{'}}}$ is relatively compact. This concludes the proof of the implication ${(d) \Rightarrow (c)}$.
\end{proof}
Note that with some additional effort one can prove the following theorem that augments the $BFU$ Small Uniform Adjustment Theorem with stability of finite co-dimensions under simultaneous upper and lower uniform $\lambda-$adjustment of nested subspaces:
\begin{theorem}[Co-Dimensions of Nested Subspaces are Bounded]\label{cdnsb}
Let X be a $BFU-$space and ${P \subset Q}$ are two closed subspaces of $X$ such that ${\dim Q/P < \infty}$. Suppose that two sequences ${(P_{n})_{\mathbb{N}^{'}}}$ and ${(Q_{n})_{\mathbb{N}^{'}}}$  of closed subspaces from $X$ are such that ${P_{n} \subset Q_{n}}$ for all ${n \in \mathbb{N}^{'}}$; also suppose that uniform $\lambda-$adjustment numbers ${\lambda_{\mathbb{N}^{'}}[Q_{n}, Q]}$ and ${\lambda_{\mathbb{N}^{'}}[P, P_{n}]}$ are small enough. Then there exists a number ${K \in \mathbb{N}}$ such that ${\dim Q_{n} / P_{n} < K}$ for large enough ${n \in \mathbb{N}^{'}}$. 
\end{theorem}

\subsection[Upper Semi--Fredholm Stability of Subspaces in $BFU-$Spaces]{Upper Semi--Fredholm Stability of Subspaces in $BFU-$Spaces}\label{S:bfust}
Our previous Example \ref{Ex:4} had shown that there is no stability of any of the Fredholm properties under upper uniform $\lambda-$adjustment in ${l_{\infty}}$ even if ${\lambda = 0}$. However, our next theorem shows that in $BFU-$spaces upper semi--Fredholm pair of closed subspaces is stable under upper $\lambda-$adjustment for small $\lambda$.\\

In order to formulate that theorem consider some preliminary concepts. Let $M$, $N$ be two closed subspaces in a Banach space $X$ such that\linebreak[3] ${M + N = \overline{M + N}}$, ${X = \overline(M+N) \oplus S}$ for some closed subspace ${S \subset X}$. Define a natural mapping
\[
\Psi\ :\ \Omega\ =\ M\ \times\ N\ \times\ S\ \rightarrow\ X,
\]
\[
\Psi\ :\ (a, b, c)\ \rightarrow\ a + b + c.
\]
Also, define a complete norm on ${\Omega}$
\[
\left\|(a, b, c)\right\| = \max\{\left\|a\right\|, \left\|b\right\|, \left\|c\right\|\}.
\]
Then, since
\[
X\ =\ \overline{M + N}\ \oplus\ S\ =\ (M + N)\ \oplus\ S,
\]
it is clear that $\Psi$ is a continuous surjection from the Banach space ${\Omega}$ onto the Banach space ${X}$. Therefore, according to the open mapping theorem operator ${\Psi}$ is an open operator and there exists a real number ${\gamma > 0}$ such that any vector ${x \in X}$ can be decomposed into vectors ${a \in M}$, ${b \in N}$ and ${c \in S}$ so that
\[
\left\|x\right\| \times \gamma \ \geq\ \max\{\left\|a\right\|, \left\|b\right\|, \left\|c\right\|\}.
\]
Denote ${\psi_{S}(M, N)}$ the $\inf$ of such numbers $\gamma$.\\

Note that for an upper semi--Fredholm pair there always exist many closed subspaces $S$ such that ${X = (M+N) \oplus S}$ if ${\beta(M,N) = \dim X/(M+N) < \infty}$. Therefore, denote
\[
\psi(M, N)\ :=\ \inf \{ \psi_{S}(M, N)\ \mid\ X\ =\ (M+N) \oplus S\}.
\]
We are now ready to formulate our next stability theorem: 
\begin{theorem}[Upper Semi--Fredholm Pairs from $BFU-$Spaces are Stable]
Let ${(M,N)}$ be an upper semi--Fredholm pair of closed subspaces in a $BFU-$space $X$. Let ${(M_{n})_{\mathbb{N}^{'}}}$, ${(N_{n})_{\mathbb{N}^{'}}}$ are two sequences of closed subspaces from $X$, and set
\[
\lambda_{M} = \lambda_{\mathbb{N}^{'}}[M, M_{n}],\ \ \ \lambda_{N} = \lambda_{\mathbb{N}^{'}}[N, N_{n}];
\]
then the following propositions are true
\begin{enumerate}
  \item Define a real number
\[
\omega = \psi(M, N) \times (\lambda_{M} + \lambda_{N}).
\]  
If ${\omega < 1}$, then for large enough ${n \in \mathbb{N}^{'}}$ pairs ${(M_{n},N_{n})}$ are also upper semi--Fredholm and
\[
\varlimsup_{n \in\mathbb{N}^{'}} \beta (M_{n}, N_{n})\ <\ \infty.
\]
  \item Suppose that ${X = (M + N) \oplus S}$ and define a real number
\[
\omega_{S}(M) = \psi_{S}(M, N) \times (\lambda_{M} + \lambda_{N}).
\]  
If ${\omega_{S}(M) < 1}$ and ${(W_{n})_{\mathbb{N^{''}}}}$ is a sequence of non-null closed subspaces from ${X^{*}}$ such that ${W_{n} \subset M_{n}^{\bot}}$ for all ${n \in \mathbb{N}^{'}}$ and ${\delta(W_{n}, N_{n}^{\bot}) \rightarrow 0}$, then there exists ${K \in \mathbb{N}}$ such that ${\dim W_{n} \leq K < \infty}$ for large enough ${n \in \mathbb{N}^{'}}$
\end{enumerate}
\end{theorem}
\begin{proof}
Proposition 2 is the key to the proof of proposition 1, so we start with 2. Let ${S_{n}}$ be a unit sphere from ${W_{n}}$ for each ${n \in \mathbb{N}^{'}}$. Define
\[
V_{n} = \bigcup_{i \geq n} S_{i}.
\]
Our first goal is to prove that ${\theta \notin \overline{V_{n}}^{w^{*}}}$ for some ${n \in \mathbb{N}^{'}}$. Assume the opposite and consider the unit ball ${B_{1} \subset X^{*}}$ centered in ${\theta \in X^{*}}$. Since ${B_{1}}$ is norm-bounded it is compact in ${w^{*}}$ topology by the Banach-Alaoglu theorem. Also, since $X$ is a $BFU-$space, ${(B_{1}, w^{*})}$ is Fr\'{e}chet-Urysohn. Therefore, according to Lemma \ref{L:css} there exists a subsequence of functionals ${(g_{n})_{\mathbb{N}_{1}^{'}} \triangleleft (V_{n})_{\mathbb{N}_{1}^{'}}}$ convergent to $\theta$ in ${w^{*}}$ topology. Then, by Lemma \ref{L:ccus} there exists a subsequence ${(f_{n})_{\mathbb{N}^{''}} \triangleleft (S_{n})_{\mathbb{N}^{''}}}$ convergent to $\theta$ in ${w^{*}}$ topology. Since each ${f_{n}}$ is a unit functional, there exists a unit sequence of vectors ${(x_{n})_{\mathbb{N}^{''}} \subset X}$ such that ${f_{n}x_{n} = 1}$. For each ${x_{n}}$ consider its decomposition
\begin{equation}\label{E:291}
x_{n} = a_{n} + b_{n} + c_{n},\ \ a_{n} \in M_{n},\ b_{n} \in N_{n},\ c_{n} \in S
\end{equation}
such that
\begin{equation}\label{E:290}
\psi_{S}(M, N)\ =\ 1 \times \psi_{S}(M, N)\ =\ \left\|x_{n}\right\| \times \psi_{S}(M, N)\ \geq\ max\{\left\|a_{n}\right\|, \left\|b_{n}\right\|, \left\|c_{n}\right\|\}.
\end{equation}
There are four cases to consider in respect to values ${a_{n}}$ and ${b_{n}}$:
\begin{enumerate}
  \item There exists a subsequence ${\mathbb{N}^{'''}}$ such that all ${\alpha_{n} = \left\|a_{n}\right\|}$ and ${\beta_{n} = \left\|b_{n}\right\|}$ are greater than zero and ${\alpha_{n} \rightarrow \alpha \in \mathbb{R}}$, ${\beta_{n} \rightarrow \beta \in \mathbb{R}}$. Choose a positive number $\eta$ from an open interval
\begin{equation}\label{E:299}
\eta \in (0,\ \frac{1}{2} \times (\frac{1}{\psi_{S}(M, N)} - \lambda_{M} - \lambda_{N}));
\end{equation}
by definition of $\lambda-$adjustment there exist two sequences of vectors\linebreak[4] ${(u_{n})_{\mathbb{N}^{'4}} \triangleleft (M_{n})_{\mathbb{N}^{'4}}}$ and ${(v_{n})_{\mathbb{N}^{'4}} \triangleleft (N_{n})_{\mathbb{N}^{'4}}}$ and vectors ${u, v \in X}$ such that
\begin{multline}
\notag
\begin{aligned}
\varlimsup_{\mathbb{N}^{'4}} \left\|\frac{1}{\alpha_{n}} \times a_{n} - u_{n} - u\right\| & \ \leq\ \lambda_{M} + \eta,\\
\varlimsup_{\mathbb{N}^{'4}} \left\|\frac{1}{\beta_{n}} \times b_{n} - v_{n} - v\right\|  & \ \leq\ \lambda_{N} + \eta.
\end{aligned}
\end{multline}
Multiplying first inequality by ${\alpha_{n}}$ and second by ${\beta_{n}}$, taking into account estimate from \eqref{E:290} and limits ${\alpha_{n} \rightarrow \alpha \in \mathbb{R}}$, ${\beta_{n} \rightarrow \beta \in \mathbb{R}}$ obtain
\begin{multline}\label{E:294}
\begin{aligned}
\varlimsup_{\mathbb{N}^{'4}} \left\|a_{n} - \alpha  u_{n} - \alpha u\right\| &\ \leq\ \alpha \times (\lambda_{M} + \eta)\ &\leq\ \psi_{S}(M, N) \times (\lambda_{M} + \eta),\\
\varlimsup_{\mathbb{N}^{'4}} \left\|b_{n} - \beta v_{n} - \beta v\right\| &\ \leq\ \beta \times (\lambda_{M} + \eta)\ &\leq\ \psi_{S}(M, N) \times (\lambda_{N} + \eta).
\end{aligned}
\end{multline}
  \item For any subsequence ${\mathbb{N}^{'''}}$ such that all ${\alpha_{n} = \left\|a_{n}\right\| > 0}$ the opposite is true for ${\beta_{n}}$: ${\beta_{n} = \left\|b_{n}\right\| = 0}$. Applying definition of $\lambda-$adjustment as before obtain the first inequality from \eqref{E:294}. Since ${\beta_{n} = \left\|b_{n}\right\| = 0}$, the second inequality from \eqref{E:294} trivially holds for any choice of bounded sequence ${(v_{n})_{\mathbb{N}^{'4}}}$ and vector ${v}$.
  \item For any subsequence ${\mathbb{N}^{'''}}$ such that all ${\beta_{n} = \left\|b_{n}\right\| > 0}$ the opposite is true for ${\alpha_{n}}$:  ${\alpha_{n} = \left\|a_{n}\right\| = 0}$. This case is treated as above by symmetry -- applying definition of $\lambda-$adjustment as before obtain the second inequality from \eqref{E:294}. Since ${\alpha_{n} = \left\|a_{n}\right\| = 0}$, the first inequality from \eqref{E:294} trivially holds for any choice of bounded sequence ${(u_{n})_{\mathbb{N}^{'4}}}$ and vector ${u}$.
  \item For large enough ${n \in \mathbb{N}^{'}}$ all elements ${a_{n}}$, and ${b_{n}}$ are null. In this case both inequalities from \eqref{E:294} trivially hold for any choice of bounded sequences ${(u_{n})_{\mathbb{N}^{'4}}}$, ${(v_{n})_{\mathbb{N}^{'4}}}$ and vectors ${u}$, ${v}$.
\end{enumerate}
We have just proved that \eqref{E:294} is true in all possible cases.
Now, taking into account decomposition \eqref{E:291}, the inequalities from \eqref{E:294}, the choice of $\eta$ form \eqref{E:299} and the triangle inequality estimate
\begin{multline}
\notag
\begin{aligned}
& \varlimsup_{\mathbb{N}^{'4}} \left\|x_{n} - c_{n} - \alpha u_{n} - \beta v_{n} - \alpha u - \beta v \right\| =\\
&\ \ \ \ \ \ \ \ \ \ \ \ \ \ \ \ \ \ \ \ \ \ =\ \varlimsup_{\mathbb{N}^{'4}} \left\|a_{n} + b_{n} - \alpha u_{n} - \beta v_{n} - \alpha u - \beta v \right\|\\
&\ \ \ \ \ \ \ \ \ \ \ \ \ \ \ \ \ \ \ \ \ \ =\ \varlimsup_{\mathbb{N}^{'4}} \left\|a_{n} - \alpha u_{n} - \alpha u + b_{n} - \beta v_{n}  - \beta v \right\|\\
&\ \ \ \ \ \ \ \ \ \ \ \ \ \ \ \ \ \ \ \ \ \ \leq\ \varlimsup_{\mathbb{N}^{'4}} \left\|a_{n} - \alpha  u_{n} - \alpha u\right\|\ +\ \left\|b_{n} - \beta v_{n} - \beta v\right\|\\
&\ \ \ \ \ \ \ \ \ \ \ \ \ \ \ \ \ \ \ \ \ \ \leq\ \psi_{S}(M, N) \times (\lambda_{M} + \eta) + \psi_{S}(M, N) \times (\lambda_{N} + \eta)\\
&\ \ \ \ \ \ \ \ \ \ \ \ \ \ \ \ \ \ \ \ \ \ =\ \psi_{S}(M, N) \times (\lambda_{M} + \lambda_{N} + 2 \times \eta)\ <\ 1.
\end{aligned}
\end{multline}
Recall that ${\dim S < \infty}$ and sequence ${(c_{n})_{\mathbb{N}^{'4}}}$ is bounded, so there exists a subsequence ${\mathbb{N}^{'5}}$ such that ${(c_{n})_{\mathbb{N}^{'5}} \rightarrow c \in S}$. Therefore, since ${(f_{n})_{\mathbb{N}^{''}} \triangleleft (S_{n})_{\mathbb{N}^{''}}}$ converges to $\theta$ in ${w^{*}}$ topology, it is clear that ${(f_{n}c_{n})_{\mathbb{N}^{'5}} \rightarrow \lim_{n \in \mathbb{N}^{'5}} f_{n}c = 0}$. Taking this into account, and also recalling that ${f_{n}u_{n} = 0}$ since each ${f_{n}}$ belongs to annihilator of ${M_{n}}$, estimate
\begin{multline}
\notag
\begin{aligned}
& \varlimsup_{\mathbb{N}^{'5}} |f_{n}x_{n} - \beta f_{n} v_{n} - f_{n}(\alpha u + \beta v) | =\\ 
&\ \ \ \ \ \ \ \ \ \ \ =\ \varlimsup_{\mathbb{N}^{'5}} |f_{n}x_{n} - f_{n}c_{n} + f_{n}c_{n} - \alpha f_{n} u_{n} - \beta f_{n} v_{n} - f_{n}(\alpha u + \beta v) |\\
                                                                                            &\ \ \ \ \ \ \ \ \ \ \ =\ \varlimsup_{\mathbb{N}^{'5}} |f_{n}(x_{n} - c_{n} - \alpha u_{n} - \beta v_{n} - \alpha u - \beta v) + f_{n}c_{n}|\\
                                                                                            &\ \ \ \ \ \ \ \ \ \ \ \leq\ \varlimsup_{\mathbb{N}^{'5}} \left\|f_{n}\right\| \times \left\|x_{n} - c_{n} - \alpha u_{n} - \beta v_{n} - \alpha u - \beta v \right\| + \varlimsup_{\mathbb{N}^{'5}} |f_{n}c_{n}|\\
                                                                                            &\ \ \ \ \ \ \ \ \ \ \ =\ 1 \times \varlimsup_{\mathbb{N}^{'5}} \left\|x_{n} - c_{n} - \alpha u_{n} - \beta v_{n} - \alpha u - \beta v \right\| + 0\\
                                                                                            &\ \ \ \ \ \ \ \ \ \ \ <\ 1 \times 1\ =\ 1.
\end{aligned}
\end{multline}
By construction ${f_{n}x_{n} = 1}$, therefore, using the previous inequality and the triangle inequality two times, estimate
\begin{multline}\label{E:293}
\begin{aligned}
1 &\ =\ \varlimsup_{\mathbb{N}^{'5}} f_{n}x_{n}\\
  &\ =\ \varlimsup_{\mathbb{N}^{'5}} |f_{n}x_{n} - \beta f_{n} v_{n} - f_{n}(\alpha u + \beta v) + \beta f_{n} v_{n} + f_{n}(\alpha u + \beta v)|\\
  &\ \leq\ \varlimsup_{\mathbb{N}^{'5}} |f_{n}x_{n} - \beta f_{n} v_{n} - f_{n}(\alpha u + \beta v)| + \varlimsup_{\mathbb{N}^{'5}} |\beta f_{n} v_{n} + f_{n}(\alpha u + \beta v)|\\
  &\ <\ 1 + \varlimsup_{\mathbb{N}^{'5}} |\beta f_{n} v_{n} + f_{n}(\alpha u + \beta v)|\\
  &\ <\ 1 + \varlimsup_{\mathbb{N}^{'5}} |\beta f_{n} v_{n}| + \varlimsup_{\mathbb{N}^{'5}} |f_{n}(\alpha u + \beta v)|
\end{aligned}
\end{multline}
Now recall that ${(f_{n})_{\mathbb{n}^{'5}} \rightarrow \theta}$ in ${w^{*}}$ topology, therefore
\[
\varlimsup_{\mathbb{N}^{'5}} |f_{n}(\alpha u + \beta v)| = 0.
\]
Also recall that ${\delta(W_{n}, N_{n}^{\bot}) \rightarrow 0}$ -- therefore, since ${f_{n}}$ are all unit functionals, there exists a sequence of functionals ${(h_{n})_{\mathbb{N}^{'5}} \triangleleft (N_{n}^{\bot})_{\mathbb{N}^{'5}}}$ such that
\[
\lim_{\mathbb{N}^{'5}} \left\|f_{n} + h_{n}\right\| = 0.
\]
Therefore, recalling from \eqref{E:294} that all ${v_{n}}$ are bounded by some ${K < \infty}$, we can further estimate \eqref{E:293} to an obvious contradiction
\begin{multline}
\notag
\begin{aligned}
1 &\ <\ 1 + \varlimsup_{\mathbb{N}^{'5}} |\beta f_{n} v_{n}| + \varlimsup_{\mathbb{N}^{'5}} |f_{n}(\alpha u + \beta v)|\\
  &\ =\ 1 + \varlimsup_{\mathbb{N}^{'5}} |\beta f_{n} v_{n}| + 0\\
  &\ =\ 1 + \varlimsup_{\mathbb{N}^{'5}} |\beta f_{n} v_{n} + \theta|\\
  &\ =\ 1 + \varlimsup_{\mathbb{N}^{'5}} |\beta f_{n} v_{n} + \beta h_{n} v_{n}|\\
  &\ \leq\ 1 + \varlimsup_{\mathbb{N}^{'5}} |\beta| \times \left\|f_{n} + h_{n}\right\| \times \left\|v_{n}\right\|\\
  &\ \leq\ 1 + \varlimsup_{\mathbb{N}^{'5}} |\beta| \times \left\|f_{n} + h_{n}\right\| \times K\\
  &\ =\ 1 + |\beta| \times K \times \varlimsup_{\mathbb{N}^{'5}} \left\|f_{n} + h_{n}\right\|\\
  &\ \leq\ 1 + |\beta| \times K \times 0\ =\ 1 + 0\ =\ 1.
\end{aligned}
\end{multline}
Therefore, our assumption is not correct and we shall conclude that ${\theta \notin \overline{V_{m}}^{w^{*}}}$ for some ${m \in \mathbb{N}^{'}}$. Hence, there exists a base neighborhood of ${\theta \in X^{*}}$ such that for the same ${m}$
\begin{equation}\label{E:295}
V_{m} \cap W = \emptyset.
\end{equation}
According to definition of the weak topology ${w^{*}}$ the base neighborhood $W$ is determined by a real number ${\eta > 0}$ and by a finite number of vectors ${q_{1}, ..., q_{k}}$ from $X$ such that ${f \in W}$ if and only if ${|fq_{i}| < \nu}$ for all ${i \in \{1,...,k\}}$. Therefore, due to \eqref{E:295}, for every ${f \in V_{m}}$ there always exists some ${i \in \{1,...,k\}}$ such that ${|fq_{i}| \geq \nu > 0}$. Hence, defining a finite-dimensional space ${Q = sp(\{q_{1},...,q_{k}\})}$, we conclude that ${V_{m} \cap Q^{\bot} = \emptyset}$. Therefore, due to \eqref{E:295} and recalling the definition of ${V_{n}}$ obtain for all ${n > m, n \in \mathbb{N}^{'}}$
\[
S_{n} \cap Q^{\bot} = \emptyset.
\]
Since ${S_{n}}$ are unit spheres from ${W_{n}}$, the previous equation yields ${W_{n} \cap Q^{\bot} = \{\theta\}}$ for ${n > m, n \in \mathbb{N}^{'}}$. Since ${\dim X^{*} / Q^{\bot} \leq k < \infty}$ it follows that ${\dim W_{n} \leq k < \infty}$ for the same $n$ which concludes the proof of proposition 2.\\

In order to prove proposition 1 we first prove that there exists a closed finite-dimensional subspace ${S \subset X}$ such that ${X = (M+N) \oplus S}$ and 
\begin{equation}\label{E:297}
\omega_{S} = \psi_{S}(M, N) \times (\lambda_{M} + \lambda_{N}) < 1.
\end{equation}
Consider two cases:
\begin{itemize}
\item ${\beta(M,N) = 0}$. In this case ${X = (M+N) \oplus \{\theta\}}$ and ${\psi_{\{\theta\}}(M, N) = \psi(M, N)}$. Therefore, by the condition of proposition 1
\[
\omega_{\{\theta\}} = \psi_{\{\theta\}}(M, N) \times (\lambda_{M} + \lambda_{N}) = \psi(M, N) \times (\lambda_{M} + \lambda_{N}) = \omega < 1.
\]
\item ${\beta(M,N) > 0}$. In this case, since
\[
\omega = \psi(M, N) \times (\lambda_{M} + \lambda_{N}) < 1,\ \ \ \psi(M, N) < \frac{1}{\lambda_{M} + \lambda_{N}},
\]
by definition of ${\psi(M, N)}$ one can find a closed finite-dimensional subspace ${S \subset X}$ such that ${X = (M+N) \oplus S}$ and ${\psi(M, N) \leq \psi_{S}(M, N) < \frac{1}{\lambda_{M} + \lambda_{N}}}$. Therefore, \eqref{E:297} holds.
\end{itemize}
Thus, in any case there exists a closed finite-dimensional subspace ${S \subset X}$ such that ${X = (M+N) \oplus S}$ and \eqref{E:297} holds.\\

Now suppose that there exists a subsequence ${\mathbb{N}^{''}}$ such that\linebreak[4] ${M_{n}^{\bot} + N_{n}^{\bot} \neq \overline{(M_{n}^{\bot} + N_{n}^{\bot})}}$ for all ${n \in \mathbb{N}^{''}}$. Then, according to Theorem 4.19 from \cite{kato} there exists a sequence of infinite-dimensional closed subspaces ${W_{n} \subset M_{n}^{\bot}}$ such that ${\delta(W_{n}, N_{n}^{\bot}) \rightarrow 0}$. However, according to proposition 2 such subspaces ${W_{n}}$ must have limited finite dimension for large enough ${n \in \mathbb{N}^{''}}$. Therefore, our assumption is incorrect and ${M_{n}^{\bot} + N_{n}^{\bot} = \overline{(M_{n}^{\bot} + N_{n}^{\bot})}}$ for large enough ${n \in \mathbb{N}^{'}}$.

Now consider spaces ${W_{n} = M_{n}^{\bot} \cap N_{n}^{\bot}}$ -- obviously ${W_{n} \subset M_{n}^{\bot}}$, also\linebreak[4] ${\delta(W_{n}, N_{n}^{\bot}) = 0}$ since ${W_{n} \subset N_{n}^{\bot}}$. Again, according to proposition 2 there exists ${K \in \mathbb{N}}$ such that ${\dim W_{n} = \dim M_{n}^{\bot} \cap N_{n}^{\bot} < K < \infty}$ for large enough ${n \in \mathbb{N}^{'}}$.

Thus, we have just proved that pairs ${(M_{n}^{\bot}, N_{n}^{\bot})}$ are lower semi--Fredholm, and ${\alpha(M_{n}^{\bot}, N_{n}^{\bot}) < K < \infty}$  for large enough ${n \in \mathbb{N}^{'}}$. Then according to Theorem 4.8 from \cite{kato} pairs ${(M_{n}, N_{n})}$ are upper semi--Fredholm and ${\beta(M_{n}, N_{n}) < K < \infty}$ for the same ${n \in \mathbb{N}^{'}}$. This concludes the proof of the proposition 1 and of the entire theorem.
\end{proof}

Note that using Theorem \ref{cdnsb} and with some additional effort one can prove the following theorem that extends the $BFU$ semi--Fredholm stability theorems to cases when regular defect numbers $\alpha$ and $\beta$ are infinite:
\begin{theorem}[Extended $BFU$ Stability of (Semi--)Fredholm Pairs]\label{esssfp}
Let X be a $BFU-$space and ${M}$, $N$ are two closed subspaces of $X$ such that their sum is closed ${\overline{M+N} = M+N}$. Let ${(M_{n})_{\mathbb{N}^{'}}}$ and ${(N_{n})_{\mathbb{N}^{'}}}$ be two sequences of closed subspaces from $X$. Then the following propositions are true:
\begin{enumerate}
  \item Suppose that ${Q}$ is a closed subspace from $X$ containing both ${M}$ and $N$ such that ${\dim Q/(M+N) < \infty}$. Also suppose that ${(Q_{n})_{\mathbb{N}^{'}}}$ is a sequence of closed subspaces from $X$ such that $Q_{n}$ contains both ${M_{n}}$ and ${N_{n}}$, i.e. ${\subset M_{n} + N_{n} \subset Q_{n}}$ for all ${n \in \mathbb{N}^{'}}$. Further suppose that uniform $\lambda-$adjustment numbers ${\lambda_{\mathbb{N}^{'}}[Q_{n}, Q]}$, ${\lambda_{\mathbb{N}^{'}}[M, M_{n}]}$ and ${\lambda_{\mathbb{N}^{'}}[N, N_{n}]}$ are small enough. Then for large enough ${n \in \mathbb{N}^{'}}$ the sums of ${M_{n}}$ and ${N_{n}}$ are closed ${\overline{M_{n}+N_{n}} = M_{n}+N_{n}}$ and there exists a number ${K \in \mathbb{N}}$ such that\linebreak[4] ${\dim Q_{n} / (M_{n} + N_{n}) < K}$ for the same $n$.
  \item Suppose that ${P}$ is a closed subspace from $X$ contained in both ${M}$ and $N$ such that ${\dim (M \cap N) /P < \infty}$. Also suppose that ${(P_{n})_{\mathbb{N}^{'}}}$ is a sequence of closed subspaces from $X$ such that $P_{n}$ is contained in both ${M_{n}}$ and ${N_{n}}$, i.e. ${P_{n} \subset M_{n} \cap N_{n}}$ for all ${n \in \mathbb{N}^{'}}$. Further suppose that uniform $\lambda-$adjustment numbers ${\lambda_{\mathbb{N}^{'}}[P, P_{n}]}$, ${\lambda_{\mathbb{N}^{'}}[M_{n}, M]}$ and ${\lambda_{\mathbb{N}^{'}}[N_{n}, N]}$ are small enough. Then for large enough ${n \in \mathbb{N}^{'}}$ the sums of ${M_{n}}$ and ${N_{n}}$ are closed ${\overline{M_{n}+N_{n}} = M_{n}+N_{n}}$ and there exists a number ${K \in \mathbb{N}}$ such that\linebreak[4] ${\dim (M_{n} \cap N_{n}) / P_{n} < K}$ for the same $n$.
\end{enumerate} 
\end{theorem}

\subsection[Upper Stability of Index of Continuous Operators in $BFU-$spaces]{Upper Stability of Index of Continuous Operators in $BFU-$spaces}\label{soioclobfu}
It is easy to see that index of Fredholm pair of subspaces is not stable upper uniform $\lambda-$adjustment -- any two pairs of subspaces in a finite-dimensional space are uniformly $0-$adjusted yet their indices may obviously differ. However, for linear operators in $BFU-$spaces their index is stable under perturbation by upper uniformly $0-$adjusted continuous operators. In order to prove that we first establish a technical lemma that uncovers the structure of upper uniform $0-$adjusted operators in $BFU-$spaces.

\begin{lemma}[The Structure of Upper Uniform $0-$Adjustment]\label{L:suua}
Let $X$ and $Y$ be two Banach spaces, ${(A_{n})_{\mathbb{N}^{'}} \subset \mathcal{B}(X,Y)}$ a sequence of continuous operators and ${\lambda_{\mathbb{N}^{'}}[\theta, A_{n}] = 0}$. Then the following propositions are true:
\begin{enumerate}
  \item Suppose that two sequences of functionals ${(f_{n})_{\mathbb{N}^{'}} \subset Y^{*}}$ and ${(f_{n}A_{n})_{\mathbb{N}^{'}} \subset X^{*}}$ are both point wise converging to zero. Then ${\left\|f_{n}A_{n}\right\| \rightarrow 0}$.
  \item Suppose that ${X \times Y}$ is a $BFU-$space. Then for every ${\epsilon \in (0, 1)}$ there exists a closed finite co-dimensional subspace ${R \subset Y^{*}}$, ${\dim Y^{*} / R < \infty}$ such that for large enough ${n \in \mathbb{N}^{'}}$
\[
\left\|A^{*}_{n} \mid {}_{R} \right\| \leq \epsilon.
\]
  \item ${\left\|A_{n}\right\|_{q} \rightarrow 0}$.
\end{enumerate}
\end{lemma}
\begin{proof}
In order to prove proposition 1 assume the opposite -- then there exists a unit sequence ${(x_{n})_{\mathbb{N}^{''}} \subset X}$ and some ${\epsilon > 0}$ such that ${\left\|f_{n}A_{n}x_{n}\right\| > \epsilon}$ for all ${n \in \mathbb{N}^{''}}$. Due to the point wise convergence condition, norms of all functionals ${(f_{n})_{\mathbb{N}^{'}} \subset Y^{*}}$ and ${(f_{n}A_{n})_{\mathbb{N}^{'}} \subset X^{*}}$ are limited by some ${K \in \mathbb{R}}$. 
By definition of $\lambda-$adjustment there exists a sequence ${(y_{n})_{\mathbb{N}^{'''}} \subset X}$ and vectors ${t \in X}$, ${u \in Y}$ such that
\[
\varlimsup_{\mathbb{N}^{'''}} \left\| (x_{n} - y_{n} - z,\ A_{n}y_{n} - t) \right\| \leq \frac{\epsilon}{2 \times K}
\]
Applying functionals ${f_{n}A_{n}}$ to the first component and taking into account that ${f_{n}A_{n}z \rightarrow 0}$ obtain
\begin{multline}
\notag
\begin{aligned}
\varlimsup_{\mathbb{N}^{'''}} \left\|f_{n}A_{n}x_{n} - f_{n}A_{n}y_{n}\right\| &\ =\ \varlimsup_{\mathbb{N}^{'''}} \left\|f_{n}A_{n}x_{n} - f_{n}A_{n}y_{n} - f_{n}A_{n}z + f_{n}A_{n}z\right\|\\
                                                                               &\ \leq\ \varlimsup_{\mathbb{N}^{'''}} \left\|f_{n}A_{n}x_{n} - f_{n}A_{n}y_{n} - f_{n}A_{n}z\right\| + \varlimsup_{\mathbb{N}^{'''}} \left\|f_{n}A_{n}z\right\|\\
                                                                               &\ =\ \varlimsup_{\mathbb{N}^{'''}} \left\|f_{n}A_{n}x_{n} - f_{n}A_{n}y_{n} - f_{n}A_{n}z\right\| + 0\\
                                                                               &\ =\ \varlimsup_{\mathbb{N}^{'''}} \left\|f_{n}A_{n}(x_{n} - y_{n} - z)\right\|\\
                                                                               &\ \leq\  \varlimsup_{\mathbb{N}^{'''}} (\left\|f_{n}A_{n}\right\| \times \left\|x_{n} - y_{n} - z\right\|)\\
                                                                               &\ \leq\ \varlimsup_{\mathbb{N}^{'''}} \left\|f_{n}A_{n}\right\| \times \frac{\epsilon}{2 \times K}\ \leq\ K \times \frac{\epsilon}{2 \times K}\ =\ \frac{\epsilon}{2}.
\end{aligned}
\end{multline}
Applying functionals ${f_{n}}$ to the second component and taking into account that ${f_{n}t \rightarrow 0}$ estimate
\begin{multline}
\notag
\begin{aligned}
\varlimsup_{\mathbb{N}^{'''}} \left\|f_{n}A_{n}y_{n}\right\| &\ =\ \varlimsup_{\mathbb{N}^{'''}} \left\|f_{n}A_{n}y_{n} - f_{n}t + f_{n}t\right\|\\
                                                             &\ \leq\ \varlimsup_{\mathbb{N}^{'''}} \left\|f_{n}A_{n}y_{n} - f_{n}t\right\| + \varlimsup_{\mathbb{N}^{'''}} \left\|f_{n}t\right\|\\
                                                             &\ =\ \varlimsup_{\mathbb{N}^{'''}} \left\|f_{n}A_{n}y_{n} - f_{n}t\right\| + 0\\
                                                             &\ \leq\  \varlimsup_{\mathbb{N}^{'''}} (\left\|f_{n}\right\| \times \left\|A_{n}y_{n} - t\right\|)\\
                                                             &\ \leq\ \varlimsup_{\mathbb{N}^{'''}} \left\|f_{n}\right\| \times \frac{\epsilon}{2 \times K}\ \leq\ K \times \frac{\epsilon}{2 \times K}\ =\ \frac{\epsilon}{2}.
\end{aligned}
\end{multline}
Now taking into account the previous two inequalities and using the triangle inequality estimate further
\begin{multline}
\notag
\begin{aligned}
\varlimsup_{\mathbb{N}^{'''}} \left\|f_{n}A_{n}x_{n}\right\| &\ =\ \varlimsup_{\mathbb{N}^{'''}} \left\|f_{n}A_{n}x_{n} - f_{n}A_{n}y_{n} + f_{n}A_{n}y_{n}\right\|\\
                                                             &\ \leq\ \varlimsup_{\mathbb{N}^{'''}} \left\|f_{n}A_{n}x_{n} - f_{n}A_{n}y_{n}\right\| + \left\|f_{n}A_{n}y_{n}\right\|\\
                                                             &\ \leq\ \frac{\epsilon}{2} + \frac{\epsilon}{2}\ =\ \epsilon
\end{aligned}
\end{multline}
which contradicts our assumption that ${\left\|f_{n}A_{n}x_{n}\right\| > \epsilon}$ for all ${n \in \mathbb{N}^{''}}$. Therefore, ${\left\|f_{n}A_{n}\right\| \rightarrow 0}$ which concludes the proof of the proposition 1.\\

In order to prove proposition 2 choose ${\epsilon \in (0, 1)}$ and define sets of functionals for every ${n \in \mathbb{N}^{'}}$ like this:
\[
T_{n}\ =\ \{(fA_{n}, f) \in (X \times Y)^{*}\ \bigm|\ \left\|f\right\|\ =\ 1,\ \left\|fA_{n}\right\|\ >\ \epsilon\}
\]
Suppose that for large enough ${n \in \mathbb{N}^{'}}$ sets ${T_{n}}$ are empty. Then proposition 2 is true as we can choose ${Q}$ being the whole space ${Y^{*}}$. We are now left with the case when there is a subsequence ${\mathbb{N}^{''}}$ such that ${T_{n}}$ is not empty for every ${n \in \mathbb{N}^{''}}$. In this case consider the following sets for the same $n$:
\[
G_{n}\ =\ \{(\frac{fA_{n}}{\left\|fA_{n}\right\|}, \frac{f}{\left\|fA_{n}\right\|}) \in (X \times Y)^{*}\ \bigm|\  \left\|f\right\|\ =\ 1,\ \left\|fA_{n}\right\|\ >\ \epsilon\}.
\] 
Sets ${G_{n}}$ are not empty since sets ${T_{n}}$ are not empty. Now for each ${n \in \mathbb{N}^{''}}$ define sets
\[
V_{n}\ =\ \bigcup_{i \geq n} G_{i}.
\]
Denote by ${(\theta, \theta)}$ the null element from ${(X \times Y)^{*}}$ and suppose that ${(\theta, \theta) \in {\overline{V_{n}}}^{w^{*}}}$ for an infinite number of $n$. Then note that all sets ${G_{n}}$ and ${V_{n}}$ are contained in a ball ${B_{\frac{1}{\epsilon}} \subset (X \times Y)^{*}}$ of radius ${\frac{1}{\epsilon}}$ centered in ${(\theta, \theta) \in (X \times Y)^{*}}$ which is a bounded set and therefore is a compact space in ${w^{*}}$ topology of ${(X \times Y)^{*}}$ by the Banach-Alaoglu theorem. Also, it is a Fr\'{e}chet-Urysohn space since ${X \times Y}$ is a $BFU-$space. Therefore, applying first Lemma \ref{L:css}, then Lemma \ref{L:ccus} we can find a subsequence  ${(g_{n})_{\mathbb{N}^{'''}} \triangleleft (G_{n})_{\mathbb{N}^{'''}}}$ such that ${(g_{n})_{\mathbb{N}^{'''}} \rightarrow (\theta, \theta)}$ in the ${w^{*}}$ topology. By construction each ${g_{n}}$ can be rewritten as
\[
g_{n} = (\frac{f_{n}A_{n}}{\left\|f_{n}A_{n}\right\|}, \frac{f_{n}}{\left\|f_{n}A_{n}\right\|}).
\]
However, by proposition 1
\[
\left\| \frac{f_{n}A_{n}}{\left\|f_{n}A_{n}\right\|} \right\| \rightarrow 0
\]
which is impossible as each of these norms are obviously equal to 1. This means that our assumption is wrong and ${(\theta, \theta) \notin {\overline{V_{n}}}^{w^{*}}}$ for some ${m \in \mathbb{N}^{'}}$. Hence, there exists a base neighborhood of ${(\theta, \theta) \in (X \times Y)^{*}}$ such that for the same ${m}$
\begin{equation}\label{E:300}
V_{m} \cap W = \emptyset.
\end{equation}
According to definition of the weak topology ${w^{*}}$ the base neighborhood $W$ is determined by a real number ${\eta > 0}$ and by a finite number of vectors ${q_{1}, ..., q_{k}}$ from $X \times Y$ such that ${f \in W}$ if and only if ${|fq_{i}| < \nu}$ for all ${i \in \{1,...,k\}}$. Therefore, due to \eqref{E:300}, for every ${f \in V_{m}}$ there always exists some ${i \in \{1,...,k\}}$ such that ${|fq_{i}| \geq \nu > 0}$. Hence, defining a finite-dimensional space ${Q = sp(\{q_{1},...,q_{k}\})}$, we conclude that ${V_{m} \cap Q^{\bot} = \emptyset}$. Therefore, recalling the definition of ${V_{n}}$ obtain for all ${n > m, n \in \mathbb{N}^{'}}$
\[
G_{n} \cap Q^{\bot} = \emptyset.
\]
However, each ${G_{n}}$ is a scaled image of ${T_{n}}$ while ${Q^{\bot}}$ is a linear subspace, therefore for the same $n$
\begin{equation}\label{E:301}
T_{n} \cap Q^{\bot} = \emptyset.
\end{equation}
If ${P : (X \times Y)^{*} \rightarrow Y^{*}}$ is a natural projection, then define ${R = P(Q^{\bot})}$. It is clear that ${\dim Y^{*} / R < \infty}$ since ${\dim (X \times Y)^{*} / Q^{\bot} < \infty}$. Also, from \eqref{E:301} and from construction of ${T_{n}}$ it is clear that 
\[
\left\|A^{*}_{n} \mid {}_{R} \right\| \leq \epsilon.
\]
which concludes the proof of the proposition 2.\\

In order to prove proposition 3 recall that $\lambda-$norm of a continuous operator ${A \in \mathcal{B}(X,Y)}$ is defined as $\inf$ of such ${\epsilon > 0}$ for which there exists a closed subspace ${L_{\epsilon} = \overline{L_{\epsilon}} \subset X}$ such that ${\dim X / L_{\epsilon} < \infty}$ and ${\left\|A\mid{}_{L_{\epsilon}}\right\| \leq \epsilon}$. Also recall from \cite{sedaev} that for any continuous operator ${A \in \mathcal{B}(X,Y)}$
\[
\left\|A\right\|_{\lambda} = \left\|A^{*}\right\|_{q}.
\]
Therefore, since proposition 2 essentially states that ${\left\|A^{*}_{n}\right\|_{\lambda} \rightarrow 0}$, we conclude that ${\left\|A^{**}_{n}\right\|_{q} \rightarrow 0}$. But it is easy to see that ${\left\|A\right\|_{q} \leq 2 \times \left\|A^{**}\right\|_{q}}$ for any\linebreak[4] ${A \in \mathcal{B}(X,Y)}$. Therefore ${\left\|A_{n}\right\|_{q} \rightarrow 0}$ which concludes the proof of the proposition 3 and of the entire lemma.
\end{proof}

We are now ready to prove the final stability theorem of this section:
\begin{theorem}[(Semi--)Fredholm Operators are Stable under Upper $0-$Adjustment]\label{sfosua}
Let $X$ and $Y$ be two Banach spaces and ${X \times Y}$ is a $BFU-$space, ${A \in \mathcal{B}(X,Y)}$ and ${(A_{n})_{\mathbb{N}^{'}} \subset \mathcal{B}(X,Y)}$. Suppose that ${\lambda_{\mathbb{N}^{'}}[\theta, A_{n}] = 0}$. Then the following propositions are true:
\begin{enumerate}
  \item If $A$ is an upper semi--Fredholm operator, then for large enough ${n \in \mathbb{N}^{'}}$ operators ${A + A_{n}}$ are also upper semi--Fredholm and there exists ${L \in \mathbb{N}}$ such that for the same $n$
\[
\beta(A + A_{n}) \leq L.
\]
  \item If $A$ is a Fredholm operator, then for large enough ${n \in \mathbb{N}^{'}}$ operators ${A + A_{n}}$ are also Fredholm and ${ind(A + A_{n}) = ind(A)}$ and there exists ${R \in \mathbb{N}}$ such that for the same $n$
\[
\alpha(A + A_{n}) \leq R.
\]
\end{enumerate}
\end{theorem}
\begin{proof}
In order to prove proposition 1 apply a well known technique defining subspaces from ${X \times Y}$ (see \cite{kato})
\[
M\ =\ M_{n}\ =\ X \times \{\theta\},\ \ \ N\ =\ G_{A},\ \ \ N_{n}\ =\ G_{A + A_{n}};
\]
obviously ${(M, N)}$ is an upper semi--Fredholm pair of closed subspaces in a $BFU-$space ${X \times Y}$, also it is clear that ${\lambda_{\mathbb{N}^{'}}[M, M_{n}] = 0}$. Let us prove that ${\lambda_{\mathbb{N}^{'}}[N, N_{n}] = 0}$. By definition ${\lambda_{\mathbb{N}^{'}}[N, N_{n}] = \lambda_{\mathbb{N}^{'}}[A, A + A_{n}]}$. Since ${A \in \mathcal{B}(X, Y)}$ and ${\lambda_{\mathbb{N}^{'}}[\theta, A_{n}] = 0}$ it follows from Theorem \ref{nua} proposition 2 that\linebreak[4] ${\lambda_{\mathbb{N}^{'}}[A, A + A_{n}] = 0}$. Therefore, ${\lambda_{\mathbb{N}^{'}}[N, N_{n}] = 0}$. Thus, by Theorem \ref{foasucla} there exists a number ${K_{1} \in \mathbb{N}}$ such that pairs ${(M_{n}, N_{n})}$ are upper semi--Fredholm and values of ${\beta(M_{n}, N_{n})}$ are limited by some ${L \in \mathbb{N}}$ a finite number for ${n > K_{1}}$. However it is obvious that\linebreak[4] ${\beta(M_{n}, N_{n}) = \beta(A + A_{n})}$, so the values of ${\beta(A + A_{n})}$ are limited by the same finite number $L$ for the same $n$.
  
In order to prove proposition 2 notice that if operator $A$ is Fredholm, then according to \cite{goldenstein_markus,sadovskii} operator ${A + C}$ is also Fredholm and ${ind(A + C) = ind(A)}$ for any ${C \in \mathcal{B}(X,Y)}$ with the small enough $q-$norm ${\left\|C\right\|_{q}}$. Then notice that ${\left\|A_{n}\right\|_{q} \rightarrow 0}$ by proposition 3 from lemma 8. Therefore, ${A + A_{n}}$ is also Fredholm and ${ind(A + A_{n}) = ind(A)}$ for large enough $n$.  
  
Finally, recall that by definition
\[
\alpha(A + A_{n})\ =\ ind(A + A_{n})\ +\ \beta(A + A_{n}),
\]
and by the current proof there exists ${K_{2} \in \mathbb{N}}$ such that
\[
ind(A + A_{n})\ =\ ind(A)
\]
for ${n > K_{2}}$; therefore values of ${\alpha(A + A_{n})}$ are limited by the finite number
\[
R\ =\ ind(A + A_{n})\ +\ L
\]
for ${n > \max\{K_{1}, K_{2}\}}$ where ${K_{1}}$ is from the previous step.
\end{proof}

\newpage
\section[Uniform $\mu-$Approximation]{Uniform $\mu-$Approximation in Banach Spaces}\label{S:umabs}

\subsection{Motivation and Definition}\label{S:mode}
The concept of uniform $\lambda-$adjustment appears to be rather broad as it captures in one sweep both gap-close and compact-close subspaces and operators. Thus, there may be no surprise that some stability theorems break-down for $\lambda-$adjustment in general Banach spaces. Moreover, $\lambda-$adjustment may increase dimension of Fredholm operator's kernel or co-dimension of its range similar to the well known fact that a compact perturbation may increase the dimension of the kernel and co-dimension of the range. However, it is also a well known fact that small gap perturbation never do that -- kernel dimension and range co-dimension always remain the same or decrease (see \cite{kato}).

This motivates one to seek an intermediate kind of perturbation -- the one that captures gap-small perturbations and compact convergence, but leaves\linebreak[4] $K_{2}-$approximation and compact perturbations aside. It appears that a condition expressing some sort of point wise convergence is needed, and, indeed, the following two definitions do the job:

\begin{definition}[Uniform $\mu-$Approximation of Sequences of Subspaces]\label{D:umass}
Let $(M_{n})_{\mathbb{N}^{'}}$ and $(P_{n})_{\mathbb{N}^{'}}$ be a pair of sequences of closed subspaces from a Banach space $X$, ${M_{n} \neq \{\theta\}}$ for all ${n \in \mathbb{N}^{'}}$ and a vector ${\mu = (\mu_{1}, \mu_{2}, \mu_{3}) \in \mathbb{R}^{3}}$ such that ${\mu_{1} \geq 0, \mu_{2} \geq 0, \mu_{3} \geq 0}$. We say that $(M_{n})_{\mathbb{N}^{'}}$ is lower $uniformly\ \mu-approximating$  $(P_{n})_{\mathbb{N}^{'}}$ (equivalently -- $(P_{n})_{\mathbb{N}^{'}}$ is upper $uniformly\ \mu-approximating$  $(M_{n})_{\mathbb{N}^{'}}$) if the following two conditions hold
\begin{enumerate}
  \item The sequence $(M_{n})_{\mathbb{N}^{'}}$ is uniformly ${\mu_{1}-}$adjusted with the sequence $(P_{n})_{\mathbb{N}^{'}}$, i.e. ${\lambda_{\mathbb{N}^{'}}[M_{n}, P_{n}] = \mu_{1}}$,
  \item If ${(x_{n})_{\mathbb{N}^{''}} \triangleleft (M_{n})_{\mathbb{N}^{''}}}$ is a unit sequence of vectors such that\linebreak[4] ${\varlimsup_{\mathbb{N}^{''}}\left\|x_{n} - z \right\| \leq \mu_{2}}$ for some vector ${z \in X}$, then there exists a subsequence ${\mathbb{N}^{'''} \subset \mathbb{N}^{''}}$ and a sequence of vectors ${(t_{n})_{\mathbb{N}^{'''}} \triangleleft (P_{n})_{\mathbb{N}^{'''}}}$ such that ${\varlimsup_{\mathbb{N}^{'''}}\left\|t_{n} - z \right\| \leq \mu_{3}}$.
\end{enumerate}

The uniform $\mu-$proximity between $(M_{n})_{\mathbb{N}^{'}}$ and $(P_{n})_{\mathbb{N}^{'}}$ is any ${\mathbb{R}^{3}}$ vector ${\mu = (\mu_{1}, \mu_{2}, \mu_{3})}$ such that $(M_{n})_{\mathbb{N}^{'}}$ $\mu-$approximates $(P_{n})_{\mathbb{N}^{'}}$ -- this vector is denoted by
\[
\mu_{\mathbb{N}^{'}}[M_{n}, P_{n}].
\]
\end{definition}

\begin{definition}[Uniform $\mu-$Approximation between a Subspace and a Sequence of Subspaces]\label{D:umasss}
Similarly to the uniform approximation of a pair of sequences of subspaces, we define uniform  approximation between a subspace $M$ and a sequence of subspaces $(P_{n})_{\mathbb{N}^{'}}$, as well as uniform approximation between a sequence of subspaces $(M_{n})_{\mathbb{N}^{'}}$ and a subspace $P$, denoting uniform $\mu-$proximities
\[
\begin{aligned}
&\mu_{\mathbb{N}^{'}}[M, P_{n}] := \mu_{\mathbb{N}^{'}}[M_{n}, P_{n}]\ \ where\ \ M = M_{n}\ for\ n \in {N}^{'},\\
&\mu_{\mathbb{N}^{'}}[M_{n}, P] := \mu_{\mathbb{N}^{'}}[M_{n}, P_{n}]\ \ where\ \ P = P_{n}\ for\ n \in {N}^{'}.
\end{aligned}
\]
\end{definition}

The concept of $\mu-$approximation is well defined as it is obvious that any two sequences of subspaces  $\mu-$approximate each other with the $\mu-$proximity vector ${(1, 0, 1) \in \mathbb{R}^{3}}$. 

Also, it is not hard to see that when space $P$ is finite dimensional, then lower $\mu-$approximation degenerates to the convergence in the gap topology if $\mu = (0,0,0)$. Similarly, when space $P$ has a final co-dimension, then upper $\mu-$approximation degenerates to the convergence in the gap topology if $\mu = (0,0,0)$.

Once $\mu-$approximation is defined for closed subspaces, it is trivially applied to closed operators -- two sequences of closed operators $\mu-$approximate each other if their graphs $\mu-$approximate each other in the product space. 

As a simple but interesting fact note that if $Y$ is a finite-dimensional Banach space, then a sequence of operators ${(A_{n})_{\mathbb{N}}}$ point wise converges to the null operator ${\theta \in \mathcal{B}(X, Y)}$ if and only if ${(A_{n})_{\mathbb{N}}}$ $(0,0,0)-$approximates the null operator ${\theta}$. In particular, taking ${Y = \mathbb{K}}$ where ${\mathbb{K}}$ is a field of $Y$, we conclude that a sequence of functionals weakly converges to the null-functional if and only if that sequence $(0,0,0)-$approximates that null functional.

It is not hard to see that $\mu-$approximation is a weaker concept than small gap, small norm and compact convergence. Also, when the Banach space $X$ is finite-dimensional, $\mu-$approximation degenerates to small-gap. 

At the same time $\mu-$approximation captures cases that did not use to fit into any of the previous schemes -- for example, if ${(A_{n})_{\mathbb{N}} \subset \mathcal{B}(X,Y)}$ compactly converges to the null operator ${\theta}$ and ${\left\|C_{n}\right\| \leq \epsilon}$ for all ${C_{n} \in \mathcal{B}(X, Y)}$, then ${(A_{n} + C_{n})_{\mathbb{N}}}$ $\mu-$approximates $\theta$ for ${\mu = (\epsilon, \epsilon, \epsilon) \in \mathbb{R}^{3}}$, yet it is clear that ${(A_{n} + C_{n})_{\mathbb{N}}}$ does not compactly converge to  $\theta$, neither norms of ${A_{n} + C_{n}}$ may be limited by $\epsilon$.

\subsection[Fredholm Stability under $\mu-$Approximation]{Fredholm Stability under $\mu-$Approximation}\label{S:fruma}
The concept of $\mu-$approximation allows for full range of stability theorems in the setting of general Banach spaces (without any additional constraints on the structure of $X$) -- they encompass both the stability facts for (semi--)Fredholm pairs of subspaces and for (semi--)Fredholm linear operators. Here is an example of such a theorem:
\begin{theorem}[(semi--)Fredholm Pairs are Stable under $\mu-$Approximation]\label{sfpsma}
Let ${X}$ be a Banach space, ${M, N}$ are two closed subspaces from $X$. Let ${(M_{n})_{\mathbb{N}^{'}}}$ and ${(N_{n})_{\mathbb{N}^{'}}}$ are two sequences of closed subspaces from $X$. Then the following propositions are true:
\begin{enumerate}
  \item Suppose that pair ${(M,N)}$ is lower semi--Fredholm and that the sequence ${(M_{n})_{\mathbb{N}^{'}}}$ lower uniformly $\mu-$approximates space $M$ for some ${\mu = \mu_{M}}$ while the sequence ${(N_{n})_{\mathbb{N}^{'}}}$ lower uniformly $\mu-$approximates space $N$ for some ${\mu = \mu_{N}}$. Then, if ${\mu_{M}}$ and ${\mu_{N}}$ are small enough, the pairs ${(M_{n},N_{n})}$ are also lower semi--Fredholm for large enough ${n \in \mathbb{N}^{'}}$, while\linebreak[4] ${\alpha(M_{n}, N_{n}) \leq \alpha(M,N)}$ for the same $n$.
  \item Suppose that pair ${(M,N)}$ is upper semi--Fredholm and that the sequence ${(M_{n})_{\mathbb{N}^{'}}}$ upper uniformly $\mu-$approximates space $M$ for some ${\mu = \mu_{M}}$ while the sequence ${(N_{n})_{\mathbb{N}^{'}}}$ upper uniformly $\mu-$approximates space $N$ for some ${\mu = \mu_{N}}$. Then, if ${\mu_{M}}$ and ${\mu_{N}}$ are small enough, the pairs ${M_{n},N_{n})}$ are also upper semi--Fredholm for large enough ${n \in \mathbb{N}^{'}}$, while ${\beta(M_{n}, N_{n}) \leq \beta(M,N)}$ for the same $n$.
  \item Suppose that pair ${(M,N)}$ is Fredholm and both lower and upper\linebreak[4] $\mu-$approximation conditions from the above two cases are true -- that is sequences ${(M_{n})_{\mathbb{N}^{'}}}$ and ${(N_{n})_{\mathbb{N}^{'}}}$ upper and lower $\mu-$approximate respective spaces $M$ and $N$. Then when all the proximities are small enough, the pairs ${(M_{n},N_{n})}$ are also Fredholm for large enough ${n \in \mathbb{N}^{'}}$, while
\begin{multline}
\notag
\begin{aligned}
\ \ \ \ \alpha(M_{n}, N_{n}) &\ \leq\ \alpha(M,N),\\
\ \ \ \ \beta(M_{n}, N_{n})  &\ \leq\ \beta(M,N),\\
\ \ \ \ ind(M_{n}, N_{n})    &\ =\ ind(M,N).
\end{aligned}
\end{multline}
for the same $n$.
\end{enumerate}
\end{theorem} 
There are more stability facts in this area in addition to the above theorem. For example, the reverse stability is true:
\begin{theorem}[Reverse Fredholm Stability]\label{rfs}
If ${M}$, $N$, ${(M_{n})_{\mathbb{N}^{'}}}$ and ${(N_{n})_{\mathbb{N}^{'}}}$ are closed subspaces from a Banach space $X$ then the following propositions are true:
\begin{enumerate}
  \item If all pairs ${(M_{n},N_{n})_{\mathbb{N}^{'}}}$ are lower semi--Fredholm so that their coefficients ${\varphi(M_{n},N_{n})}$ are bounded from above, then the pair ${(M,N)}$ is also lower semi--Fredholm for small enough upper uniform proximities ${\mu_{\mathbb{N}^{'}}[M, M_{n}]}$ and  ${\mu_{\mathbb{N}^{'}}[N, N_{n}]}$; moreover, ${\alpha(M, N) \geq \alpha(M_{n}, N_{n})}$ for large enough $n$.
  \item If all pairs ${(M_{n},N_{n})_{\mathbb{N}^{'}}}$ are upper semi--Fredholm so that their coefficients ${\psi(M_{n},N_{n})}$ are bounded from above, then the pair ${(M,N)}$ is also upper semi--Fredholm for small enough lower uniform proximities ${\mu_{\mathbb{N}^{'}}[M_{n}, M]}$ and  ${\mu_{\mathbb{N}^{'}}[N_{n}, N]}$; moreover, ${\beta(M, N) \geq \beta(M_{n}, N_{n})}$ for large enough $n$.
  \item If all pairs ${(M_{n},N_{n})_{\mathbb{N}^{'}}}$ are Fredholm so that their coefficients ${\varphi(M_{n},N_{n})}$ and ${\psi(M_{n},N_{n})}$ are bounded, then the pair ${(M,N)}$ is also Fredholm for small enough uniform proximities ${\mu_{\mathbb{N}^{'}}[M, M_{n}]}$, ${\mu_{\mathbb{N}^{'}}[N, N_{n}]}$, ${\mu_{\mathbb{N}^{'}}[M_{n}, M]}$ and  ${\mu_{\mathbb{N}^{'}}[N_{n}, N]}$; moreover, ${ind(M_{n}, N_{n}) = ind(M, N)}$ for large enough $n$.
\end{enumerate}
\end{theorem}
As another example one can cite a theorem of a complemented kernel:
\begin{theorem}[Complemented Kernels and Sums are Stable]\label{ckas}
If ${M}$, $N$, ${(M_{n})_{\mathbb{N}^{'}}}$ and ${(N_{n})_{\mathbb{N}^{'}}}$ are closed subspaces from a Banach space $X$ then the following propositions are true:
\begin{enumerate}
  \item When the kernel ${M \cap N}$ of an upper semi--Fredholm pair ${(M, N)}$ is complemented, then so will be complemented the kernels ${M_{n} \cap N_{n}}$ of the approximating pairs ${(M_{n}, N_{n})}$ for small enough uniform proximities\linebreak[4] ${\mu_{\mathbb{N}^{'}}[M_{n}, M]}$, ${\mu_{\mathbb{N}^{'}}[M, M_{n}]}$, ${\mu_{\mathbb{N}^{'}}[N_{n}, N]}$ and ${\mu_{\mathbb{N}^{'}}[N, N_{n}]}$.
  \item When the sum ${M + N}$ of a lower semi--Fredholm pair ${(M, N)}$ is complemented, then so will be complemented the sums ${M_{n} + N_{n}}$ of the approximating pairs ${(M_{n}, N_{n})}$ for small enough uniform proximities ${\mu_{\mathbb{N}^{'}}[M_{n}, M]}$, ${\mu_{\mathbb{N}^{'}}[M, M_{n}]}$, ${\mu_{\mathbb{N}^{'}}[N_{n}, N]}$ and ${\mu_{\mathbb{N}^{'}}[N, N_{n}]}$.
\end{enumerate}
\end{theorem}
As another interesting fact we can mention that Fredholm pairs remain stable when considering them in subspaces of the enclosing space $X$ -- note that this theorem is a stronger version of the previous Theorem \ref{sfpsma}:
\begin{theorem}[Extended Stability of (Semi--)Fredholm Pairs]\label{esossfp}
Let X be a Banach space and ${M}$, $N$ are two closed subspaces of $X$ such that their sum is closed ${\overline{M+N} = M+N}$. Let ${(M_{n})_{\mathbb{N}^{'}}}$ and ${(N_{n})_{\mathbb{N}^{'}}}$ be two sequences of closed subspaces from $X$. Then the following propositions are true:
\begin{enumerate}
  \item Suppose that ${Q}$ is a closed subspace from $X$ containing both ${M}$ and $N$ such that ${\beta_{Q}(M,N) = \dim Q/(M+N) < \infty}$. Also suppose that ${(Q_{n})_{\mathbb{N}^{'}}}$ is a sequence of closed subspaces from $X$ such that $Q_{n}$ contains both ${M_{n}}$ and ${N_{n}}$, i.e. ${M_{n} + N_{n} \subset Q_{n}}$ for all ${n \in \mathbb{N}^{'}}$. Further suppose that uniform $\mu-$proximity vectors ${\mu_{\mathbb{N}^{'}}[Q_{n}, Q]}$, ${\mu_{\mathbb{N}^{'}}[M, M_{n}]}$ and ${\mu_{\mathbb{N}^{'}}[N, N_{n}]}$ are small enough. Then for large enough ${n \in \mathbb{N}^{'}}$ the sums of ${M_{n}}$ and ${N_{n}}$ are closed ${\overline{M_{n}+N_{n}} = M_{n}+N_{n}}$ and
\[
\beta_{Q_{n}}(M_{n},N_{n})\ =\ \dim Q_{n} / (M_{n} + N_{n})\ \leq\ \beta_{Q}(M,N).
\]
  \item Suppose that ${P}$ is a closed subspace from $X$ contained in both ${M}$ and $N$ such that ${\alpha_{P}(M,N) = \dim (M \cap N) /P < \infty}$. Also suppose that ${(P_{n})_{\mathbb{N}^{'}}}$ is a sequence of closed subspaces from $X$ such that $P_{n}$ is contained in both ${M_{n}}$ and ${N_{n}}$, i.e. ${P_{n} \subset M_{n} \cap N_{n}}$ for all ${n \in \mathbb{N}^{'}}$. Further suppose that uniform $\mu-$proximity vectors ${\mu_{\mathbb{N}^{'}}[P, P_{n}]}$, ${\mu_{\mathbb{N}^{'}}[M_{n}, M]}$ and ${\mu_{\mathbb{N}^{'}}[N_{n}, N]}$ are small enough. Then for large enough ${n \in \mathbb{N}^{'}}$ the sums of ${M_{n}}$ and ${N_{n}}$ are closed ${\overline{M_{n}+N_{n}} = M_{n}+N_{n}}$ and
\[
\alpha_{P_{n}}(M_{n}, N_{n})\ =\ \dim (M_{n} \cap N_{n}) / P_{n}\ \leq\ \alpha_{P}(M,N).
\]
  \item Suppose that all the previous conditions are satisfied and in addition uniform $\mu-$proximity vectors ${\mu_{\mathbb{N}^{'}}[Q, Q_{n}]}$ and ${\mu_{\mathbb{N}^{'}}[P_{n}, P]}$ are small enough. Define
\begin{multline}
\notag
\begin{aligned}
ind_{(P,Q)}(M,N) &\ =\ \alpha_{(P,Q)}(M,N)\ -\ \beta_{(P,Q)}(M,N),\\
ind_{(P_{n},Q_{n})}(M_{n},N_{n}) &\ =\ \alpha_{(P_{n},Q_{n})}(M_{n},N_{n})\ -\  \beta_{(P_{n},Q_{n})}(M_{n},N_{n});
\end{aligned}
\end{multline}
  Then for large enough ${n \in \mathbb{N}^{'}}$
\[
ind_{(P_{n},Q_{n})}(M_{n},N_{n})\ =\ ind_{(P,Q)}(M,N).
\] 
\end{enumerate}
\end{theorem}
\subsection[Essentially Kato Stability under $\mu-$Approximation]{Essentially Kato Stability under $\mu-$Approximation}\label{S:ruma}
The concept of uniform $\mu-$approximation also allows for strengthening previously known stability facts about Riesz kernels and ranges for essentially Kato continuous operators disturbed by small commuting continuous operators proved in \cite{goldman_krachkovskii2,goldman_krachkovskii22,goldman_krachkovskii3,forster_muller,aiena}. That is, Riesz kernel and range remains stable for non-continuous essentially Kato closed operators under perturbations which are not necessarily continuous:
\begin{theorem}[Stability of Riesz Kernel and Range]\label{sro}
Let $X$ be a Banach space and ${A \in \mathcal{C}(X,X)}$ is a closed operator. Define its Riesz Range and Riesz Kernel respectively 
\[
\mathfrak{M}(A)\ :=\ \bigcap_{n=1}^{\infty} R(A^{n}),\ \ \ \mathfrak{N}(A)\ :=\  \bigcup_{n=1}^{\infty}Ker(A^{n}).
\]
Suppose that the range of $A$ is closed ${\overline{R(A)} = R(A)}$ and the following dimensions are finite
\begin{multline}
\notag
\begin{aligned}
\epsilon(A)\ =\ & \dim\ (N(A) + \mathfrak{M}(A))\ /\ \mathfrak{M}(A)\ <\ \infty,\\
\gamma(A)\ =\ & \dim\ (\mathfrak{N}(A) + \mathfrak{M}(A))\ /\ \mathfrak{M}(A)\ <\ \infty.
\end{aligned}
\end{multline}
Operators of this kind are called essentially Kato operators (see \cite{forster_muller}). Let there be a continuous operator ${C \in \mathcal{BC}(X,X)}$ commuting with $A$, i.e. such that both ${ACx}$ and ${CAx}$ are defined and ${ACx = CAx}$ for each ${x \in dom(A)}$; suppose also that there exists a continuous reverse operator ${(A + C)^{-1} \in \mathcal{B}(X, X)}$. Finally, suppose that ${(B_{n})_{\mathbb{N}}}$ is a sequence of operators (continuous or not) all commuting with both $A$ and $C$ such that the sequence ${(A + B_{n})_{\mathbb{N}} \subset \mathcal{C}(X,X)}$ uniformly lower and upper $\mu-$approximates operator $A$ with ${\mu = (0,0,0)}$. Then for large enough ${n \in  \mathbb{N}}$
\begin{enumerate}
  \item ${\mathfrak{M}(A + B_{n}) \supset \mathfrak{M}(A)}$,
  \item ${\overline{R(A + B_{n})} = R(A + B_{n})}$,
  \item ${\epsilon(A + B_{n}) \leq \epsilon(A)}$,
  \item ${\gamma(A + B_{n}) \leq \gamma(A)}$,
  \item If $X$ can be decomposed ${X = R(A) \oplus S}$, then there exist closed subspaces ${S_{n} \subset X}$ such that ${X = R(A_{n}) \oplus S_{n}}$; if ${\epsilon(A) = 0}$ then ${S_{n} = S}$,
  \item ${\mathfrak{M}(A + B_{n}) \subset \mathfrak{M}(A) + \mathfrak{N}(A)}$,
  \item ${(A + B_{n})^{-1}(\overline{\mathfrak{N}(A)}) \subset \overline{\mathfrak{N}(A)}}$,
  \item ${\overline{\mathfrak{N}(A + B_{n})} \subset \overline{\mathfrak{N}(A)}}$
  \item If $X$ can be decomposed ${X = Ker(A) \oplus S}$, then there exist closed subspaces ${S_{n} \subset X}$ such that ${X = Ker(A_{n}) \oplus S_{n}}$; if ${\epsilon(A) = 0}$ then ${S_{n} = S}$,
  \item If ${\overline{\mathfrak{N}(A)}}$ is a $BFU-$space then ${\overline{\mathfrak{N}(A + B_{n})} \supset \overline{\mathfrak{N}(A)} \cap \mathfrak{M}(A)}$,
  \item If ${\delta(A + B_{n}, A) \rightarrow 0}$ then ${\overline{\mathfrak{N}(A + B_{n})} \supset \overline{\mathfrak{N}(A)} \cap \mathfrak{M}(A)}$.
\end{enumerate}
\end{theorem}

\newpage
\section[Tuples and Complexes]{Tuples and Complexes of Closed Subspaces}\label{S:tccs}

While traditionally Fredholm stability theorems for subspaces were limited to pairs of subspaces, Fredholm stability theorems for continuous operators were considered in recent years for commuting tuples and for complexes of continuous linear operators. In this section we extend the concept of (semi--)Fredholm pair of subspaces to the concepts of (semi--)Fredholm tuple and Fredholm complex of closed subspaces. We also present Fredholm style stability theorems for tuples and complexes of subspaces.

\subsection[Tuple of Closed Subspaces]{Tuple of Closed Subspaces}\label{S:tcs}
We define the concept of an upper semi--Fredholm tuple of closed subspaces in a straightforward way; the concept of a lower semi--Fredholm is less obvious (its rationale will be revealed in the next lemma), while the concept of Fredholm tuple is quite clear:

\begin{definition}[(Semi--)Fredholm Tuples]\label{usft}
Let $X$ be a Banach space, and\linebreak[4] ${T = \{M^{1}, ..., M^{k}\}}$ be a finite set of ${k > 1}$ closed subspaces ${M^{i}}$ from $X$, ${i = 1,...,k}$. We will call $T$ a tuple. Suppose that the sum of tuple's components is closed
\[
Sum(T)\ =\ \sum_{i=1}^{k} M^{i}\ =\ \overline{\sum_{i=1}^{k} M^{i}}.
\]
Then the following definitions are in order:
\begin{enumerate}
  \item Suppose that the space $Sum(T)$ has a finite co-dimension in $X$, i.e.\linebreak[4] ${\beta(T) = \dim X/Sum(T) < \infty}$. Then tuple $T$ is called an upper semi--Fredholm tuple and the number ${\beta(T)}$ is called its upper defect number.
  \item For each ${i= 1,...,k-1}$ define the tuple's $i-$th kernel ${Ker(i,T)}$ to be the subspace equal to
\[
Ker(i,T)\ =\ M^{i} \cap (\sum_{j = i+1}^{k} M^{j}),
\]
Suppose that all tuple's kernels ${Ker(i,T)}$ have finite dimensions\linebreak[4]
${\alpha(i,T) = \dim Ker(i,T) < \infty}$; then tuple $T$ is called a lower semi--Fredholm tuple and the number
\[
\alpha(T)\ =\ \sum_{i=1}^{k-1}\alpha(i,T)
\]
is called its lower defect number.
  \item Suppose that $T$ is both upper and lower semi--Fredholm tuple -- then it is called a Fredholm tuple. The index ${ind(T)}$ of a Fredholm tuple is defined as ${ind(T) = \alpha(T) - \beta(T)}$.
\end{enumerate}
\end{definition}
Although the above definition may give an impression that tuple's defect numbers depend on the particular enumeration of tuple's elements, the following lemma shows it is not the case -- the tuple's defect numbers remain the same for any permutation of indices of tuple's elements; also it gives the explicit formula for index calculation in finite-dimensional spaces:
\begin{lemma}[Tuple's Defect Numbers are Well Defined]\label{tdnwd}
Let ${T=\{M^{1},...,M^{k}\}}$ be a tuple of ${k \geq 2}$ subspaces of a Banach space $X$. Then the following propositions are true:
\begin{enumerate}
  \item The upper defect number ${\beta(T)}$ of an upper semi--Fredholm tuple does not depend on enumeration of tuple's elements.
  \item The lower defect number ${\alpha(T)}$ of a lower semi--Fredholm tuple does not depend on enumeration of tuple's elements.
  \item If $X$ is a finite-dimensional Banach space, then any tuple of its subspaces ${T = \{M^{1},...,M^{k}\}}$ is a Fredholm tuple and $T$'s index can be effectively calculated using dimensions of $X$ and $M^{i}$ like this:
\[
ind(T)\ =\ \sum_{i=1}^{k} \dim M^{i}\ -\ \dim X.
\]
\end{enumerate}
\end{lemma}
\begin{proof}
Proposition 1 is obvious since addition of vectors in Banach spaces is commutative, so for any permutation ${\{j_{1},...,j_{k}\}}$ of the set ${\{1,...,k\}}$
\[
Sum(T)\ =\ \sum_{i=1}^{k} M^{i}\ =\ \sum_{i=1}^{k} M^{j_{i}}.
\]
In order to prove proposition 2 consider direct product ${\Pi_{T}}$ of the elements of the tuple ${T = \{M^{1},...,M^{k}\}}$:
\[
\Pi_{T}\ =\ \prod_{i=1}^{k} M_{i}\ =\ M^{1} \times ... \times M^{k}.
\]
Define a natural linear mapping
\[
\Phi_{T} : \Pi_{T} \rightarrow X,\ \ \ \ \Phi_{T}:(x^{1},...,x^{k}) \mapsto \sum_{i=1}^{k}x^{i}.
\]
Let us prove that ${\dim Ker(\Phi_{T}) = \alpha(T)}$. Let ${Pr_{1}}$ is a natural projection from ${Ker(\Phi_{T})}$ onto its first component ${M^{1}}$; it is a simple algebraic fact that the linear space ${Ker(\Phi_{T})}$ is isomorphic a to direct product of its image and kernel
\[
Ker(\Phi_{T})\ \equiv\ Pr_{1}(Ker(\Phi_{T})) \times Ker(Pr_{1}).
\]
It is also clear that ${x_{1} \in Pr_{1}(Ker(\Phi_{T})) \subset M^{1}}$ if and only if ${x^{1} + x^{2} + ... +x^{k} = \theta}$ for some ${x^{2} \in M^{2}, ..., x^{k} \in M^{k}}$, in other words ${x^{1} = -(x^{2} + ... +x^{k})}$, or, equivalently, ${x^{1} \in M^{1} \cap(\sum_{i=2}^{k} M^{i})}$. Therefore,
\[
Pr_{1}(Ker(\Phi_{T}))\ =\ M^{1} \cap(\sum_{i=2}^{k} M^{i}).
\]
Combining the last two relations we obtain algebraic isomorphism
\[
Ker(\Phi_{T})\ \ \equiv\ \ (M^{1} \cap(\sum_{i=2}^{k} M^{i}))\ \times\ Ker(Pr_{1}).
\]
Now consider the space ${Ker(Pr_{1})}$ -- it consists of vectors ${(\theta, x^{2},...,x^{k}) \in Ker(\Phi_{T})}$ such that ${x^{2} + ...+ x^{k} = \theta}$. Thus, if ${k > 2}$, then we may consider a sub-tuple ${T_{1} \subset T}$, ${T_{1} = \{M^{2},...,M^{k}\}}$ and apply the same reasoning to ${T_{1}}$ -- for a natural projection ${Pr_{2}}$ of ${Ker(\Phi_{T_{1}})}$ onto its first component ${M^{2}}$ we will obtain
\[
Ker(\Phi_{T_{1}})\ \ \equiv\ \ (M^{2} \cap(\sum_{i=3}^{k} M^{i}))\ \times\ Ker(Pr_{2}).
\]
However, obviously
\[
Ker(Pr_{1})\ \ \equiv\ \ Ker(\Phi_{T_{1}});
\]
therefore, combining the previous three isomorphisms we obtain isomorphism
\[
Ker(\Phi_{T})\ \ \equiv\ \ (M^{1} \cap(\sum_{i=2}^{k} M^{i}))\ \times\ (M^{2} \cap(\sum_{i=3}^{k} M^{i}))\ \times\ Ker(Pr_{2}).
\]
It is clear that continuing this way we end up with the isomorphism
\[
Ker(\Phi_{T})\ \ \equiv\ \ (M^{1} \cap(\sum_{i=2}^{k} M^{i}))\ \times\ (M^{2} \cap(\sum_{i=3}^{k} M^{i}))\ \times\ ...\ \times\ (M^{k-1} \cap M^{k}).
\]
Since all numbers ${\alpha(j,T)}$ for ${j=1,...,k-1}$ are by definition dimensions of the corresponding components ${M^{j} \cap(\sum_{i=j+1}^{k} M^{i})}$ it then follows that
\begin{equation}\label{E:500}
\dim Ker(\Phi_{T})\ =\ \sum_{j=1}^{k-1} \alpha(j,T)\ =\ \alpha(T).
\end{equation}
In order to finish the proof of the proposition 2 note that for any permutation ${\{j_{1},...,j_{k}\}}$ of the set ${\{1,...,k\}}$ we may consider a re-ordered tuple\linebreak[4] ${T^{'} = \{M^{j_{1}},...,M^{j_{k}}\}}$ and a new product
\[
\Pi_{T^{'}}\ =\ \prod _{i=1}^{k} M^{j_{i}}\ =\ M^{j_{1}} \times ... \times M^{j_{k}},
\]
and define a new natural linear mapping
\[
\Phi_{T^{'}} : \Pi_{T^{'}} \rightarrow X,\ \ \ \ \Phi_{T^{'}}:(x^{j_{1}},...,x^{j_{k}}) \mapsto \sum_{i=1}^{k}x^{j_{i}}. 
\]
Due to commutativity of addition in vector space $X$ it is obvious that the kernels of the mappings ${\Pi_{T}}$ and ${\Pi_{T^{'}}}$ are isomorphic:
\[
Ker(\Phi_{T})\ \ \equiv\ \ Ker(\Phi_{T^{'}}).
\]
Therefore, their dimensions are the same. Hence, according to \eqref{E:500} we obtain the equality
\begin{equation}\label{E:502}
\alpha(T)\ =\ \dim Ker(\Phi_{T})\ =\ \dim Ker(\Phi_{T}^{'})\ =\ \alpha(T^{'}) 
\end{equation}
which means that the value of the lower defect number ${\alpha(T)}$ of a lower semi--Fredholm tuple $T$ does not depend on the permutation of $T$'s elements.\\

In order to prove proposition 3 recall a simple algebraic fact about a linear operator -- dimension of its domain is equal to the sum of dimensions of its range and kernel; for ${\Phi_{T}}$ it means
\[
\dim R(\Phi_{T})\ +\ \dim Ker(\Phi_{T})\ =\ \dim dom ( \Phi_{T})
\]
Substituting ${R(\Phi_{T})}$ with ${\sum_{i=1}^{k} M^{i}}$, ${\dim Ker(\Phi_{T})}$ with ${\alpha(T)}$ (from \eqref{E:500}) and ${\dim dom ( \Phi_{T})}$ with ${\sum_{i=1}^{k} \dim M^{i}}$ rewrite the previous equality
\[
\dim \sum_{i=1}^{k} M^{i}\ +\ \alpha(T)\ =\ \sum_{i=1}^{k} \dim M^{i}.
\]
Now recall definition of ${\beta(T)}$:
\[
X\ -\ \dim \sum_{i=1}^{k} M^{i}\ =\ \beta(T).
\]
Adding the previous two equalities obtain
\begin{multline}
\notag
\begin{aligned}
\dim \sum_{i=1}^{k} M^{i}\ +\ \alpha(T)\ +\ X\ -\ \dim \sum_{i=1}^{k} M^{i} & \ =\ \sum_{i=1}^{k} \dim M^{i}\ +\ \beta(T),\\
\alpha(T)\ +\ X                                                        & \ =\ \beta(T)\ +\ \sum_{i=1}^{k} \dim M^{i}.
\end{aligned}
\end{multline}
The last equality can be then rewritten like this:
\[
ind(T)\ =\ \alpha(T)\ -\ \beta(T)\ =\ \sum_{i=1}^{k} \dim M^{i}\ -\ X
\]
which concludes the proof of the proposition 3 and of the entire lemma.
\end{proof}

\subsection[Stability of (Semi--)Fredholm Tuples]{Stability of (Semi--)Fredholm Tuples}\label{S:ssft}
It appears that (semi--)Fredholm tuples of closed subspaces are stable under uniform $\lambda-$adjustment and under uniform $\mu-$approximation -- every theorem establishes so far for (semi--)Fredholm pairs of subspaces has its counterpart for (semi--)Fredholm tuples of subspaces. The proof of every such theorem is quite standard -- for a given tuple of $k$ closed subspaces from a space $X$ we may consider a canonical pair in the product space ${X \overset{k+1}{\times ... \times} X}$ and apply the already established stability theorem for the pair -- the tuple stability then follows automatically since defect numbers of the canonical pair and defect numbers of the tuple are the same. The following lemma establishes equivalence between tuples and pairs, between their defect numbers, as well as between their uniform $\lambda-$adjustment numbers and between their uniform $\mu-$proximity vectors.
\begin{lemma}[Canonical Mapping between Tuples and Pairs]\label{cmbtp}
Let $X$ be a Banach space and ${T = \{M^{1},...,M^{k}\}}$ be a tuple of ${k > 1}$ closed subspaces from $X$. Define a new space ${X^{(k)} = X \overset{k}{\times ... \times} X}$ and a canonical linear map
\begin{multline}
\notag
\begin{aligned}
& \Phi_{T} \in \mathcal{BC}(X^{(k)}, X),\\
& \Phi_{T} : dom(\Phi_{T}) = \prod_{i=1}^{k}M^{i}\ \rightarrow\ X,\\
& \Phi_{T} : (x^{1}, ... , x^{k})\ \mapsto\ \frac{1}{k} \times \sum_{i=1}^{k} x^{i};
\end{aligned}
\end{multline}
let ${G_{\Phi_{T}} \subset X^{(k)} \times X}$ be a graph of the operator ${\Phi_{T}}$; also define a pair ${P_{T}}$ of closed subspaces in the product space ${X^{(k)} \times X}$:
\[
P_{T}\ =\ ( X^{(k)} \times \{\theta\},\ G_{\Phi_{T}}).
\]
Then the following propositions are true:
\begin{enumerate}
  \item $T$ is a lower semi--Fredholm tuple if and only if ${P_{T}}$ is a lower semi--Fredholm pair; in this case ${\alpha(P_{T}) = \alpha(T)}$.
  \item $T$ is an upper semi--Fredholm tuple if and only if ${P_{T}}$ is an upper semi--Fredholm pair; in this case ${\beta(P_{T}) = \beta(T)}$.
  \item $T$ is a Fredholm tuple if and only if ${P_{T}}$ is a Fredholm pair; in this case ${ind(P_{T}) = ind(T)}$.
  \item Let ${(T_{n})_{\mathbb{N}^{'}}}$ and ${(S_{n})_{\mathbb{N}^{'}}}$ are two sequences of $k-$tuples of closed subspaces from a Banach space X such that ${T_{n} = \{M^{1}_{n},...,M^{k}_{n}\}}$ and ${S_{n} = \{P^{1}_{n},...,P^{k}_{n}\}}$ for all ${n \in \mathbb{N}^{'}}$. For the same $n$ consider graphs ${G_{\Phi_{T_{n}}} \subset X^{(k)} \times X}$ and ${G_{\Phi_{S_{n}}} \subset X^{(k)} \times X}$ of the canonical mappings ${\Phi_{T_{n}}}$ and ${\Phi_{S_{n}}}$.\linebreak[4]For ${j = 1,...,k}$ define numbers ${\lambda^{j} \in \mathbb{R}}$ and vectors ${\mu^{j} \in \mathbb{R}^{3}}$:
\begin{multline}
\notag
\begin{aligned}
&\ \ \ \ \ \ \lambda^{j}\ :=\ \lambda_{\mathbb{N}^{'}}[M^{j}_{n}, P^{j}_{n}],\\
&\ \ \ \ \ \ \mu^{j}\ :=\ \mu_{\mathbb{N}^{'}}[M^{j}_{n}, P^{j}_{n}]\ =\ (\mu^{j}_{1},\mu^{j}_{2},\mu^{j}_{3}) .
\end{aligned}
\end{multline}
Also define new numbers ${\lambda}$, ${\mu_{m} \in \mathbb{R}}$:
\begin{multline}
\notag
\begin{aligned}
\ \ \ \ \ \ \lambda\ & :=\ \max \{\lambda^{j} \mid j = 1,...,k\},\\
\ \ \ \ \ \ \mu_{m}\ & :=\ \max \{\mu^{j}_{m}  \mid j = 1,...,k\}\ \ for\ m = 1,2,3.
\end{aligned}
\end{multline}
Then the following formulas hold:
\begin{multline}
\notag
\begin{aligned}
&\ \ \ \ \ \ \lambda_{\mathbb{N}^{'}}[G_{\Phi_{T_{n}}}, G_{\Phi_{S_{n}}}]\ =\ \lambda,\\
&\ \ \ \ \ \ \mu_{\mathbb{N}^{'}}[G_{\Phi_{T_{n}}}, G_{\Phi_{S_{n}}}]\ =\ (\mu_{1}, \mu_{2}, \mu_{3}).
\end{aligned}
\end{multline}
\end{enumerate}
\end{lemma}
\begin{proof}
The proof of the proposition 1 exactly follows the proof of the proposition 2 from the previous Lemma \ref{tdnwd}. Proposition 2 is quite obvious, while proposition 3 is a direct consequence of propositions 1 and 2. Proposition 4 is mechanically verified using definitions of uniform $\lambda-$adjustment and uniform $\mu-$approximation and the definition of norm on the product space ${X^{(k)} \times X}$ as the ${\max}$ of the norms of its components.
\end{proof}

As we have mentioned in the beginning of this subsection, every stability result for (semi--)Fredholm pair of subspaces can be extended to a similar result for (semi--)Fredholm tuples. Here is just one example of such extension -- others can be formulated and proved in the same fashion:
\begin{theorem}[Stability of (Semi--)Fredholm Tuples]\label{ssft}
Let $X$ be a Banach space and ${T = \{M^{1},...,M^{k}\}}$ be a tuple of ${k > 1}$ closed subspaces from $X$. Let ${(T_{n})_{\mathbb{N}^{'}}}$ be a sequence of $k-$tuples of closed subspaces from X such that\linebreak[4] ${T_{n} = \{M^{1}_{n},...,M^{k}_{n}\}}$ for all ${n \in \mathbb{N}^{'}}$. Then the following propositions are true:
\begin{enumerate}
  \item If tuple $T$ is lower semi--Fredholm and numbers ${\lambda_{\mathbb{N}^{'}}[M^{j}_{n}, M^{j}]}$ are small enough for all ${j=1,...,k}$, then tuples ${T_{n}}$ are also lower semi--Fredholm for large enough ${n \in \mathbb{N}^{'}}$ and there exists ${K_{\alpha} \in \mathbb{N}}$ such that ${\alpha(T_{n}) \leq K_{\alpha}}$ for the same $n$. If vectors ${\mu_{\mathbb{N}^{'}}[M^{j}_{n}, M^{j}] \in \mathbb{R}^{3}}$ are small enough for all ${j=1,...,k}$, then ${K_{\alpha} = \alpha(T)}$.
  \item Suppose that $X$ is a $BFU-$space. If tuple $T$ is upper semi--Fredholm and numbers ${\lambda_{\mathbb{N}^{'}}[M^{j}, M^{j}_{n}]}$ are small enough for all ${j=1,...,k}$, then tuples ${T_{n}}$ are also upper semi--Fredholm for large enough ${n \in \mathbb{N}^{'}}$ and there exists ${K_{\beta} \in \mathbb{N}}$ such that ${\beta(T_{n}) \leq K_{\beta}}$ for the same $n$. If X is a Banach space and vectors ${\mu_{\mathbb{N}^{'}}[M^{j}, M^{j}_{n}] \in \mathbb{R}^{3}}$ are small enough for all ${j=1,...,k}$, then ${K_{\beta} = \beta(T)}$.
  \item Suppose that $X$ is a $BFU-$space. If tuple $T$ is Fredholm and both numbers ${\lambda_{\mathbb{N}^{'}}[M^{j}, M^{j}_{n}]}$ and ${\lambda_{\mathbb{N}^{'}}[M^{j}, M^{j}_{n}]}$ are small enough for all ${j=1,...,k}$, then tuples ${T_{n}}$ are also Fredholm for large enough ${n \in \mathbb{N}^{'}}$ and there exists ${K_{\alpha}, K_{\beta} \in \mathbb{N}}$ such that ${\alpha(T_{n}) \leq K_{\alpha}}$ and ${\beta(T_{n}) < K_{\beta}}$ for the same $n$. If X is a Banach space and both vectors ${\mu_{\mathbb{N}^{'}}[M^{j}_{n}, M^{j}] \in \mathbb{R}^{3}}$, ${\mu_{\mathbb{N}^{'}}[M^{j}, M^{j}_{n}] \in \mathbb{R}^{3}}$ are small enough for all ${j=1,...,k}$, then ${K_{\alpha} = \alpha(T)}$ and ${K_{\beta} = \beta(T)}$, also ${ind(T_{n}) = ind(T)}$ for large enough ${n \in \mathbb{N}^{'}}$.
\end{enumerate}
\end{theorem}
\begin{proof}
We will prove only proposition 1 -- propositions 2 and 3 can be dealt with similarly. 

First suppose that all numbers ${\lambda_{\mathbb{N}^{'}}[M^{j}_{n}, P^{j}_{n}]}$ are small for ${j = 1,...,k}$. Consider canonical pairs ${P_{T} = (X^{(k)} \times \{\theta\}, G_{\Phi_{T}})}$ and ${P_{T_{n}} = (X^{(k)} \times \{\theta\}, G_{\Phi_{T_{n}}})}$ in the product space ${X^{(k)} \times X}$. Since tuple $T$ is lower semi--Fredholm, pair ${P_{T}}$ is also lower semi--Fredholm  and ${\alpha(P_{T}) = \alpha(T)}$ according to proposition 1 from Lemma \ref{cmbtp}. According to proposition 4 of the same Lemma \ref{cmbtp} number\linebreak[4] ${\lambda_{\mathbb{N}^{'}}[G_{\Phi_{T_{n}}}, G_{\Phi_{S_{n}}}]}$ is small when all $k$ numbers ${\lambda_{\mathbb{N}^{'}}[M^{j}_{n}, P^{j}_{n}]}$ are small for ${j = 1,...,k}$. Therefore, according to proposition 1 from Theorem \ref{lsfps} applied to pairs ${P_{T}}$ and ${P_{T_{n}}}$, pairs ${P_{T_{n}}}$ are also lower semi--Fredholm for large enough ${n \in \mathbb{N}^{'}}$ and there exists ${K_{\alpha} \in \mathbb{N}}$ such that ${\alpha(P_{T_{n}}) \leq K_{\alpha}}$ for the same $n$. Then, applying proposition 1 from Lemma \ref{cmbtp} to pairs ${P_{T_{n}}}$ and tuples ${T_{n}}$ we conclude that for the same $n$ tuples ${T_{n}}$ are also lower semi--Fredholm and ${\alpha(T_{n}) = \alpha(P_{T_{n}}) \leq K_{\alpha}}$. This concludes the proof of the first part of proposition 1. 

For the second part note that according to proposition 4 of Lemma \ref{cmbtp} vector ${\mu_{\mathbb{N}^{'}}[G_{\Phi_{T_{n}}}, G_{\Phi_{S_{n}}}]}$ is small when all $k$ vectors ${\mu_{\mathbb{N}^{'}}[M^{j}_{n}, P^{j}_{n}]}$ are small for ${j = 1,...,k}$. Therefore, according to proposition 1 from Theorem \ref{sfpsma} pairs ${P_{T_{n}}}$ are also lower semi--Fredholm for large enough ${n \in \mathbb{N}^{'}}$ and ${\alpha(P_{T_{n}}) \leq \alpha(P_{T})}$ for the same $n$. Again, applying proposition 1 from Lemma \ref{cmbtp} to pairs ${P_{T_{n}}}$ and tuples ${T_{n}}$ we conclude that for the same $n$ tuples ${T_{n}}$ are also lower semi--Fredholm and ${\alpha(T_{n}) = \alpha(P_{T_{n}}) \leq \alpha(P_{T}) = \alpha(T)}$. Thus the second part of proposition 1 has been proved as well.
\end{proof}

\subsection[Pair Complex]{Pair Complex}\label{pc}
The previously defined concept of tuples of closed subspaces captures the properties of closed subspaces related to their sums and intersections. However, it does not reflect the case when some subspaces are included in others. In this section we introduce the concept of a pair complex -- a finite set of subspaces that captures the inclusion relations of its members. 
\begin{definition}[Pair Complex]\label{D:pc}
Let $X$ be a Banach space and ${PC = \{P^{1},...,P^{k}\}}$ is a finite set of ${k \geq 2}$ pairs of closed subspaces, each pair ${P^{j}}$ comprised of two subspaces ${M^{j}}$ and ${N^{j}}$: ${P^{j} = (M^{j}, N^{j})}$. We call ${PC}$ a pair complex if the sum of every but the last pair is included in the intersection of its successor pair, in other words if for each ${j=1,...,k-1}$
\[
M^{j} + N^{j}\ \subset\ M^{j+1} \cap N^{j+1}.
\]
\end{definition}
\begin{definition}[Fredholm Pair Complex]\label{D:fpc}
Let $X$ be a Banach space and\linebreak[4] ${PC = \{P^{1},...,P^{k}\}}$ is a pair complex. Suppose that the sum of each pair's components is closed ${\overline{M^{j} + N^{j}} = M^{j} + N^{j}}$ for each ${j=1,...,k}$, and that the following dimensions, called defect numbers of pair complex $PC$, are finite:
\begin{multline}
\notag
\begin{aligned}
\gamma_{1}\   & := \ \dim M^{1} \cap N^{1} < \infty,\\
\gamma_{j}\   & := \ \dim (M^{j} \cap N^{j}) / (M^{j-1} + N^{j-1}) < \infty,\ j = 2,...,k,\\
\gamma_{k+1}\ & := \ \dim X / (M^{k} + N^{k}) < \infty.
\end{aligned}
\end{multline}
Then pair complex $PC$ is called a Fredholm complex. Its index is calculated according to the following formula:
\[
ind(PC)\ :=\ \sum_{j=1}^{k+1} (-1)^{j-1} \times \gamma_{j}.
\]
\end{definition}
The concept of a pair complex closely mirrors the well known concept of an operator complex defined as a sequence of linear operators between Banach spaces such that the range of each but the last operator is included in the kernel of the next operator (see \cite{ambrozie}). Indeed, we will now show how to build a pair complex from a given operator complex. 

Consider an operator complex $OC$ comprised of a finite set of Banach spaces ${\{X^{1},...,X^{k+1}\}}$ where ${k \geq 1}$, and of a finite set of closed linear operators\linebreak[4] ${\{A^{1},...,A^{k}\}}$ where ${A^{i} \in \mathcal{C}(X^{i}, X^{i+1})}$ and ${R(A^{i}) \subset Ker(A^{i+1})}$ for each\linebreak[4] ${i=1,...,k}$:
\begin{equation}\label{E:600}
\begin{CD}
X^{1} @>A^{1}>> X^{2}\ \cdots\ X^{k} @>A^{k}>> X^{k+1}.
\end{CD}
\end{equation}
Define a new Banach space ${X}$ and a finite set ${\{M^{1},...,M^{k}\}}$ of its closed subspaces 
\begin{multline}
\notag
\begin{aligned}
X &\ =\ \prod_{i=1}^{k+1} X^{i},\\
M^{j} &\ =\ \prod_{i=1}^{j} X^{i}\ \times\ \prod_{i=j+1}^{k+1} \{\theta\}\ \subset\ X\ \ for\ j = 1,...,k.
\end{aligned}
\end{multline}
Also define $k$ closed subspaces ${N^{j} \subset X}$ for each ${j \in 1,...,k}$ like this
\[
N^{j} = \{(\underbrace{x^{1}, ... , x^{j}}_{j}, A^{j}x^{j}, \underbrace{\theta, ... ,\theta}_{k-j})\ \mid\ x^{i} \in X^{i},\ i = 1,...,j\}.
\]
Now consider $k$ pairs of closed subspaces from $X$ for each ${j \in 1,...,k}$:
\[
P_{j}\ =\ (M^{j},\ N^{j}).
\]
The following lemma reveals the structure of the set ${\{P^{1},...,P^{k}\}}$:
\begin{lemma}[Operator Complex Induces a Pair Complex]\label{ociapc}
Let $OC$ be an operator complex of size $k+1$ with ${k \geq 1}$ from \eqref{E:600}. Consider the set of pairs ${PC = \{P^{1},...,P^{k}\}}$ built as above. Then the following propositions are true
\begin{enumerate}
 \item The set of pairs $PC$ is a pair complex.
 \item Recall that an operator complex $OC$ is Fredholm if all operators ${A^{j}}$ have close ranges and the following dimensions, called defect numbers of operator complex $OC$, are finite:
\begin{multline}
\notag
\begin{aligned}
\alpha_{1}   &\ :=\ \dim Ker(A^{1}) < \infty,\\
\alpha_{j}   &\ :=\ \dim Ker(A^{j}) / R(A^{j-1}) < \infty,\ j = 2,...,k,\\
\alpha_{k+1} &\ :=\ \dim X^{k+1} / R(A^{k}) < \infty.
\end{aligned}
\end{multline}
The index of Fredholm complex is calculated as
\[
ind(OC) = \sum_{j=1}^{k+1} (-1)^{j-1} \times \alpha_{j}.
\]
Now, $OC$ is a Fredholm operator complex if and only if $PC$ is a Fredholm pair complex. In this case ${\gamma_{j} = \alpha_{j}}$ for each ${j=1,...,k+1}$; in particular
\[
ind(OC) = ind(PC).
\]
\end{enumerate}
\end{lemma}
\begin{proof}
In order to prove proposition 1 notice that obviously every component of each pair is a closed subspace from ${X}$. We now only need to show that for each ${j=1,...,k-1}$
\[
M^{j}\ +\ N^{j}\ \ \subset\ \ M^{j+1}\ \cap\ N^{j+1}.
\]
For that it is sufficient to prove that
\begin{multline}
\notag
\begin{aligned}
& (a)\ \ & M^{j}\ \ \subset\ \ & M^{j+1},\\
& (b)\ \ & M^{j}\ \ \subset\ \ & N^{j+1},\\
& (c)\ \ & N^{j}\ \ \subset\ \ & M^{j+1},\\
& (d)\ \ & N^{j}\ \ \subset\ \ & N^{j+1}.\\
\end{aligned}
\end{multline}
Formula $(a)$ is true by definition of ${M^{j}}$. Formula $(b)$ is true since for every ${y \in M^{j}}$
\begin{multline}
\notag
\begin{aligned}
y &\ =\ (\underbrace{x^{1}, ... , x^{j}}_{j}, \underbrace{\theta, ... ,\theta}_{k-j+1})\\
  &\ =\ (\underbrace{x^{1}, ... , x^{j}}_{j}, \theta, \theta, \underbrace{\theta, ... ,\theta}_{k-j-1})\\
  &\ =\ (\underbrace{x^{1}, ... , x^{j}}_{j}, \theta, A^{j+1}\theta, \underbrace{\theta, ... ,\theta}_{k-j-1})\ \in\ N^{j+1}.
\end{aligned}
\end{multline}
Formula $(c)$ is true since all vectors from ${N^{j}}$ have null components at indices higher than ${j+1}$. Finally, formula $(d)$ is true due to the property of operator complex $OC$ having ${A^{j+1} \circ A^{j} = \theta}$; that is, for every ${z \in N^{j}}$:
\begin{multline}
\notag
\begin{aligned}
z &\ =\ (\underbrace{x^{1}, ... , x^{j}}_{j}, A^{j}x^{j}, \underbrace{\theta, ... ,\theta}_{k-j})\\
  &\ =\ (\underbrace{x^{1}, ... , x^{j}}_{j}, A^{j}x^{j}, \theta, \underbrace{\theta, ... ,\theta}_{k-j-1})\\
  &\ =\ (\underbrace{x^{1}, ... , x^{j}}_{j}, A^{j}x^{j}, A^{j+1}A^{j}x^{j}, \underbrace{\theta, ... ,\theta}_{k-j-1})\ \in\ N^{j+1}.
\end{aligned}
\end{multline}

In order to prove proposition 2 assume that operator complex $OC$ is Fredholm. We will prove that pair complex $PC$ is Fredholm as well. First notice that from definitions of ${M^{j}}$ and ${N^{j}}$ it follows that
\begin{equation}\label{E:601}
M^{j} + N^{j}\ =\ \prod_{i=1}^{j} X^{i}\ \times\ R(A^{j}) \times \prod_{i=j+2}^{k+1} \{\theta\},\ j=1,...,k. 
\end{equation}
Also notice that ranges ${R(A^{j})}$ of all operators ${A^{j}}$ are closed since operator complex $OC$ is Fredholm. Thus, it follows from the above formula that\linebreak[4] ${\overline{M^{j} + N^{j}} = M^{j} + N^{j}}$ for all ${j=1,...,k}$. 

Now let us show that ${\gamma_{j} = \alpha_{j}}$ for each ${j=1,...,k}$. It is obvious that these equalities are true due to the following easily verifiable isomorphisms:
\begin{multline}\label{E:602}
\begin{aligned}
M^{1} \cap N^{1} &\ \equiv\ Ker(A^{1}),\\
(M^{j} \cap N^{j}) / (M^{j-1} + N^{j-1}) &\ \equiv\ Ker(A^{j})/R(A^{j-1}),\ j=2,...,k,\\
X / (M^{k} + N^{k}) &\ \equiv\ X^{k+1} / R(A^{k}).
\end{aligned}
\end{multline}
From the fact that ${\gamma_{j} = \alpha_{j}}$ for each ${j=1,...,k}$ it then clearly follows that
\[
ind(OC)\ =\ \sum_{j=1}^{k+1} (-1)^{j-1} \times \alpha_{j}\ =\ \sum_{j=1}^{k+1} (-1)^{j-1} \times \gamma_{j}\ =\ ind(PC).
\]
Now assume that pair complex $PC$ is Fredholm and let us prove that operator complex $OC$ is Fredholm as well. First notice, that since ${\overline{M^{j} + N^{j}} = M^{j} + N^{j}}$ for all ${j=1,...,k}$, then it follows from \eqref{E:601} that the ranges of operators ${A^{j}}$ are closed ${\overline{R(A^{j})} = R(A^{j})}$ for the same $j$. Also, from the inclusion property of a Fredholm pair complex ${M^{j-1} + N^{j-1} \subset M^{j} \cap N^{j}}$, from \eqref{E:601} and from the following obvious formula
\[
M^{j} \cap N^{j}\ =\ \prod_{i=1}^{j-1} X^{i}\ \times\ Ker(A^{j}) \times \prod_{i=j+1}^{k+1} \{\theta\},\ 
j=1,...,k. 
\]
it follows that ${R(A^{j-1}) \subset Ker(A^{j})}$ for all ${j=2,...,k}$. Finally, from the isomorphisms of \eqref{E:602} it also follows that ${\gamma_{j} = \alpha_{j}}$ for each ${j=1,...,k}$; consequently ${ind(OC) = ind(PC)}$. This concludes the proof of the proposition 2 and of the entire lemma.
\end{proof}
While the previous lemma demonstrates how to construct a pair complex from an operator complex, the next lemma shows the opposite by constructing an operator complex from the pair complex:
\begin{lemma}[Building Operator Complex from a Pair Complex]\label{bocpc}
Let $X$ be a Banach space and ${PC = \{P^{1},...,P^{k}\}}$ be a pair complex from $X$ with ${k \geq 1}$ pairs ${P^{j} = (M^{j},N^{j})}$ for each ${j=1,...,k}$. Define the set of ${k+1}$ Banach spaces ${\{X^{1}, ..., X^{k}, X^{k+1}\}}$ and the set of $k$ subspaces ${D^{j} \subset X^{j}}$ for ${j=1,...,k}$:
\begin{multline}
\notag
\begin{aligned}
X^{j}   &\ :=\ X \times X,\ D^{j}\ :=\ M^{j} \times N^{j} \subset X^{j},\ for\ j=1,...,k,\\
X^{k+1} &\ :=\ X.
\end{aligned}
\end{multline}
Also define the set of $k-1$ operators ${\{A^{1},...,A^{k-1}\}}$:
\begin{multline}
\notag
\begin{aligned}
& A^{j}\ \in\ \mathcal{BC}(X^{j}, X^{j+1}),\ \ dom(A^{j}) = D^{j},\ \ A^{j}\ :\ (x, y)\ \mapsto\ (x+y,\ -(x+y)),
\end{aligned}
\end{multline}
and operator $A^{k}$
\begin{multline}
\notag
\begin{aligned}
& A^{k}\ \in\ \mathcal{BC}(X^{k}, X^{k+1}),\ \ dom(A^{k}) = D^{k},\ \ A^{k}\ :\ (x, y)\ \mapsto\ x+y.
\end{aligned}
\end{multline}
Then the following propositions are true:
\begin{enumerate}
  \item The set of spaces ${\{X^{1}, ..., X^{k}, X^{k+1}\}}$ and the set of operators ${\{A^{1},...,A^{k}\}}$ constitute an operator complex $OC$:
\begin{equation}\label{E:700}
\begin{CD}
X^{1} @>A^{1}>> X^{2}\ \cdots\ X^{k} @>A^{k}>> X^{k+1}.
\end{CD}
\end{equation}
  \item Operator complex $OC$ is Fredholm if and only if pair complex $PC$ is Fredholm. In this case ${\alpha_{j}(OC) = \gamma_{j}(PC)}$ for each ${j=1,...,k+1}$ and consequently ${ind(OC) = ind(PC)}$.
\end{enumerate}
\end{lemma}
\begin{proof}
In order to prove proposition 1 first notice that all operators $A^{j}$ for\linebreak[4] ${j=1,...,k-1}$ are well defined -- since $PC$ is a pair complex then
\[
M^{j} + N^{j} \subset M^{j+1} \cap N^{j+1},
\]
and therefore $A^{j}$ maps Banach space ${D^{j} = M^{j} \times N^{j} \subset X^{j}}$ into Banach space space
\begin{multline}
\notag
\begin{aligned}
(M^{j} + N^{j})\ \times (M^{j} + N^{j}) &\ \subset\ (M^{j+1} \cap N^{j+1}) \times (M^{j+1} \cap N^{j+1})\\
                                        &\ \subset\ M^{j+1} \times N^{j+1}\\
                                        &\ \subset\ X \times X\\
                                        &\ =\ X^{j+1}.
\end{aligned}
\end{multline}
Also, by construction for ${j < k-1}$
\begin{multline}
\notag
\begin{aligned}
A^{j+1} \circ A^{j} (x, y) &\ =\ A^{j+1}(x+y,\ -(x+y))\\
                           &\ =\ ((x+y) +(-(x+y)),\ -((x+y) +(-(x+y))))\ =\ (\theta,\ \theta),
\end{aligned}
\end{multline}
and for the last operator $A^{k}$
\begin{multline}
\notag
\begin{aligned}
A^{k} \circ A^{k-1} (x, y) &\ =\ A^{k}(x+y,\ -(x+y))\\
                           &\ =\ (x+y) +(-(x+y))\ =\ \theta.
\end{aligned}
\end{multline}
Therefore ${A^{j+1} \circ A^{j} = \theta}$ which proves that the set of Banach spaces\linebreak[4] ${\{X^{1}, ..., X^{k}, X^{k+1}\}}$ and the set of operators ${\{A^{1},...,A^{k}\}}$ constitute an operator complex $OC$ from \eqref{E:700}.

In order to prove proposition 2 notice that by construction of operators $A^{j}$ their ranges are closed ${\overline{R(A^{j})} = R(A^{j})}$ if and only if the sums of respective subspaces are closed ${\overline{M^{j} + N^{j}} = M^{j} + N^{j}}$. Also, from the same construction it is easy to establish the following isomorphisms for each ${j=1,...,k}$:
\[
Ker(A^{j})\ \equiv\ M^{j} \cap N^{j},\ \ \ \ R(A^{j})\ \equiv\ M^{j} + N^{j}.
\]
Hence
\begin{multline}
\notag
\begin{aligned}
\alpha_{1}(OC) &\ =\ \dim Ker(A^{1})\ =\ \dim (M^{1} \cap N^{1})\ =\ \gamma_{1}(PC),\\
\alpha_{k+1}(OC) &\ =\ \dim X / R(A^{k})\ =\ \dim X / (M^{k} + N^{k})\ =\ \gamma_{k+1}(PC),
\end{aligned}
\end{multline}
and for each ${j=2,...,k}$
\begin{multline}
\notag
\begin{aligned}
\alpha_{j}(OC) &\ =\ \dim Ker(A^{j})/R(A^{j-1})\ =\ \dim (M^{j} \cap N^{j}) / (M^{j-1} + N^{j-1})\ =\ \gamma_{j}(PC).
\end{aligned}
\end{multline}
Therefore
\[
ind(OC)\ =\ \sum_{j=1}^{k+1} (-1)^{j-1} \times \alpha_{j}\ =\ \sum_{j=1}^{k+1} (-1)^{j-1} \times \gamma_{j}\ =\ ind(PC).
\]
\end{proof}
It is worth noticing that in a finite-dimensional space $X$ the index of a pair complex can be calculated directly from the dimensions of all pair's components and from the dimension of $X$. Similarly, the index of an operator complex in finite-dimensional spaces can be calculated from the dimensions of its spaces:
\begin{lemma}[Index of a Pair Complex]\label{ipc}
Then the following propositions are true:
\begin{enumerate}
  \item Let $X$ be a finite-dimensional Banach space and ${PC = \{P^{1},...,P^{k}\}}$ be a pair complex from $X$ with ${k \geq 1}$ pairs ${P^{j} = (M^{j},N^{j})}$ for each ${j=1,...,k}$. Then the following formula holds:
\[
ind(PC)\ =\ \sum_{j=1}^{k}(-1)^{j-1} \times (\dim M^{j} + \dim N^{j})\ +\ (-1)^k \times \dim X.
\]
  \item Let $OC$ be an operator complex
\[
\begin{CD}
X^{1} @>A^{1}>> X^{2}\ \cdots\ X^{k} @>A^{k}>> X^{k+1}.
\end{CD}
\]
Suppose that ${\dim X^{j} < \infty}$ and ${dom(A^{j}) = X^{j}}$ for each ${j=1,...,k}$. Then the following formula holds:
\[
ind(OC)\ =\ \sum_{j=1}^{k+1}(-1)^{j-1} \times \dim X^{j}.
\]
\end{enumerate}
\end{lemma}
\begin{proof}
In order to prove proposition 1 first recall definition of $\gamma_{1}$
\[
\gamma_{1}\ =\ \dim (M^{1} \cap N^{1}).
\]
Then note that for any two subspaces $M$ and $N$ from a finite-dimensional space $X$ the following obvious formula is true:
\[
\dim(M+N)\ =\ \dim M + \dim N - \dim M \cap N.
\]
Applying this formula to each component $\gamma_{j}$ for ${j=2,...,k}$ recalculate ${\gamma_{j}}$ like this
\begin{multline}
\notag
\begin{aligned}
\gamma_{j} &\ =\ \dim (M^{j} \cap N^{j})/(M^{j-1} + N^{j-1})\\
           &\ =\ \dim (M^{j} \cap N^{j})\ -\ \dim(M^{j-1} + N^{j-1})\\
           &\ =\ \dim (M^{j} \cap N^{j})\ -\ (\dim M^{j-1} + \dim N^{j-1}\ -\ \dim (M^{j-1} \cap N^{j-1}))\\
           &\ =\ \dim (M^{j} \cap N^{j})\ -\ (\dim M^{j-1} + \dim N^{j-1})\ +\ \dim (M^{j-1} \cap N^{j-1}))\\
           &\ =\ \dim (M^{j-1} \cap N^{j-1})\ +\ \dim (M^{j} \cap N^{j})\ -\ (\dim M^{j-1} + \dim N^{j-1}).
\end{aligned}
\end{multline}
Applying the same formula to $\gamma_{k+1}$ obtain
\begin{multline}
\notag
\begin{aligned}
\gamma_{k+1} &\ =\ \dim X/(M^{k} + N^{k})\\
             &\ =\ \dim X\ -\ \dim(M^{k} + N^{k})\\
             &\ =\ \dim X\ -\ (\dim M^{k}\ +\ \dim N^{k}\ -\ \dim (M^{k} \cap N^{k}))\\
             &\ =\ \dim (M^{k} \cap N^{k})\ -\ (\dim M^{k}\ +\ \dim N^{k})\ +\ \dim X. 
\end{aligned}
\end{multline}
It is clear now that the components ${\dim (M^{j} \cap N^{j})}$ will cancel each other out in the index formula since each next additive member $\gamma_{j}$ of that formula has the sign opposite to the previous member $\gamma_{j-1}$ -- the only members left will be ${\dim M^{j} + \dim N^{j}}$ and ${\dim X}$:
\[
ind(PC)\ =\ \sum_{j=1}^{k+1}(-1)^{j-1} \times \gamma_{j}\ =\ \sum_{j=1}^{k}(-1)^{j-1} \times (\dim M^{j} + \dim N^{j})\ +\ (-1)^k \times \dim X.
\]
\\

In order to prove proposition 2 recall definition of ${\alpha_{1}}$
\[
\alpha_{1}\ =\ \dim Ker(A^{1}).
\]
Then recall for each ${j=2,...,k}$ the simple algebraic fact
\[
\alpha_{j}\ =\ \dim Ker(A^{j+1})/R(A^{j})\ =\ \dim Ker(A^{j+1})\ -\ (\dim X^{j}\ -\ \dim Ker(A^{j})).
\]
Finally
\[
\alpha_{k}\ =\ \dim X^{k+1}\ -\ (\dim X^{k}\ -\ \dim Ker(A^{k})).
\]
It is now clear that in the final sum for ${ind(OC)}$ all members ${\dim Ker(A^{j})}$ cancel each other out, so only dimensions of ${X^{j}}$ will remain:
\[
ind(OC)\ =\ \sum_{j=1}^{k+1}(-1)^{j-1} \times \dim X^{j}.
\]
\end{proof}
\subsection[Stability of Fredholm Pair Complexes]{Stability of Fredholm Pair Complexes}\label{sfpcx}
We are now ready to formulate a Fredholm stability theorem for pair complexes:
\begin{theorem}[Stability of Fredholm Pair Complexes]\label{sfpc}
Let ${PC = \{P^{1},...,P^{k}\}}$ be a Fredholm pair complex in a Banach space $X$ such that ${P^{j} = (M^{j}, N^{j})}$ for each ${j=1,...,k}$. Consider a sequence ${(PC_{n})_{\mathbb{N}^{'}}}$ of pair complexes from $X$, such that ${PC_{n} = \{P^{1}_{n},...,P^{k}_{n}\}}$ for each ${n \in \mathbb{N}^{'}}$ and ${P^{j}_{n} = (M^{j}_{n}, N^{j}_{n})}$ for each ${j=1,...,k}$. Then the following propositions are true:
\begin{enumerate}
  \item Consider the following uniform proximity vectors from ${\mathbb{R}^{3}}$ for each\linebreak[4] ${j=1,...,k}$:
\[
\mu_{\mathbb{N}^{'}}[M^{j}_{n}, M^{j}],\ \ \mu_{\mathbb{N}^{'}}[N^{j}_{n}, N^{j}],\ \ \mu_{\mathbb{N}^{'}}[M^{j}, M^{j}_{n}],\ \  \mu_{\mathbb{N}^{'}}[N^{j}, N^{j}_{n}].
\]
If all of these vectors are small enough, then pair complexes ${PC_{n}}$ are Fredholm for large enough $n$. Moreover, for the same $n$ defect numbers from ${PC_{n}}$ do not exceed respective defect numbers from ${PC}$ while their indices remain the same:
\begin{multline}
\notag
\begin{aligned}
\gamma_{j}(PC_{n}) & \ \leq\ \gamma_{j}(PC),\ j = 1,...,k,\\
ind(PC_{n}) & \ =\ ind(PC). 
\end{aligned}
\end{multline}
  \item Suppose that $X$ is a $BFU-$space. Consider the following uniform adjustment numbers from ${\mathbb{R}}$ for each ${j=1,...,k}$:
\[
\lambda_{\mathbb{N}^{'}}[M^{j}_{n}, M^{j}],\ \ \lambda_{\mathbb{N}^{'}}[N^{j}_{n}, N^{j}],\ \ \lambda_{\mathbb{N}^{'}}[M^{j}, M^{j}_{n}],\ \ \lambda_{\mathbb{N}^{'}}[N^{j}, N^{j}_{n}].
\]
If all of these numbers are small enough, then pair complexes ${PC_{n}}$ are Fredholm for large enough $n$. Moreover, there exists ${K \in \mathbb{N}}$ such that for the same $n$ all defect numbers from ${PC_{n}}$ do not exceed ${K}$ while absolute value of ${ind(PC_{n})}$ does not exceed ${\frac{k+1}{2} \times K}$:
\begin{multline}
\notag
\begin{aligned}
\gamma_{j}(PC_{n}) & \ \leq\ K,\ j = 1,...,k,\\
ind(PC_{n}) & \ \leq\ \frac{k+1}{2} \times K. 
\end{aligned}
\end{multline}
\end{enumerate}
\end{theorem}
By converting an operator complex of closed operators to a pair complex as described in the previous Lemma \ref{ociapc} and applying the above Theorem \ref{sfpc} one can prove the following theorem about stability of operator complexes.
\begin{theorem}[Stability of a Fredholm Operator Complex]\label{sofop}
Let $OC$ be a Fredholm operator complex comprised of a finite set of Banach spaces ${\{X^{1},...,X^{k+1}\}}$ where ${k \geq 1}$, and of a finite set of closed linear operators ${\{A^{1},...,A^{k}\}}$ where ${A^{i} \in \mathcal{C}(X^{i}, X^{i+1})}$ and ${R(A^{i}) \subset Ker(A^{i+1})}$ for each ${i=1,...,k}$:
\begin{equation}
\begin{CD}
X^{1} @>A^{1}>> X^{2}\ \cdots\ X^{k} @>A^{k}>> X^{k+1}.
\end{CD}
\end{equation}
Let ${(OC_{n})_{\mathbb{N}^{'}}}$ be a sequence of operator complexes each comprised of a finite set of Banach spaces ${\{X^{1}_{n},...,X^{k+1}_{n}\}}$ where ${k \geq 1}$, and of a finite set of closed linear operators ${\{A^{1}_{n},...,A^{k}_{n}\}}$ where ${A^{i}_{n} \in \mathcal{C}(X^{i}_{n}, X^{i+1}_{n})}$ and ${R(A^{i}_{n}) \subset Ker(A^{i+1}_{n})}$ for each ${i=1,...,k}$ and for each ${n \in \mathbb{N}^{'}}$:
\begin{equation}
\begin{CD}
X^{1}_{n} @>A^{1}_{n}>> X^{2}_{n}\ \cdots\ X^{k}_{n} @>A^{k}_{n}>> X^{k+1}_{n}.
\end{CD}
\end{equation}
then the following propositions are true:
\begin{enumerate}
  \item Consider the following uniform proximity vectors from ${\mathbb{R}^{3}}$ for each\linebreak[4] ${j=1,...,k}$:
\[
\mu_{\mathbb{N}^{'}}[A^{j}_{n}, A^{j}],\ \ \mu_{\mathbb{N}^{'}}[A^{j}, A^{j}_{n}].
\]
If all these vectors are small enough, then operator complexes ${OC_{n}}$ are Fredholm for large enough $n$. Moreover, for the same $n$ defect numbers from ${OC_{n}}$ do not exceed respective defect numbers from ${OC}$ while their indices remain the same:
\begin{multline}
\notag
\begin{aligned}
\alpha_{j}(OC_{n}) & \ \leq\ \alpha_{j}(OC),\ j = 1,...,k,\\
ind(OC_{n}) & \ =\ ind(OC). 
\end{aligned}
\end{multline}
  \item Suppose that the product space ${\prod_{j=1}^{k+1}X^{j}}$ is a $BFU-$space. Consider the following uniform adjustment numbers from ${\mathbb{R}}$ for each ${j=1,...,k}$:
\[
\lambda_{\mathbb{N}^{'}}[A^{j}_{n}, A^{j}],\ \ \lambda_{\mathbb{N}^{'}}[A^{j}, A^{j}_{n}].
\]
If all of these numbers are small enough, then operator complexes ${OC_{n}}$ are Fredholm for large enough $n$. Moreover, there exists ${K \in \mathbb{N}}$ such that for the same $n$ all defect numbers from ${OC_{n}}$ do not exceed ${K}$ while absolute value of ${ind(OC_{n})}$ does not exceed ${\frac{k+1}{2} \times K}$:
\begin{multline}
\notag
\begin{aligned}
\alpha_{j}(OC_{n}) & \ \leq\ K,\ j = 1,...,k,\\
ind(OC_{n}) & \ \leq\ \frac{k+1}{2} \times K. 
\end{aligned}
\end{multline}
\end{enumerate}
\end{theorem}
\subsection[Tuple Complex]{Tuple Complex}\label{S:ftc}
So far we have extended the basic concept of a pair of closed subspaces in two different directions. On one hand a \emph{tuple of closed subspaces} from subsection \ref{S:tcs} captures the properties of a finite set of closed subspaces related to their sums and intersections and allows for (semi--)Fredholm stability theorems. On the other hand a \emph{pair complex} of closed subspaces from subsection \ref{pc} captures inclusion properties of its member pairs. It is hard not to notice that while a tuple may have any finite number of closed subspaces, a pair complex allows only for two-elements components (although for any finite number of components).

In this section we extend the concept of pair complex to \emph{tuple complex} -- a complex of closed subspaces which elements can be tuples of any finite number of subspaces, not just of pairs of subspaces. Although at first glance the construction of a finite tuple complex may seem like an esoteric  construction, we shall immediately explain that in reality it is a very simple concept as it naturally extends all the previously known concepts -- pairs, tuples, and pair complexes.
\begin{definition}[Tuple Complex]\label{D:tc}
Let $X$ be a Banach space and ${T = \{T^{1},...,T^{k}\}}$ be a finite set of tuples, each tuple $T^{j}$ consisting of ${t_{j} \in \mathbb{N}}$ closed subspaces ${\{M^{j,1},...,M^{j,t_{j}}\}}$ from $X$. Consider the set of $k+1$ Banach spaces\linebreak[4] ${\{X^{1},...,X^{k},X^{k+1}\}}$ defined like this
\begin{multline}
\notag
\begin{aligned}
X^{j}   &\ :=\ \underbrace{X \times \cdots \times X}_{t_{j}},\ \ j = 1,...,k,\\
X^{k+1} &\ :=\ X. 
\end{aligned}
\end{multline}
Also let ${A = \{A^{1},...,A^{k}\}}$ be a set of $k$ continuous linear operators $A^{j}$:
\[
A^{j}\ \in\ \mathcal{BC}(X^{j}, X^{j+1}),\ \ dom(A^{j}) = \prod_{i=1}^{t_{j}} M^{j,i}.
\]
Suppose that $k = 1$, or, if ${k > 1}$, then for each ${j=1,...,k-1}$ the range ${R(A^{j})}$ of operator $A^{j}$ is a subset of the kernel\ ${Ker(A^{j+1})}$ of the next operator ${A^{j+1}}$. In other words, suppose that the spaces $X^{j}$ and operators $A^{j}$ constitute an operator complex $OC$
\begin{equation}
\begin{CD}
X^{1} @>A^{1}>> X^{2}\ \cdots\ X^{k} @>A^{k}>> X^{k+1}.
\end{CD}
\end{equation}
Then the following definitions are in order
\begin{itemize}
  \item The ordered triple ${(X, T, A)}$ is called a tuple complex.
  \item If operator complex $OC$ is a Fredholm operator complex, then tuple complex ${(X, T, A)}$ is called a Fredholm tuple complex; its defect numbers ${\gamma_{j}(X,T,A)}$ and index ${ind(X, T, A)}$ are defined to be equal to the respective defect numbers ${\gamma_{j}(OC)}$ and index of the operator complex ${ind(OC)}$.
  \end{itemize} 
\end{definition}
The following lemma shows that with a suitable choice of operators $A^{j}$ each of the previous objects -- pair of subspaces, tuple of subspaces and pair complex -- can be converted into a tuple complex. Moreover, if any of those objects is Fredholm then its corresponding tuple complex is Fredholm.
\begin{lemma}[Examples of Tuple Complexes]\label{eotc}
Let $X$ be a Banach space. Then the following propositions are true:
\begin{enumerate}
  \item Let $(M^{1}, M^{2})$ be a pair of closed subspaces from $X$. Define subspaces ${M^{1,1} = M^{1} \subset X}$ and ${M^{1,2} = M^{2} \subset X}$ and construct a tuple of subspaces consisting of just two elements ${T^{1} = (M^{1,1}, M^{1,2})}$. Define Banach spaces ${X^{1} = X \times X}$ and ${X^{2} = X}$ and a set ${A = \{A^{1}\}}$ of one linear operator ${A^{1}}$
\[
A^{1} \in \mathcal{BC}(X^{1}, X^{2}),\ \ dom(A^{1}) = M^{1,1} \times M^{1,2},\ \ A^{1}(x,y) = x + y.
\]
Define a set of tuples consisting of just one tuple ${T = \{T^{1}\}}$. Then the triplet ${(X, T, A)}$ is a tuple complex. That tuple complex is Fredholm if and only if pair ${(M{^1}, M^{2})}$ is Fredholm; in this case ${ind(X, T, A) = ind(M^{1}, M^{2})}$.
  \item Let ${T^{1} = \{M^{1},...,M^{k}\}}$ be a tuple of $k$ closed subspaces from $X$. For each ${j= 1,...,k}$ define a subspace ${M^{1,j} = M^{j}} \subset X$. Also define two Banach spaces ${X^{1} = \underbrace{X \times \cdots \times X}_{k}}$ and ${X^{2} = X}$, and a set of tuples consisting of only one tuple ${T = \{T^{1}\}}$. Finally define a set ${A = \{A^{1}\}}$ of one linear operator ${A^{1}}$
\[
A^{1} \in \mathcal{BC}(X^{1}, X^{2}),\ \ dom(A^{1}) = \prod_{i=1}^{k} M^{1,i},\ \ A^{1}(x^{1},...,x^{k}) = \sum_{i=1}^{k}x^{i}.
\]
Then the triplet ${(X, T, A)}$ is a tuple complex. That tuple complex is Fredholm if and only if the original tuple of subspaces ${T^{1} = \{M^{1},...,M^{k}\}}$ is Fredholm; in this case ${ind(X, T, A) = ind(T^{1})}$.
  \item Let ${PC}$ be a pair complex of $k$ pairs of subspaces from $X$
\[
\{(M^{1,1}, M^{1,2}),...,(M^{k,1}, M^{k,2})\}.
\]
Build an operator complex $OC$ from $PC$ as explained in the previous Lemma \ref{bocpc} and define a set ${T = \{T^{1},...,T^{k}\}}$ of $k$ tuples each tuple consisting from the corresponding pair elements ${T^{j} = \{(M^{j,1}, M^{j,2})\}}$ from the pair complex $PC$. Also define a set of operators ${A = \{A^{1},...,A^{k}\}}$ where each ${A^{j}}$ is from the constructed operator complex $OC$. Then the triplet ${(X, T, A)}$ is a tuple complex. That tuple complex is Fredholm if and only if the original pair complex $PC$ is Fredholm; in this case\linebreak[4] ${ind(X, T, A) = ind(PC)}$.
\end{enumerate}
\end{lemma}
\begin{proof}
The proof of being a tuple complex is quite mechanical: in the first two cases the constructed triplets ${(X, T, A)}$ are tuple complexes just because they consist of only one element; in the third case ${(X, T, A)}$ is a tuple complex simply because ${M^{j,1} + M^{j,2} \subset M^{j+1,1} \cap M^{j+1,2}}$ means that ${R(A^{j}) \subset Ker(A^{j+1})}$ for ${j < k}$. 

The Fredholm property is also trivially verified in every case. First, ranges ${R(A^{j})}$ of all ${A^{j}}$ are closed if and only if corresponding sum ${M^{1} + M^{2}}$ for a pair case and sums ${\sum_{i=1}^{k} M^{j,i}}$ for tuple and pair complex cases are closed. 

Then, in the case of a pair of subspaces the dimension of kernel ${Ker(A^{1})}$ and the co-dimension of range ${R(A^{1})}$ can be trivially checked to be the same as the respective dimensions of the common subspace ${M^{1} \cap M^{2}}$ and of the quotient space ${X/(M^{1} + M^{2})}$; therefore, ${ind(X, T, A) = ind(M^{1}, M^{2})}$

In the case of a tuple the dimension of kernel ${Ker(A^{1})}$ and the co-dimension of range ${R(A^{1})}$ can be trivially checked to be the same as the values of ${\alpha(T^{1})}$ from Definition \ref{usft} and of ${\dim X/(\sum_{i=1}^{k} M^{i})}$ respectively; therefore\linebreak[4] ${ind(X, T, A) = ind(T^{1})}$. 

In the case of a pair complex the dimension of kernel ${Ker(A^{1})}$ is equal to the dimension of subspace ${M^{1,1} \cap M^{1,2}}$; the dimensions of ${Ker(A^{j+1})/R(A^{j})}$ can be trivially checked to be the same as the dimension of the quotient spaces\linebreak[4] ${(M^{j+1,1} \cap M^{j+1,2})/(M^{j,1} + M^{j,2})}$ for ${j < k}$; for ${j = k}$ it is also a trivial fact that ${\dim X^{k+1}/R(A^{k}) = \dim X/(M^{k,1}+M^{k,2})}$; therefore ${ind(X, T, A) = ind(PC)}$.
\end{proof}
It is worth noticing from the above lemma and from the previous Lemma \ref{bocpc} that a tuple complex captures the inclusion relation between successive elements of a pair complex due to the same structure of every linear operator ${A^{j}}$ for ${j < k}$ 
\[
A^{j}(x,y)\ \mapsto\ (x+y,\ -(x+y)).
\]
These operators can be defined through a single ${2 \times 2}$ nilpotent matrix:
\[
\mathfrak{A}\ =\  
\left(
\begin{matrix}
+1 & +1\\
-1 & -1
\end{matrix}
\right),\ \ \ \ \mathfrak{A} \times \mathfrak{A}\ =\ \left(
\begin{matrix}
0 & 0\\
0 & 0
\end{matrix}
\right)\ =\ \theta.
\]
Note that since the matrix is nilpotent of degree $2$ it guarantees the main condition of the tuple complex which is ${A^{j+1} \circ A^{j} = \theta}$. 

Obviously, for the tuple complexes which tuples have more than two elements, the inclusion relation can be captured through a variety of matrices of higher degrees: when all the matrices are the same (which may only happen when the ranks of all tuples are the same) these matrices shall be nilpotent of degree $2$; for different matrices (which may happen regardless of the values of the tuple's ranks) the nilpotent condition should be replaced with the condition for two successive matrices ${\mathfrak{A}^{j}}$ and ${\mathfrak{A}^{j+1}}$ having a null product ${\mathfrak{A}^{j+1} \times \mathfrak{A}^{j} = \theta}$. 

For example, in case when the sums of all ${t_{j}}$ subspaces from tuple ${T^{j}}$ belong to the intersection of all ${t_{j+1}}$ subspaces from tuple ${T^{j+1}}$
\[
\sum_{i=1}^{t_{j}} M^{j,i}\ \subset\ \bigcap_{i=1}^{t_{j+1}} M^{j+1,i},
\]
one can use ${t_{j} \times t_{j+1}}$ matrices
\[
\mathfrak{A}^{t_{j} \times t_{j+1}}\ =\  
\left(
\begin{matrix}
+1                  & \dots  & +1\\
-(t_{j+1} - 1)^{-1} & \dots  & -(t_{j+1} - 1)^{-1}\\
\vdots              & \ddots & \vdots\\
-(t_{j+1} - 1)^{-1} & \dots  & -(t_{j+1} - 1)^{-1}
\end{matrix}
\right)
\]
so that
\[
\mathfrak{A}^{t_{j+1} \times t_{j+2}}\ \times\ \mathfrak{A}^{t_{j} \times t_{j+1}}\ =\  
\left(
\begin{matrix}
0      & \dots  & 0\\
\vdots & \ddots & \vdots\\
0      & \dots  & 0
\end{matrix}
\right)\ =\ \theta.
\]
With the help of the above matrices ${\mathfrak{A}^{t_{j} \times t_{j+1}}}$ one can define linear operators ${A^{j} \in \mathcal{BC}(X^{j}, X^{j+1})}$ like this
\begin{multline}
\notag
\begin{aligned}
dom(A^{j}) &\ =\ \prod_{i=1}^{t_{j}} M^{j,i},\\
A^{j}\ :\ (\underbrace{x^{1},x^{2},...,x^{t_{j}}}_{t_{j}}) & \ \mapsto\ (\sum_{i=1}^{t_{j}} x^{i},\ \underbrace{-\frac{\sum_{i=1}^{t_{j}} x^{i}}{t_{j+1} - 1},\ ...\ ,\ -\frac{\sum_{i=1}^{t_{j}} x^{i}}{t_{j+1} - 1}}_{t_{j+1} - 1})\ \in\ \prod_{i=1}^{t_{j+1}} M^{j+1,i}.
\end{aligned}
\end{multline}
Obviously
\[
R(A^{j})\ \subset\ \bigcap_{i=1}^{t_{j+1}} M^{j+1,i}\ =\ dom(A^{j+1})
\]
according to our assumption and ${A^{j+1} \circ A^{j} = \theta}$ since ${\mathfrak{A}^{t_{j+1} \times t_{j+2}} \times \mathfrak{A}^{t_{j} \times t_{j+1}} = \theta}$, or in other words ${R(A^{j}) \subset Ker(A^{j+1})}$.\\

The previous examples show that the concept of a pair of closed subspaces and the concept of a pair complex are just the simplest cases of a tuple complex since their corresponding inclusion matrices are very simple. Obviously, more elaborate examples of inclusion relations can be presented.

However, in general it appears that the concept of a tuple complex captures much more than the inclusion relation between subspaces from successive tuples since the linear operators ${A^{j}}$ do not have to be induced by any ${t_{j} \times t_{j+1}}$ matrix; in fact they can be operators from ${\mathcal{BC}(T^{j}, T^{j+1})}$ of any complexity. Therefore, one can think of a tuple complex having the same generality as an operator complex with an additional twist of having operator's domains being products of closed subspaces from a given Banach space. As such, it is worth noticing that when $X$ is finite-dimensional, then the index of a tuple complex ${(X,T,A)}$ can be calculated directly from Lemma \ref{ipc}:
\begin{lemma}[Index of Finite-dimensional Tuple Complex]\label{iofdtc}
Let $X$ be a finite-dimensional Banach space and ${(X,T,A)}$ be a tuple complex such that\linebreak[4] ${T = \{T^{1},...,T^{k}\}}$ and ${T^{j} = \{M^{j1},...,M^{jt_{j}}\}}$ for each ${j=1,...,t_{k}}$. Then the following formula holds
\[
ind(X,T,A) = \sum_{j=1}^{k} (-1)^{j-1} \times (\sum_{i=1}^{t_{j}} \dim M^{ji})\ +\ (-1)^{k} \times \dim X.
\]
\end{lemma}
\subsection[Inclusion Complex and Its Stability]{Inclusion Complex and Its Stability}\label{icats}
Since tuple complexes can be considered a special case of operator complexes, the stability of Fredholm tuple complexes can be described by the already presented Theorem \ref{sofop} about the stability of Fredholm operator complexes. That theorem guarantees stability as long as the complex operators   ${(A^{j}_{n})_{\mathbb{N}^{'}}}$ closely $\lambda-$adjust or $\mu-$approximate respective operators $A^{j}$ from the target complex. 

It is worth noticing that since operators from the tuple complex are continuous, their adjustment and approximation necessarily yields adjustment and approximation of their domains. Therefore, since these domains are finite direct products of closed subspaces from the enclosing Banach space, it follows that individual subspaces from the operator complexes have small $\lambda-$adjustment numbers and small $\mu-$approximation vectors as long as those numbers and vectors are small for the complex operators. However, the opposite statement may not be always true -- since operators $A^{j}_{n}$ and $A^{j}$ are not restricted in their nature, their $\lambda-$adjustment numbers and $\mu-$approximation vectors are not always small when the respective numbers and vectors of their domains are small. 

This situation may seen to be unfortunate as in all the previous examples of Fredholm pairs, Fredholm tuples, and Fredholm pair complexes stability theorems could be expressed purely in terms of closeness of the subspaces involved and without even mentioning of any complex operators. In order to remedy this situation we introduce a special class of tuple complexes called \emph{inclusion complexes} -- they  capture the inclusion relation in a tight manner by allowing only operators induced by fixed finite--dimensional matrices. Then it follows that for inclusion complexes the Fredholm stability theorem can be expressed almost purely in terms of the subspaces involved.
\begin{definition}[Inclusion Matrix and Inclusion Operator]\label{D:imao}
Let $X$ be a Banach space over a field $\mathbb{K}$ and ${T^{1} = \{M^{11},...,M^{1k}\}}$, ${T^{2} = \{M^{21},...,M^{2l}\}}$ are two tuples of closed subspaces from $X$. Consider a ${k \times l}$ matrix ${\mathfrak{A}^{k \times l}}$ with elements from $\mathbb{K}$
\[
\mathfrak{A}^{k \times l}\ =\  
\left(
\begin{matrix}
\mathfrak{a}^{11} & \dots  & \mathfrak{a}^{1k}\\
\vdots            & \ddots & \vdots\\
\mathfrak{a}^{l1} & \dots  & \mathfrak{a}^{lk}\\
\end{matrix}
\right).
\]
Define two Banach spaces
\[
X^{k}\ :=\ \underbrace{X \times \cdots \times X}_{k},\ \ \ X^{l}\ :=\ \underbrace{X \times \cdots \times X}_{l}, 
\]
and a continuous linear operator ${A \in \mathcal{BC}(X^{k}, X^{l})}$:
\[
dom(A)\ =\ \prod_{i=1}^{k}M^{1i},\ \ \ A(x^{1},\ ...,\ x^{k})\ \mapsto\ (\sum_{i=1}^{k}\mathfrak{a}^{1i} \times x^{i},\ ...,\ \sum_{i=1}^{k}\mathfrak{a}^{li} \times x^{i}).
\]
Suppose that the image of operator $A$ is contained in the product of subspaces from tuple $T^{2}$:
\[
R(A)\ \subset\ \prod_{i=1}^{k}M^{2i}.
\]
Then matrix ${\mathfrak{A}^{k \times l}}$ is called the inclusion matrix and operator $A$ is called the inclusion operator in respect to tuples $T^{1}$ and $T^{2}$.
\end{definition}
It is obvious now that in the case of a tuple of subspaces the natural mapping ${\Phi_{T}}$ from Lemma \ref{tdnwd} is an inclusion operator related to the sum of the subspaces from the tuple with the inclusion matrix of dimension ${k \times 1}$:
\[
\mathfrak{A}^{k \times 1}\ =\ (\underbrace{1\ \dots\ 1}_{k}).
\]
Also recall from the previous Subsection \ref{S:ftc} that a pair complex can be treated as a tuple complex with tuples of rank $2$ and all but last tuple operators being inclusion operators defined by a single ${2 \times 2}$ nilpotent matrix
\[
\mathfrak{A}^{2 \times 2}\ =\  
\left(
\begin{matrix}
+1 & +1\\
-1 & -1
\end{matrix}
\right),\ \ \ \ \mathfrak{A}^{2 \times 2} \times \mathfrak{A}^{2 \times 2}\ =\ \left(
\begin{matrix}
0 & 0\\
0 & 0
\end{matrix}
\right)\ =\ \theta.
\]
The last operator from a pair complex is defined by a ${2 \times 1}$ inclusion matrix
\[
\mathfrak{A}^{2 \times 1}\ =\ \left(
\begin{matrix}
1 & 1
\end{matrix}
\right).
\]
It is easy to prove the following lemma which establishes dependency between $\lambda-$adjustment numbers and $\mu-$proximity vectors of the tuples' components and inclusion operators.
\begin{lemma}[$\lambda-$Adjustment and $\mu-$Proximity of Inclusion Operators]\label{laampoio}
Let $X$ be a Banach space over a field $\mathbb{K}$ and ${T^{1} = \{M^{11},...,M^{1k}\}}$, ${T^{2} = \{M^{21},...,M^{2l}\}}$ are two tuples of closed subspaces from $X$ with the inclusion ${k \times l}$ matrix ${\mathfrak{A}^{k \times l}}$ that defines inclusion operator $A$. 

Let ${(T^{1}_{n})_{\mathbb{N}^{'}}}$ and ${(T^{2}_{n})_{\mathbb{N}^{'}}}$ be two sequences of tuples so that ${T^{1}_{n} = \{M^{11}_{n},...,M^{1k}_{n}\}}$ and ${T^{2}_{n} = \{M^{21}_{n},...,M^{2l}_{n}\}}$ for each ${n \in \mathbb{N}^{'}}$. 

Also suppose that inclusion ${k \times l}$ matrices ${\mathfrak{A}^{k \times l}_{n}}$ define inclusion operators $A_{n}$ for each ${n \in \mathbb{N}^{'}}$. Identify each inclusion matrix ${\mathfrak{A}^{k \times l}_{n}}$ and ${\mathfrak{A}^{k \times l}}$ with a linear operator from ${\mathcal{B}(\mathbb{K}^{k}, \mathbb{K}^{l})}$ and suppose that norms of their differences approach $0$:
\[
\left\| \mathfrak{A}^{k \times l}_{n}\ -\ \mathfrak{A} \right\|\ \rightarrow\ 0. 
\]
Then the following propositions are true:
\begin{enumerate}
  \item $\lambda-$Adjustment numbers ${\lambda_{\mathbb{N}^{'}}[A_{n},A]}$ and ${\lambda_{\mathbb{N}^{'}}[A,A_{n}]}$ approach number $0$ when $\lambda-$adjustment numbers ${\lambda_{\mathbb{N}^{'}}[M^{1j}_{n},M^{1j}]}$ and ${\lambda_{\mathbb{N}^{'}}[M^{1j},M^{1j}_{n}]}$ approach number $0$.
  \item $\mu-$Proximity vectors ${\mu_{\mathbb{N}^{'}}[A_{n},A]}$ and ${\mu_{\mathbb{N}^{'}}[A,A_{n}]}$ approach vector ${\theta \in \mathbb{R}^{3}}$ when $\mu-$proximity vectors ${\mu_{\mathbb{N}^{'}}[M^{1j}_{n},M^{1j}]}$ and ${\mu_{\mathbb{N}^{'}}[M^{1j},M^{1j}_{n}]}$ approach vector ${\theta \in \mathbb{R}^{3}}$.
\end{enumerate}
\end{lemma}
Equipped with the concept of the inclusion matrix and inclusion operator, we may now introduce \emph{inclusion complex}:
\begin{definition}[Inclusion Complex]\label{D:ic}
Let ${(X,T,A)}$ be a tuple complex and each operator ${A^{j}}$ from the set of operators ${A = \{A^{1},...,A^{k}\}}$ is an inclusion operator defined by some inclusion matrix ${\mathfrak{A}^{t_{j} \times t_{j+1}}}$ for every ${j=1,...,k}$ (here ${t_{j}}$ is the count of subspaces in the tuple ${T^{j} = \{M^{1},...,M^{t_{j}}\}}$ for each ${j=1,...,k}$ while ${t_{j+1} = 1}$ if one recalls that the last space ${X^{k+1}}$ from the tuple complex definition is defined to be just $X$). 

Then tuple complex ${(X,T,A)}$ is called an inclusion complex. If the set of its inclusion matrices is denoted as
\[
\mathfrak{A}\ =\ \{\mathfrak{A}^{t_{1} \times t_{2}},\ ...\ ,\mathfrak{A}^{t_{k} \times t_{k+1}}\},
\]
then that inclusion complex is denoted as ${(X,T,\mathfrak{A})}$. An inclusion complex is written schematically using its inclusion matrices instead of its operators:
\[
\begin{CD}
X^{1} @>\mathfrak{A}^{t_{1} \times t_{2}}>> X^{2}\ \cdots\ X^{k} @>\mathfrak{A}^{t_{k} \times t_{k+1}}>> X^{k+1}.
\end{CD}
\]
Inclusion complex ${(X,T,\mathfrak{A})}$ is called a Fredholm inclusion complex if tuple complex ${(X,T,A)}$ is Fredholm -- defect and index numbers from ${(X,T,\mathfrak{A})}$ are defined to be respective defect and index numbers from ${(X,T,A)}$. 
\end{definition}
It follows directly from the above definition that for any two consecutive inclusion matrices from an inclusion complex their product is always a null matrix
\[
\mathfrak{A}^{t_{j+1} \times t_{j+2}}\ \times\ \mathfrak{A}^{t_{j} \times t_{j+1}}\ =\  
\left(
\begin{matrix}
0      & \dots  & 0\\
\vdots & \ddots & \vdots\\
0      & \dots  & 0
\end{matrix}
\right)\ =\ \theta.
\]
Also note that we have already established in the previous subsection \ref{icats} that a pair of subspaces, a tuple of subspaces and a pair complex all can be considered to be inclusion complexes.\\

We are now ready to present the final theorem of this paper
\begin{theorem}[Stability of a Fredholm Inclusion Complex]\label{soafic}
Let $X$ be a Banach space and ${IC=(X,T,\mathfrak{A})}$ be a Fredholm inclusion complex with the set of tuples ${T=\{T^{1},...,T^{k}\}}$ of closed subspaces ${T^{j}=\{M^{j1},...,M^{jt_{k}}\}}$ for each ${j=1,...,k}$ and the set of inclusion matrices ${\mathfrak{A}\ =\ \{\mathfrak{A}^{t_{1} \times t_{2}},\ ...\ ,\mathfrak{A}^{t_{k} \times t_{k+1}}\}}$
\[
\begin{CD}
X^{1} @>\mathfrak{A}^{t_{1} \times t_{2}}>> X^{2}\ \cdots\ X^{k} @>\mathfrak{A}^{t_{k} \times t_{k+1}}>> X^{k+1}.
\end{CD}
\]
Let ${(IC_{n})_{\mathbb{N}^{'}}}$ be a sequence of inclusion complexes from $X$, such that\linebreak[4] ${IC_{n}=(X,T_{n},\mathfrak{A}_{n})}$ with the sets ${T_{n}=\{T^{1}_{n},...,T^{k}_{n}\}}$ of tuples of closed subspaces\linebreak[4] ${T^{j}_{n}=\{M^{j1}_{n},...,M^{jt_{k}}_{n}\}}$ for each ${j=1,...,k}$  and the sets of inclusion matrices\linebreak[4] ${\mathfrak{A}_{n}\ =\ \{\mathfrak{A}^{t_{1} \times t_{2}}_{n},\ ...\ ,\mathfrak{A}^{t_{k} \times t_{k+1}}_{n}\}}$
\[
\begin{CD}
X^{1}_{n} @>\mathfrak{A}^{t_{1} \times t_{2}}_{n}>> X^{2}_{n}\ \cdots\ X^{k}_{n} @>\mathfrak{A}^{t_{k} \times t_{k+1}}_{n}>> X^{k+1}_{n}.
\end{CD}
\]
for each ${n \in \mathbb{N}^{'}}$. Identify each inclusion matrix ${\mathfrak{A}^{t_{j} \times t_{j+1}}}$ and ${\mathfrak{A}^{t_{j} \times t_{j+1}}_{n}}$ with a linear operator from ${\mathcal{B}(\mathbb{K}^{t_{j}}, \mathbb{K}^{t_{j+1}})}$ and suppose that
\[
\left\| \mathfrak{A}^{t_{j} \times t_{j+1}}_{n}\ -\ \mathfrak{A}^{t_{j} \times t_{j+1}}\right\|\ \rightarrow\ 0.
\]
Then the following propositions are true:
\begin{enumerate}
  \item Consider the following proximity vectors from ${\mathbb{R}^{3}}$:
\[
\left\{
\begin{matrix}
\mu_{\mathbb{N}^{'}}[M^{11}, M^{11}_{n}] & \dots  & \mu_{\mathbb{N}^{'}}[M^{1t_{1}}, M^{1t_{1}}_{n}]\\
\vdots                                   & \ddots & \vdots\\
\mu_{\mathbb{N}^{'}}[M^{k1}, M^{k1}_{n}] & \dots  & \mu_{\mathbb{N}^{'}}[M^{kt_{k}}, M^{kt_{k}}_{n}]\\
\\
\hline
\\
\mu_{\mathbb{N}^{'}}[M^{11}_{n}, M^{11}] & \dots  & \mu_{\mathbb{N}^{'}}[M^{1t_{1}}_{n}, M^{1t_{1}}]\\
\vdots                                   & \ddots & \vdots\\
\mu_{\mathbb{N}^{'}}[M^{k1}_{n}, M^{k1}] & \dots  & \mu_{\mathbb{N}^{'}}[M^{kt_{k}}_{n}, M^{kt_{k}}]
\end{matrix}
\right\}
\]
If all these vectors are small enough, then inclusion complexes ${IC_{n}}$ are Fredholm for large enough $n$. Moreover, for the same $n$ defect numbers from ${IC_{n}}$ do not exceed respective defect numbers from ${IC}$ while their indices remain the same:
\begin{multline}
\notag
\begin{aligned}
\gamma_{j}(IC_{n}) & \ \leq\ \gamma_{j}(IC),\ j = 1,...,k,\\
ind(IC_{n}) & \ =\ ind(IC). 
\end{aligned}
\end{multline}
  \item Let $X$ be a $BFU-$space and consider the following adjustment numbers from ${\mathbb{R}}$:
\[
\left\{
\begin{matrix}
\lambda_{\mathbb{N}^{'}}[M^{11}, M^{11}_{n}] & \dots  & \lambda_{\mathbb{N}^{'}}[M^{1t_{1}}, M^{1t_{1}}_{n}]\\
\vdots                                       & \ddots & \vdots\\
\lambda_{\mathbb{N}^{'}}[M^{k1}, M^{k1}_{n}] & \dots  & \lambda_{\mathbb{N}^{'}}[M^{kt_{k}}, M^{kt_{k}}_{n}]\\
\\
\hline
\\
\lambda_{\mathbb{N}^{'}}[M^{11}_{n}, M^{11}] & \dots  & \lambda_{\mathbb{N}^{'}}[M^{1t_{1}}_{n}, M^{1t_{1}}]\\
\vdots                                       & \ddots & \vdots\\
\lambda_{\mathbb{N}^{'}}[M^{k1}_{n}, M^{k1}] & \dots  & \lambda_{\mathbb{N}^{'}}[M^{kt_{k}}_{n}, M^{kt_{k}}]
\end{matrix}
\right\}
\]
If all of these numbers are small enough, then inclusion complexes ${IC_{n}}$ are Fredholm for large enough $n$. Moreover, there exists ${K \in \mathbb{N}}$ such that for the same $n$ all defect numbers from ${IC_{n}}$ do not exceed ${K}$ while absolute value of ${ind(IC_{n})}$ does not exceed ${\frac{k+1}{2} \times K}$:
\begin{multline}
\notag
\begin{aligned}
\gamma_{j}(IC_{n}) & \ \leq\ K,\ j = 1,...,k,\\
ind(IC_{n}) & \ \leq\ \frac{k+1}{2} \times K. 
\end{aligned}
\end{multline}
\end{enumerate}
\end{theorem}
\begin{proof}
According to Lemma \ref{laampoio} proximity vectors and adjustment numbers of the inclusion operators are small when respective proximity vectors and adjustment numbers of tuples' components are small. Then both of the  above propositions follow from the respective propositions about the stability of operator complexes from Theorem \ref{sofop}.
\end{proof}

\newpage
\section[Open Problems]{Open Problems}
In the conclusion of this work we list remaining open questions related to $\lambda-$adjustment.\\

The first set of  questions relate to the concept of $BFU-$spaces introduced in Section \ref{sibfus}.
\begin{enumerate}
  \item \emph{Does a product of any two ${BFU-}$spaces remains to be a ${BFU-}$space?} Note that in the system of axioms ${ZF + Martin Axiom}$ a product of two Fr\'{e}chet--Urysohn spaces may not remain Fr\'{e}chet--Urysohn (see \cite{malihin_shapirovski,archangelskii2}). 
  \item \emph{Does a product of a ${BFU-}$space and a reflexive space remains to be a\linebreak[4] ${BFU-}$space?}
  \item \emph{Is the property of being a $BFU-$space a $3SP-$property?} (For survey of $3SP-$properties see \cite{castillo_gonzales}).
  \item \emph{Is it true that any uncomplemented closed subspace of a ${BFU-}$space is a ${BFU-}$space?}
  \item The $BFU$ Small Uniform Adjustment Theorem \ref{T:bfusuat} appears to be fundamental for the stability theory of $BFU$ spaces. It differs from the non-$BFU$ Small Uniform Adjustment Theorem \ref{T:suat} by the constant employed -- the $BFU$ version uses $1$ while non-$BFU$ version uses $1/2$. The natural question arises \emph{if the value $1$ of the constant is characteristic of $BFU$ spaces?} Also, proposition 1 of the $BFU$ version is much stronger than proposition 2 from non-$BFU$ version -- \emph{does this proposition 1 characterize $BFU$ spaces?} In summary, \emph{suppose that in a given Banach space $X$ any of the propositions 1,2,3 or 4 from Theorem \ref{T:bfusuat} is true for any two sequences of closed subspaces ${(M_{n})_{\mathbb{N}^{'}}}$ and ${(P_{n})_{\mathbb{N}^{'}}}$ from $X$. Is it true that $X$ is a $BFU-$space?}
  \item If the answer to the previous question is no, then \emph{in what other classes of Banach spaces, aside from $BFU-$spaces, the Small Uniform Adjustment theorem holds for constants between ${1/2}$ and $1$?}
\end{enumerate}

Another question is related to proposition 5 from the Small Uniform Adjustment Theorem \ref{T:suat}. Note that its counterpart is absent from the $BFU$ version in Theorem \ref{T:bfusuat}. Thus the following open question remains: \emph{does proposition 5 from Theorem \ref{T:suat} remain true in $BFU$-spaces if its constant $1/2$ is replaced with constant $1$?}\\

Finally, one can notice that index of Fredholm pairs, tuples and complexes may not be preserved exactly under $\lambda-$adjustment (even in finite-dimensional Banach spaces). Thus the question remains: \emph{under which additional condition the index is preserved exactly under $\lambda-$adjustment?} Obviously this condition shall be such that it is automatically satisfied for continuous operators as their index is preserved under $\lambda-$adjustment.

\newpage
%
%
%
%
%
%
%
%
%
%
%

\addtocontents{toc}{\protect\vspace*{\baselineskip}}




\appendix


\end{document}